\documentclass[11pt]{report}
\usepackage{isuthesis}
\usepackage{amsmath}
\usepackage{amssymb}
\usepackage{amsthm}
\usepackage{float}
\usepackage[pdftex]{graphicx}
\usepackage{traditional}
\chaptertitle
\alternate
\usepackage{rotating}

\newcommand{\textdef}{\emph}

\newcommand{\Z}{{\mathbb Z}}

\newtheorem{thm}{Theorem}[chapter]
\newtheorem{defn}[thm]{Definition}
\newtheorem{fact}[thm]{Fact}
\newtheorem{lemma}[thm]{Lemma}
\newtheorem{cor}[thm]{Corollary}
\newtheorem{claim}[thm]{Claim}

\newenvironment{proofcite}[1]{\noindent{\bf Proof of #1.\,}}{\hfill$\Box$}

\begin{document}
\DeclareGraphicsExtensions{.jpg,.pdf,.mps,.png}
\title{On Vertex Identifying Codes For Infinite Lattices}
\author{Brendon Michael Stanton}
\mprof{Ryan Martin}
\notice
\degree{DOCTOR OF PHILOSOPHY}
\level{doctoral}
\format{dissertation}
\committee{4}
\members{Maria Axenovich \\ Clifford Bergman \\ Leslie Hogben \\ Chong Wang}
\maketitle

\pdfbookmark[1]{TABLE OF CONTENTS}{table}
\tableofcontents
\addtocontents{toc}{\def\protect\@chapapp{}}
\cleardoublepage
\phantomsection
\addcontentsline{toc}{chapter}{LIST OF FIGURES}
\listoffigures
\addtocontents{toc}{\def\protect\@chapapp{CHAPTER\ }}
\newpage
\pagenumbering{arabic}
\chapter{GENERAL INTRODUCTION}\label{chapter:intro}
Vertex identifying codes were first introduced by Karpovsky, Chakrabarty and Levitin~\cite{Karpovsky1998} in 1998 as a way to help with fault diagnosis in multiprocessor computer
systems.  Since then, study of these codes and their variants have exploded.  Antoine Lobstein maintains an internet bibliography~\cite{Lobstein} which at the time of this writing, contains nearly 200 articles on the subject.

Since the size of a code depends largely on topology of the particular graph, it is common to restrict our study of codes to various graphs and classes of graphs.  The main focus of this thesis is on the primary variant of these codes--called $r$-identifying codes--and their densities on various locally finite infinite graphs, all of which have a representation on a lattice.

\section{Thesis Organization}

The thesis is organized as follows:

In the Chapter~\ref{chapter:intro} the main ideas are described and a review of the literature on the subject is given.

Chapter~\ref{chapter:squaregridpaper} presents the paper ``Lower Bounds for Identifying Codes in Some Infinite Grids''~\cite{Martin2010} published in \emph{Electronic Journal of Combinatorics}.  This paper, co-written with Ryan Martin, finds lower bounds for 2-identifying codes for the square and hexagonal grids by way of a counting argument.  In addition, the paper makes use of the discharging method to further increase the lower bound for the square grid.  The technique presented here can be extended to provide lower bounds for other grids and other values of $r$.  However, it does not give bounds that are as good as the current best known lower bounds in most cases.  One exception is the case $r=3$ for the hexagonal grid, which is given in Chapter~\ref{chapter:squareAddendum}.

Chapter~\ref{chapter:hexgridpaper} presents the paper ``Improved Bounds for $r$-Identifying Codes of the Hex Grid''~\cite{Stanton2011} published in \emph{SIAM Journal on Discrete Mathematics}.  By contrast to~\cite{Martin2010} which provides lower bounds, this paper provides general constructions of codes for the hexagonal grid which decrease the upper bounds of the minimum densities of $r$-identifying codes for large values of $r$.

Chapter~\ref{chapter:latticepaper} presents the paper ``Vertex Identifying Codes for the $n$-dimensional Lattice''~\cite{StantonSubmitted} submitted for publication to \emph{Discrete Mathematics}.  Here we look at the $n$-dimensional analogue of the square grid and provide a general overview and discussion of codes on the $n$-dimensional lattice, providing both upper and lower bounds for $r$-identifying codes in both the most general case ($r$-identifying codes for the $n$-dimensional lattice) and in more specific cases such as the case when $r=1$ and even more specifically when $n=4$.

Chapter~\ref{chapter:squareAddendum} ties up some loose ends, providing a proof that was promised in ``Lower Bounds for Identifying Codes in Some Infinite Grids''\cite{Martin2010} as well as another result which has not yet been submitted for publication.

Chapter~\ref{chapter:regulargraphcodes} is part of a work in progress with Ryan Martin.  It addresses the general issue of codes on (finite) regular graphs--improving upon the general lower bound given in Theorem~\ref{thm:regularGraphCodeLowerBound} and providing constructions of graphs that attain these bounds.

\section{Definitions}

We must first begin with some basic definitions in order to introduce the notion of a vertex identifying code.  For basic definitions about graph theory, we refer the reader to \cite{West1996}.

\begin{defn}Given a graph $G$, the \emph{distance} between two vertices $u,v$ is written $d(u,v)$ which is the length of the shortest path between $u$ and $v$.
\end{defn}

\begin{defn}Given a graph $G$, the \emph{ball of radius $r$, centered at $v$}, denoted by $B_r(v)$ is defined as: $$B_r(v)=\{u:d(u,v)\le r\}.$$
\end{defn}

\begin{defn}
  Given a graph $G$, a \emph{code} $C$ is a nonempty subset of $V(G)$.  The elements of $C$ are called \emph{codewords}.
\end{defn}

Although the above definition may seem a bit unnecessary, it is useful when writing proofs involving identifying codes.  If trying to prove that some set $C$ is an \emph{$r$-identifying code } (or some other type of identifying code), it is often useful to refer to $C$ as a code, which we can simply refer to it as a code without having to worry about whether or not it has the $r$-identifying property.

\begin{defn}
  Given a graph $G$ and a code $C$, the \emph{$r$-identifying set} (or simply \emph{identifying set}) of a vertex $v$, is defined as $$I_r(v) = I_r(v, C) = B_r(v)\cap C.$$
\end{defn}

This brings us to the main definition.

\begin{defn}\label{defn:identifyingcode}
  Given a graph $G$, a code $C$ is called \emph{$r$-identifying} if for each distinct $u,v\in V(G)$ we have
  \begin{enumerate}
    \item $I_r(v)\neq \emptyset$ and
    \item $I_r(u)\neq I_r(v).$
  \end{enumerate}
  When $r=1$, we simply refer to $C$ as an \emph{identifying code}.
\end{defn}

The idea behind an $r$-identifying code is that given an identifying set $I_r(v)$ and a code $C$, we should be able to determine what $v$ is just from its identifying set.  The motivation for this definition is that if we are given a multiprocessor system, we wish to be able to determine when an error occurred by communicating with only a subset of the processors.  If there is no error, then no processors report back that there was an error.  Hence, the first condition ensures that at least 1 processor will report back if there is an error.  The second condition ensures that we are able to determine the source of the error.

\section{Literature Review}

\subsection{General Bounds and Constructions}

Karpovsky, Chakrabarty and Levitin \cite{Karpovsky1998} were the first to introduce the notion of an identifying code.  They provided many basic constructions and bounds for the size of codes--many of which are still best known for some graphs.  First though, we note that not all graphs have codes.  For instance, $K_n$--the complete graph on $n$ vertices--does not allow an $r$-identifying code.  We present some equivalent conditions for the existence of a code:

\begin{thm}\label{thm:codeexistence}
  For any graph $G$, the following are equivalent:
  \begin{enumerate}
    \item $G$ has an $r$-identifying code;
    \item $V(G)$ is an $r$-identifying code; and
    \item $B_r(u)\neq B_r(v)$ for all distinct $u,v\in V(G)$.
  \end{enumerate}
\end{thm}
\begin{proof}
  (2) $\Rightarrow$ (1) is trivial.  It is also relatively easy to see that (2) $\Leftrightarrow$ (3).  Suppose that $V(G)$ is a code for $G$.  Then for all $u\neq v$ we have $B_r(u)\cap V(G) \neq B_r(v)\cap V(G)$ by definition of a code and so $B_r(u)\neq B_r(v)$.  Likewise, if $V(G)$ is not a code, it must be the case that for some $u,v\in V(G)$ with $u\neq v$ we have $B_r(u)\cap V(G)=B_r(v)\cap V(G)$ and so $B_r(u)=B_r(v)$.

  Finally, to show (1) $\Rightarrow$ (2) we need a lemma.
  \begin{lemma}\label{lemma:codeExistenceSubsetLemma}
     Let $C$ be an $r$-identifying code for a graph $G$.  Then for any $S\subset V(G)$ $C\cup S$ is an $r$-identifying code for $G$.
  \end{lemma}
  The proof of the lemma is just an exercise in basic set theory.  Since $C$ is a code, we have $B_r(v)\cap C\neq \emptyset$ for each $v\in V(G)$.  Then by De Morgan's law, we have $$B_r(v)\cap (C\cup S) = (B_r(v)\cap C)\cup (B_r(v)\cap S)\subset B_r(v)\cap C\neq \emptyset.$$
  This shows that $C\cup S$ has condition (1) from Definition~\ref{defn:identifyingcode}.  To show it also has the second condition, for any distinct $u,v\in V(G)$, there must be some codeword in one identifying set that isn't in the other.  Without loss of generality, assume that $c\in B_r(v)\cap C$ but $c\not\in B_r(u)\cap C$.  Again by De Morgan we have $c\in B_r(v)\cap (C\cup S)$.  However, $c\in B_r(v)$ and $c\in C$.  Since $c\not\in B_r(v)\cap C$, then $c\not\in B_r(v)$.  Thus, $c\not\in B_r(v)\cap (C\cup S)$, completing the proof of our lemma.

  Using this lemma, it now follows that if $G$ has a code $C$, then take $S=V(G)$ and we have $C\cup V(G)=V(G)$ is a code for $G$, completing the proof.
\end{proof}

Next we demonstrate a general lower bound for the size of a code for any graph that has a code.

\begin{thm}[\cite{Karpovsky1998}]\label{thm:logbound}
If $C$ is an identifying code on an graph $G$, then $$|C|\ge\lceil\log_2 (|V(G)|+1)\rceil.$$
\end{thm}

\begin{proof}
  We must have more identifying sets than vertices in the graph, since every identifying set is nonempty, this gives $2^{|C|}-1\ge |V(G)|$.  Rearranging the equation gives the result and we may take the ceiling since $|C|$ is an integer.
\end{proof}

If is indeed possible to find graphs with identifying codes that attain this bound for graphs of any order.  Let $n$ be given and let $k=\lceil \log_2 (n+1)\rceil$.  Then define $C=\{1,2,\ldots, k\}$.  Next, let $S$ be a collection of $n-k$ distinct subsets of $C$ such that each subset has size at least 2.  Let $G$ be a graph with vertex set $C\cup S$ and define the edge set to be the set of all edges of the form $c\in C$ and $X\in S$ where $c\in X$.  If $C$ is our code, then for all $c\in C$ we have $I_1(c)=\{c\}$ since there are no edges between vertices in $C$.  For all $X\in S$ we see that $I_1(X)=X$ by the way we defined the edge set.  Hence, the identifying sets of $G$ all distinct and so $C$ is 1-identifying.  Note that we could also add arbitrary edges within $S$ and $C$ would still be a code since it would not affect the identifying sets of any vertex.

Our next theorem states a general upper bound for the number of vertices needed for a code in a connected graph.

\begin{thm}[\cite{Charon2007}]
  If $G$ is a connected graph of order $n$ that admits an identifying code, then $G$ admits an $r$-identifying code of size $n-1$.
\end{thm}

\begin{proof}
  Let $v\in V(G)$.  Consider $C=V(G)-v$.  If it is an $r$-identifying code, then we are finished.  Thus, suppose it is not.  Then we must have $I_r(u,C)=I_r(w,C)$ for some $u\neq w$.  Since $V(G)$ is a code, then without loss of generality, we must also have $I_r(u,V(G)) = I_r(w,V(G))\cup \{v\}$ and $v\not\in I_r(w,V(G))$.  Hence $B_r(w)=B_r(u)-v$.

   Now it is easy to check that $C'=V(G)-w$ is an $r$-identifying code for $G$.  Again, suppose it is not.  Then we must have $I_r(x,C')=I_r(y,C')$ for some $x\neq y$.  Since $V(G)$ is a code, then without loss of generality, we must also have $I_r(x,V(G)) = I_r(y,V(G))\cup \{w\}$ and $v\not\in I_r(y,V(G))$.  However, we must have $u\in I_r(x,V(G))$ since $B_r(w)\subset B_r(u)$ and so we must also have $u\in I_r(y, V(G))$, but then $d(w,y)>r$ and $d(u,y)\le r$, contradicting the fact that $B_r(w)\subset B_r(u)$.  Hence, $V(G)-w$ is a code of size $n-1$, completing the proof.
\end{proof}

It is possible to find graphs requiring at least $n-1$ vertices for a code.  For instance, consider the ``star'' on $n$-vertices.  That is, the graph with vertex set $\{v,u_1,u_2,\ldots, u_{n-1}\}$ and edge set $\{vu_i\}$.  This graph admits a 1-identifying code, but it is easy to verify that any such code must contain at least $n-1$ codewords.  In another paper, Charon, Hudry and Lobstein take this idea even further.

\begin{thm}[\cite{Charon2005}]
   For every integer $r\ge 1$ and a integer $n$ sufficiently large with respect to $r$, that for every integer $k$ in the interval $[\lceil \log_2(n+1)\rceil,n-1]$, there is a graph $G$ that admits a minimum $r$-identifying code of size $k$.
\end{thm}

Since we can find graph with codes of basically any size, we usually desire to impose some sort of structure on our graph in order to get any meaningful result.  One useful way to do this is to impose some sort of regularity condition on our graph.

\begin{thm}[\cite{Karpovsky1998}]\label{thm:karpovball}
  If $C$ is a code for a graph $G$ of order $n$ and for each vertex $v\in V(G)$ we have $|B_r(v)|=b_r$, then $$|C|\ge\frac{2n}{b_r+1}.$$
\end{thm}

\begin{proof}
  Let $C=\{c_1,c_2,\ldots, c_k\}$.  We start by summing the size of the identifying sets. $$\sum_{v\in V(G)} |I_r(v)|.$$ On the other hand, for each $c_j$, we have $c_j\in I_r(v_i)$ for exactly $|B_r(c_j)|$ distinct values of $i$.  Hence, we have $$\sum_{j=1}^k |B_r(c_j)| = \sum_{v\in V(G)} |I_r(v)|.$$  However, since $|B_r(v)|=b_r$ for any vertex, this gives  $$b_r|C| = \sum_{v\in V(G)} |I_r(v)|.$$  Now we see that there can be at most $|C|$ identifying sets of size 1 and all the rest must have size at least 2.  This gives a lower bound for the right hand side  $$b_r|C| \ge |C| + 2(n-|C|).$$  Rearranging the inequality gives the desired result.
\end{proof}

  In particular, this gives a bound for identifying codes of regular graphs.
  \begin{thm}\label{thm:regularGraphCodeLowerBound}If $G$ is $d$-regular, then any code $C$ for $G$ satisfies:
    $$|C| \ge \frac{2n}{d+1}.$$
  \end{thm}

Further discussion of codes on regular graphs and constructions of codes matching various lower bounds can be found in Chapter~\ref{chapter:regulargraphcodes}.
\subsection{Codes and Infinite Graphs}

Of particular interest to researchers who study codes are certain infinite graphs, particularly the square grid, hexagonal grid, triangular grid and the king grid.  One of the nice things about these grids is that they can all be defined so that their vertex sets are $\Z^2$.  Let $G_S$ denote the square grid, $G_H$ denote the hexagonal grid, $G_T$ denote the triangular grid and $G_K$ denote the king grid.  Then we have:
\begin{eqnarray*}
  E(G_S) &=& \{\{u=(i,j),v): u-v\in \{(0,\pm1),(\pm1,0)\}\}\\
  E(G_H) &=& \{\{u=(i,j),v): u-v\in \{(0,(-1)^{i+j+1}),(\pm1,0)\}\}\\
  E(G_T) &=& \{\{u=(i,j),v): u-v\in \{(0,\pm1),(\pm1,0),(1,1),(-1,-1)\}\}\\
  E(G_K) &=& \{\{u=(i,j),v): u-v\in \{(0,\pm1),(\pm1,0),(\pm 1, 1), (\pm1,-1)\}\}\\
\end{eqnarray*}
In the square grid, this gives exactly the representation that one would expect.  Drawing the graph in the Euclidian plan gives a tiling of the plane covered with squares.  For the hexagonal grid, this gives the so-called ``brick wall'' representation of the graph.  However, this is isomorphic to the graph obtained by tiling of the plane with hexagons.  Similarly, for the triangular grid, this gives a tiling of the plane with right angled isosceles triangles rather than the usual equilateral tiling that we would expect.  The king grid is the only one of these graphs that is not planar.  The king grid represents the graph where the vertices are squares on an infinite chess board and the edges represent the legal moves that a king could make on this chess board.  Sometimes the king grid is known as the square grid with diagonals.

Since these graphs are all infinite, it is clear that any vertex identifying code must also be infinite.  Hence, it is impossible to speak of the number of code words in a minimal $r$-identifying code.  Thus we usually like to speak of the density of a code on one of these graphs.  Roughly speaking, this is the proportion of vertices in the graph that are codewords.  However, this definition doesn't have a precise mathematical meaning.  In practice, any reasonable way of counting the density of a code will give the correct density, however, we still need a definition for this to be precise.  Let $Q_m=[-m,m]\times[-m,m]$.
\begin{defn}\label{defn:densityOfInifiniteCode}
   The density of a code $C$ is $$D(C)=\limsup_{m\rightarrow\infty}\frac{|C\cap Q_m|}{|Q_m|}.$$
\end{defn}

In Chapter~\ref{chapter:latticepaper} we extend this definition to graphs whose vertex set is $\Z^n$.

This definition is analogous to how we would define the density of a code on a finite graph $G$ (i.e. $D(C)=|C|/|V(G)|$).  In fact, Theorem~\ref{thm:karpovball} and Theorem~\ref{thm:regularGraphCodeLowerBound} extend to infinite graphs as well.

\subsubsection{1-Identifying Codes for Infinite Graphs}

\begin{thm}[\cite{Karpovsky1998}]\label{thm:regularInfiniteGraphCodeLowerBound}
  If $C$ is an $r$-identifying code for one of our infinite grids and for each $v\in V(G)$ we have $|B_r(v)|=b_r$, then $$D(C)\ge\frac{2}{b_r+1}.$$  In particular, if $G$ is $d$-regular, the density of an identifying code must be at least $2/(d+2)$.
\end{thm}

We will use $\mathcal{D}(G,r)$ to denote the minimum density of an $r$-identifying code for a graph $G$.  Since all of the aforementioned infinite graphs are regular, this gives what are known as the trivial lower bounds for infinite graphs: $$\begin{array}{cc}
    \mathcal{D}(G_T,1)\ge 1/4; & \mathcal{D}(G_S,1)\ge 1/3; \\
    \mathcal{D}(G_H,1)\ge 2/5; & \mathcal{D}(G_K,1)\ge 1/5.
  \end{array}
$$

However, the only one of these bounds that is attainable is the triangular grid.  In \cite{Karpovsky1998} a code of density 1/4 is constructed, showing that the above bound is indeed tight.
The square grid has been extensively studied.  Bounds were given in~\cite{Cohen1999},~\cite{Cohen1999a} and~\cite{Karpovsky1998} before it was shown in~\cite{Ben-Haim2005} that the $\mathcal{D}(G_S,1)=7/20$.

It was shown in \cite{Cohen2001} that $\mathcal{D}(G_K,1)\ge 2/9$ and in \cite{Charon2001} it was shown that this bound was tight.  More generally, it was shown in \cite{Charon2004} that for $r\ge 2$ that $\mathcal{D}(G_K,r)=1/(4r)$.

Of all the major grids, only $\mathcal{D}(G_H,1)$ remains an open question.  In \cite{Cohen2000}, two constructions of density 3/7 are given.  In addition, it is shown that $\mathcal{D}(G_H,1)\ge 16/39$ which has since been improved to $\mathcal{D}(G_H,1)\ge 12/29$ in~\cite{Cranston2009}.  In addition, Ari Cukierman and Gexin Yu~\cite{CukiermanSubmitted} have reported the existence of at least 3 more identifying codes of density 3/7, which are non-isomorphic to the ones given in \cite{Cohen2000} as well as improving the lower bound to $5/12$.

\subsubsection{$r$-Identifying Codes for Infinite Graphs}

We next wish to turn our attention to the more general $r$-identifying codes for infinite graphs, which will be the focus of the majority of this thesis.  In general, Theorem~\ref{thm:regularGraphCodeLowerBound} gives a very poor lower bound for these densities since $b_r = \Theta(r^2)$ (See \S 5.1 in \cite{Charon2002}).  An exception to this is that $\mathcal{D}(G_H,2)\ge 2/11$, although we improve on this bound in Chapter \ref{chapter:squaregridpaper}.  In general, however, these lower bounds all have the form $\Theta(1/r)$.

We have already addressed the case of the king grid, so we will only mention the bounds for the hexagonal, square and triangular grids.  Most of the bounds in given in \cite{Charon2001} still stand as the best known general bounds. These bounds are:

$$
    \begin{array}{ccl}
      \mathcal{D}(G_H,r) & \ge & \left\{\begin{array}{cc}
                                   \displaystyle \frac{2}{5r+3} & \text{for $r$ even} \\
                                   \displaystyle\frac{2}{5r+2} & \text{for $r$ odd}
                                 \end{array}\right.
       \\
       &&\\
      \mathcal{D}(G_S,r) & \ge & \displaystyle\frac{3}{8r+4} \\
      &&\\
      \mathcal{D}(G_T,r) & \ge & \displaystyle\frac{2}{6r+2}. \\
    \end{array}
$$

Other lower bounds for the hexagonal, square and triangular grids were given in \cite{Cohen2001a}, \cite{Honkala2002} and \cite{Cohen2001a} respectively.  In \cite{Martin2010}, which is Chapter~\ref{chapter:squaregridpaper}, we improve on some of these bounds for small values of $r$.  Our lower bound $\mathcal{D}(G_H,2)\ge 1/5$ has since been improved upon in~\cite{JunnilaSubmitted} where it is shown that $\mathcal{D}(G_H,2)=4/19$ matching the upper bound given in~\cite{Charon2002}.

For upper bounds, it is usually easiest to find a construction of a code with a given density.  For small values of $r$, ad hoc constructions are usually best.  In particular, \cite{Charon2002} used a computer program to find periodic tilings for the triangular grid ($2\le r\le 6$), square grid ($2\le r\le 6$), and hexagonal grid ($2\le r\le 30$).

There have also been general constructions of codes for arbitrary values of $r$.  In \cite{Honkala2002} it is shown that $$\mathcal{D}(G_S,r)\le\left\{\begin{array}{cc}
                                     \displaystyle\frac{2}{5r} & \text{for $r$ even} \\
                                     &\\
                                     \displaystyle\frac{2r}{5r^2-2r+1}& \text{for $r$ odd}.
                                   \end{array}
\right.$$
In \cite{Charon2001} it is shown that $$\mathcal{D}(G_T,r)\le\left\{\begin{array}{cc}
                                     \displaystyle\frac{1}{2r+4} & \text{for $r\equiv 0 \pmod 4$ } \\
                                     &\\
                                     \displaystyle\frac{1}{2r+2}& \text{for $r\equiv 1,2,3 \pmod 4$}.
                                   \end{array}
\right.$$
Also in \cite{Charon2001} it is shown that $$\mathcal{D}(G_H,r)\le\left\{\begin{array}{cc}
                                     \displaystyle\frac{8r-8}{9r^2-16r} & \text{for $r\equiv 0 \pmod 4$ } \\
                                     &\\
                                     \displaystyle\frac{8}{9r-25} & \text{for $r\equiv 1 \pmod 4$ } \\
                                     &\\
                                     \displaystyle\frac{8}{9r-34} & \text{for $r\equiv 2 \pmod 4$ } \\
                                     &\\
                                     \displaystyle\frac{8r-16}{(r-3)(9r-43)} & \text{for $r\equiv 3 \pmod 4$ }.
                                   \end{array}
\right.$$
In \cite{Stanton2011}, which is the basis of Chapter~\ref{chapter:hexgridpaper}, we improve on this last bound, giving general constructions of density $$\frac{5r+3}{6r(r+1)},  \text{ if $r$ is even} ;\qquad
            \frac{5r^2+10r-3}{(6r-2)(r+1)^2},  \text{ if $r$ is odd}.$$  In addition, these improve on the bounds given in \cite{Charon2001} for $15\le r \le 30$, $r\neq 17,21$.

One other type of infinite graph that is of interest is the $n$-dimensional lattice (denoted by $L_n$), which is the $n$-dimensional analogue of the square grid.  This is defined formally in Section~\ref{section:latticedefs} of Chapter~\ref{chapter:latticepaper} which is the based on~\cite{StantonSubmitted}.


One particular noteworthy case presented in \cite{Karpovsky1998} is that $\mathcal{D}(L_n,1)=1/(n+1)$ if and only if $n=2^k-1$ for some integer $k$.  For any $n$, it is known that $1/(n+1)$ is a lower bound for this number, although it is not possible to achieve this bound in most cases.  In Chapter~\ref{chapter:latticepaper} we strive to achieve good asymptotics, showing that the lower bound is close to the minimum density as well as addressing to more general issue of $r$-identifying codes on the $n$-dimensional lattice.

\subsection{Variants of $r$-Identifying Codes}

The concept of an identifying code has been generalized in many ways--many of which are adapted to tackle a real life problem such as sensor networks.  Most of these variants will not be discussed in detail here, but we wish to mention some of them because they have both practical and mathematical importance.

Perhaps the simplest of these variants is to simply drop condition (1) of Definition~\ref{defn:identifyingcode}--that is, we don't care if one of our vertices has an empty identifying set.  Although this doesn't have a practical application, it is a natural variant to consider and has been used, for instance in~\cite{Blass2000}, to aid in finding constructions of identifying codes for other graphs.  As it turns out, the lack of this first condition cannot affect the size of a code too much.  Let $G$ be a graph and let $M_r(G)$ be the minimum density of a traditional $r$-identifying code for $G$ and let $\overline{M}_r(G)$ denote the minimum size of an identifying code while allowing at most one empty identifying set.

\begin{thm}[\cite{Blass2000}]
   Let $G$ be a finite graph admitting an identifying code.  Then  $$\overline{M}_r(G)\le M_r(G)\le \overline{M}_r(G)+1.$$
\end{thm}
\begin{proof}
  Let $C$ be a minimum identifying code allowing at most one empty set.  Then either $C$ is a traditional identifying code or $I_r(v)=\emptyset$ for exactly one vertex.  Then, from Lemma~\ref{lemma:codeExistenceSubsetLemma}, $C\cup \{v\}$ is also a code and has no empty identifying sets.
\end{proof}

In addition, this shows that if $G$ is an infinite graph, then the omission of condition (1) of Definition~\ref{defn:identifyingcode} would have no effect on the density of a code since the density of a single vertex in an infinite graph is 0.

Another variant of an identifying code is to consider a code which not only distinguishes between individual vertices, but one that is able to distinguish between small subsets of vertices.  If $X$ is a subset of vertices of some graph $G$ and $C$ is a code, then we define $$I_r(X)=\bigcup_{x\in X} I_r(x).$$  This brings us to our next definition:

\begin{defn}\label{defn:identifyingcode}
  Given a graph $G$, a code $C$ is called \emph{$(r,\le \ell)$-identifying} if for each distinct $X,Y\subset V(G)$ with $|X|,|Y|\le \ell$ we have
  \begin{enumerate}
    \item $I_r(X)\neq \emptyset$ and
    \item $I_r(X)\neq I_r(Y).$
  \end{enumerate}
\end{defn}

These codes have been well studied, for instance in \cite{Honkala2003a}, \cite{Honkala2003}, \cite{Honkala2004a}, and \cite{Laihonen2001}.  In section~\ref{section:r2hexcodes} we give a lower bound for the minimum density of $(r,\le 2)$-identifying codes of the hex grid.

Yet another important variant of identifying codes is the idea of a locating-dominating set.  Originally introduced by Rall and Slater~\cite{Rall1984}, years before the introduction of the concept of identifying codes, the idea here is that the processors that are marked as codewords are also able to explicitly communicate when they themselves experience an error.  In this case, the second condition of Definition~\ref{defn:identifyingcode} is changed so that $C$ need only distinguish between noncodewords.

The flip side to locating dominating codes is the assumption that a faulty processor may not be able to report if it has experienced an error.  These are called \emph{strongly identifying codes} and were introduced in~\cite{Honkala2002a}.  In this case, we need the sets $\{I_r(v),I_r(v)\setminus\{v\}\}$ to be disjoint in all cases.  These have also been studied in other places, for instance~\cite{Honkala2010} and~\cite{Laihonen2002}.

It is easy to see from the definitions that $$\text{Strongly Identifying Codes}\subset \text{Identifying Codes}\subset \text{Locating-Dominating Sets}$$ and that the inclusions are strict.  Every graph admits a locating-dominating set since you may simply take the entire graph as the locating-dominating set.  On the other hand, many graphs do not admit Identifying codes, the simplest example being $K_n$--the complete graph on $n$ vertices.  We also see that $C_4$ admits an identifying code.  By Theorem~\ref{thm:logbound} the code must have size at least 3.  Assume that the code is $\{1,2,3\}\subset C$ where 2 is adjacent to 1 and 3.  Then $I_1(4)\setminus \{4\}=\{1,3\}$ regardless of whether or not $4\in C$ and so $\{1,3\}\in\{I_1(4),I_1(4)\setminus \{4\}\}$.  Then $\{I_1(2),I_1(2)\setminus \{2\}\}=\{\{1,2,3\},\{1,3\}\}$ and so these sets are not disjoint.

Finally, we will briefly mention some other variants of identifying codes before moving on to the main part of the thesis.  One concept that is commonly studied is to try to find an identifying code which is still a code if vertices or edges are removed from (or added to) the graph.  This has applications in sensor networks since failures are common and we wish to make sure we still have a code even if something has failed.  These are defined more formally in~\cite{Honkala2006} and have been studied in many other papers, for instance \cite{Laihonen2005} and \cite{Honkala2007}.  Another interesting and recent adaptation of these codes is to help identify an intruder which was discussed in \cite{Roden2009}.
\chapter{LOWER BOUNDS FOR IDENTIFYING CODES IN SOME INFINITE GRIDS}\label{chapter:squaregridpaper}
\begin{center} Based on a paper published in \emph{Electronic Journal of Combinatorics}\bigskip \\ Ryan Martin and Brendon Stanton\end{center}

\section*{Abstract}
\addcontentsline{toc}{section}{Abstract}
An $r$-identifying code on a graph $G$ is a set $C\subset V(G)$ such that for every vertex in $V(G)$, the intersection of the radius-$r$ closed neighborhood with $C$ is nonempty and unique.  On a finite graph, the density of a code is $|C|/|V(G)|$, which naturally extends to a definition of density in certain infinite graphs which are locally finite.  We present new lower bounds for densities of codes for some small values of $r$ in both the square and hexagonal grids.

\section{Introduction}

Given a connected, undirected graph $G=(V,E)$, we define
$B_r(v)$--called the ball of radius $r$ centered at $v$ to be
$$B_r(v)=\{u\in V(G): d(u,v)\le r\}. $$

We call any nonempty subset $C$ of $V(G)$ a \textdef{code} and its elements \textdef{codewords}.  A code $C$ is called \textdef{$r$-identifying} if it has the properties:
\begin{enumerate}
\item $B_r(v) \cap C \neq \emptyset$
\item $B_r(u) \cap C \neq B_r(v)\cap C$, for all $u\neq v$
\end{enumerate}
When $C$ is understood, we define $I_r(v)=I_r(v,C)=B_r(v)\cap C$.  We call
$I_r(v)$ the identifying set of $v$.

Vertex identifying codes were introduced in~\cite{Karpovsky1998} as a way to help with fault diagnosis in multiprocessor computer
systems.  Codes have been studied in many graphs, but of particular interest are codes in the infinite triangular, square, and hexagonal lattices as well as the square lattice with diagonals (king grid). For each of these graphs, there is a characterization so that the vertex set is $\Z\times\Z$.  Let $Q_m$ denote the set of vertices $(x,y)\in \Z\times\Z$ with $|x|\le m$ and $|y|\le m$.  We may then define the density of a code $C$ by
$$D(C)=\limsup_{m\rightarrow\infty}\frac{|C\cap Q_m|}{|Q_m|}.$$

Our first two theorems, Theorem~\ref{hexr2theorem} and Theorem~\ref{squarer2theorem}, rely on a key lemma, Lemma~\ref{generalpairlemma}, which gives a lower bound for the density of an $r$-identifying code assuming that we are able to show that no codeword appears in ``too many'' identifying sets of size 2.  Theorem~\ref{hexr2theorem} follows immediately from Lemma~\ref{generalpairlemma} and Lemma~\ref{hex2lemma} while Theorem~\ref{squarer2theorem} follows immediately from Lemma~\ref{generalpairlemma} and Lemma~\ref{squarepairlemma1}.

    \begin{thm}\label{hexr2theorem} The minimum density of a 2-identifying code of the
    hex grid is at least 1/5.
    \end{thm}

    \begin{thm}\label{squarer2theorem}
    The minimum density of a 2-identifying code of the
    square grid is at least $3/19\approx 0.1579$.
    \end{thm}

Theorem~\ref{squarer2theorem} can be improved via Lemma~\ref{strongsqpairlemma}, which has a more detailed and technical proof than the prior lemmas.  The idea the lemma is that even though it is possible for a codeword to be in 8 identifying sets of size 2, this forces other potentially undesirable things to happen in the code.  We use the discharging method to show that on average a codeword can be involved in no more than 7 identifying sets of size 2. Lemma~\ref{strongsqpairlemma} leads to the improvement given in Theorem~\ref{squarer2theorem}.

   \begin{thm}\label{theorem:mainsquare}
     The minimum density of a 2-identifying code of the
     square grid is at least $6/37\approx 0.1622$.
   \end{thm}

The paper is organized as follows: Section~\ref{generalsection} focuses on some key definitions that we use throughout the paper, provides the proof of Lemma~\ref{generalpairlemma} and provides some other basic facts.  Section~\ref{hexsection} states and proves Lemma~\ref{hex2lemma} from which Theorem~\ref{hexr2theorem} immediately follows.  Similarly, we may prove the following theorem:

\begin{thm}\label{hexr3theorem} The minimum density of a 3-identifying code of the
    hex grid is at least 3/25.
\end{thm}

The proof of this fact occurs in Chapter~\ref{chapter:squareAddendum}.  Section~\ref{squaresection} gives the proofs of Lemma~\ref{squarepairlemma1} and~\ref{strongsqpairlemma}.  Finally, in Section~\ref{sec:conc}, we give some concluding remarks and a summary of known results.


\section{Definitions and General Lemmas}\label{generalsection}

Let $G_S$ denote the square grid.  Then $G_S$ has vertex set $V(G_S)=\Z\times\Z$ and
$$ E(G_S)=\{\{u,v\}: u-v\in \{(0, \pm 1),(\pm 1,0)\}\} , $$
where subtraction is performed coordinatewise.

Let $G_H$ represent the hex grid.  We will use the so-called ``brick wall'' representation, whence $V(G_H)=\Z\times\Z$ and
$$ E(G_H)=\{\{u=(i,j),v\}: u-v\in \{(0,(-1)^{i+j+1}),(\pm 1,0)\}\} . $$

Consider an $r$-identifying code $C$ for a graph $G=(V,E)$.  Let
$c,c'\in C$ be distinct.  If $I_r(v)=\{c,c'\}$ for some $v\in V(G)$
we say that \begin{enumerate}
\item $c'$ \textdef{forms a pair} (with $c$) and
\item $v$ \textdef{witnesses a pair} (that contains $c$).
\end{enumerate}

For $c\in C$, we define the set of witnesses of pairs that contain $c$.  Namely,
$$ P(c)=\{v:I_r(v)=\{c,c'\}, \text{ for some }c'(\neq c)\} . $$
We also define $p(c)=|P(c)|$.  In other words,
$P(c)$ is the set of all vertices that witness a pair containing $c$ and $p(c)$ is the number of vertices that witness a pair containing $c$. Furthermore, we call $c$ a \textdef{$k$-pair codeword} if $p(c)=k$.

We start by noting two facts about pairs which are true for any code on any graph.

\begin{fact}\label{subsetfact1}
    Let $c$ be a codeword and $S$ be a subset of $P(c)$.  If $v\not\in S$ and
    $B_2(v)\subset\bigcup_{s\in S}B_2(s)$, then $v\not\in P(c)$.
\end{fact}

\begin{proof}
Suppose $v$ witnesses a pair containing $c$.  Hence, $I_2(v)=\{c,c'\}$ for some $c'\neq c$.  Then $c'\in B_2(v)$ and so $c'\in B_2(s)$ for some $s\in S$.  But then $\{c,c'\}\subset I_2(s)$.  But since $I_2(s)\neq I_2(v)$, $|I_2(s)|>2$, contradicting the fact that $s$ witnesses a pair.  Hence $v$ does not witness such a pair.
\end{proof}

\begin{fact}\label{subsetfact2}
    Let $c$ be a codeword and $S$ be any set with $|S|=k$.  If $v\in S$ and $$B_2(v)\subset\bigcup_{s\in S\atop s\neq v}B_2(s)$$ then at most $k-1$ vertices in $S$ witness pairs containing $c$.
\end{fact}

\begin{proof}
The result follows immediately from Fact~\ref{subsetfact1}.  If each vertex in $S-\{v\}$ witnesses a pair, then $v$ cannot witness a pair.  Hence, either $v$ does not witness a pair or some vertex in $S$ does not witness a pair.
\end{proof}

Lemma~\ref{generalpairlemma} is a general statement about vertex-identifying codes and has a similar proof to Theorem 2 in~\cite{Karpovsky1998}.  In fact, Cohen, Honkala, Lobstein and Z\'emor~\cite{Cohen2001} use a nearly identical technique to prove lower bounds for $1$-identifying codes in the king grid.  Their computations can be used to prove a slightly stronger statement that implies Lemma~\ref{generalpairlemma}.  We will discuss the connection more in Section~\ref{sec:conc}.

\begin{lemma}\label{generalpairlemma}
   Let $C$ be an $r$-identifying code for the square or hex grid.  Let $p(c)\le k$ for any codeword.  Let $D(C)$ represent the density of $C$, then if $b_r=|B_r(v)|$ is the size of a ball of radius $r$ centered at any vertex $v$,
   $$ D(C)\ge\frac{6}{2b_r+4+k} . $$
\end{lemma}

\begin{proof}
   We first introduce an auxiliary graph $\Gamma$.  The vertices of $\Gamma$ are the vertices in $C$ and $c$ is adjacent to $c'$ if and only if $c$ forms a pair with $c'$.  Then we clearly have $\deg_\Gamma(c)=p(c)$.  Let $\Gamma[C\cap Q_m]$ denote the induced subgraph of $\Gamma$ on $C\cap Q_m$.  It is clear that if $\deg_\Gamma(c)\le k$ then $\deg_{\Gamma[C\cap Q_m]}\le k$.

   The total number of edges in $\Gamma[C\cap Q_m]$ by the handshaking lemma is $$ \frac12\sum_{c\in \Gamma[C\cap Q_m]}\deg_{\Gamma[C\cap Q_m]} \le (k/2)|C\cap Q_m| .$$
   But by our observation above, we note that the total number of pairs in $C\cap Q_m$ is equal to the number of edges in $\Gamma[C\cap Q_m]$.  Denote this quantity by $P_m$.  Then
   $$ P_m\le (k/2)|C\cap Q_m| . $$

   Next we turn our attention to the grid in question.  The arguments work for either the square or hex grid.  Note that if $C$ is an $r$-identifying code on the grid, $C\cap Q_m$ may not be a valid $r$-identifying code for $Q_m$.  Hence, it is important to proceed carefully.  Fix $m>r$.  By definition, $Q_{m-r}$ is a subgraph of $Q_m$.  Further, for each vertex $v\in V(Q_{m-r})$, $B_r(v)\subset V(Q_m)$.  Hence $C\cap Q_m$ must be able to distinguish between each vertex in $Q_{m-r}$.

   Let $n=|Q_m|$ and $K=|C\cap Q_m|$.  Let $v_1,v_2,v_3,\ldots, v_n$ be the vertices of $Q_m$ and let $c_1,c_2,\ldots, c_K$ be our codewords.  We consider the $n\times K$ binary matrix $\{a_{ij}\}$ where $a_{ij}=1$ if $c_j\in I_r(v_i)$ and $a_{ij}=0$ otherwise.  We count the number of non-zero elements in two ways.

   On the one hand, each column can contain at most $b_r$ ones since each codeword occurs in $B_r(v_i)$ for at most $b_r$ vertices.  Thus, the total number of ones is at most $b_r\cdot K$.

   Counting ones in the other direction, we will only count the number of ones in rows corresponding to vertices in $Q_{m-r}$.  There can be at most $K$ of these rows that contain a single one and at most $P_m$ of these rows which contain 2 ones.  Then there are $|Q_{m-k}|-K-P_m$ left corresponding to vertices in $Q_{m-k}$ and so there must be at least 3 ones in each of these rows.  Thus the total number of ones counted this way is at least $K +
   2P_m+3(|Q_{m-r}|-K-P_m) = -2K +3|Q_{m-r}|-P_m$.  Thus
   \begin{equation}   \label{eqn:mainineq} b_rK\ge-2K +3|Q_{m-r}|-P_m .
  \end{equation}

   But since $P_m\le (k/2)K$, this gives
   $$ b_rK\ge-2K +3|Q_{m-r}|-(k/2)K . $$
   Rearranging the inequality and replacing $K$ with $|C\cap Q_m|$ gives $$ \frac{|C\cap Q_m|}{|Q_{m-r}|}\ge \frac{6}{2b_r +4+k} .$$

   Then
   \begin{eqnarray*}
      D(C) & = & \limsup_{m\rightarrow\infty}\frac{|C\cap Q_m|}{|Q_m|} \\
      & = & \limsup_{m\rightarrow\infty}\frac{|C\cap Q_{m}|}{|Q_{m-r}|}\cdot \limsup_{m\rightarrow\infty}\frac{|Q_{m-r}|}{|Q_m|} \\
      & \ge & \frac{6}{2b_r +4+k}\cdot \limsup_{m\rightarrow\infty}\frac{(2(m-r)+1)^2}{(2m+1)^2} \\
      & = & \frac{6}{2b_r +4+k} .
   \end{eqnarray*}
\end{proof}

\section{Lower Bound for the Hexagonal Grid}\label{hexsection}

Lemma~\ref{hex2lemma} establishes an upper bound of 6 for the degree of the graph $\Gamma$ formed by an $r$-identifying code in the hex grid, which allows us to prove Theorem~\ref{hexr2theorem}.

\begin{lemma}\label{hex2lemma}
Let $C$ be a 2-identifying code for the hex grid. For each $c\in C$, $p(c)\le 6$.
\end{lemma}

\begin{proof}
    Let $C$ be an $r$-identifying code and $c\in C$ be an arbitrary codeword.  Let $u_1,u_2,$ and $u_3$ be the neighbors of $c$ and let
    $\{u_{i1},u_{i2}\}=B_1(u_i)-\{u_i,c\}$.

    \noindent\textbf{Case 1:}  $|I_2(c)|\ge 2$

    There exists some $c'\in C\cap B_2(c)$ with $c'\neq c$.  Without
    loss of generality, assume that $c'\in \{u_1,u_{11},u_{12}\}$.
    Since $I_2(c),I_2(u_1),I_2(u_{11}),I_2(u_{12})\supseteq \{c,c'\}$
    at most one of $c,u_1,u_{11},u_{12}$ witnesses a pair containing
    $c$.

    Now, $p(c)\le 6$ unless each of $u_2,u_3,u_{21},u_{22},
    u_{31},u_{32}$ witnesses a pair.

    If $u_2$ and $u_3$ each witness a pair, then we have
    $u_i\not\in C$ for $i=1,2,3$; otherwise
    $I_2(u_2)=\{c,u_i\}=I_2(u_3)$ and so $u_2$ and $u_3$ are not
    distinguishable by our code.  Thus, there must be some $c''\in
    C\cap (B_2(u_2)-\{c,u_1,u_2,u_3\})$. This forces
    $c''\in  B_2(u_{21})\cup B_2(u_{22})$ and so either
    $\{c,c''\}\subseteq I_2(u_{21})$ or $\{c,c''\}\subseteq I_2(u_{22})$. Hence, one of these cannot witness a pair and still be
    distinguishable from $u_2$.  This ends case 1.

    \noindent\textbf{Case 2:}  $I_2(c) = \{c\}$

    First note that $c$ itself does not witness a pair.

    If $u_1$ witnesses a pair, then there is some $c''\in C\cap(
    B_2(u_1)-B_2(c)) \subseteq C\cap ( B_2(u_{11})\cup
    B_2(u_{12}))$ and so either
    $\{c,c''\}\subseteq I_2(u_{11})$ or $\{c,c''\}\subseteq I_2(u_{12})$
    and so one of these cannot witness a pair and still be
    distinguishable from $u_1$.  Hence at most two of
    $\{u_1,u_{11},u_{12}\}$ can witness a pair.

    Likewise at most at most two of
    $\{u_2,u_{21},u_{22}\}$ and
    $\{u_3,u_{31},u_{32}\}$ can witness a pair. Thus $p(c)\le 6$.
    This ends both case 2 and the proof of the lemma.
\end{proof}

\begin{proofcite}{Theorem~\ref{hexr2theorem}}
   Using Lemmas~\ref{generalpairlemma} and~\ref{hex2lemma}, if $C$ is a $2$-identifying code in the hexagonal grid, then
   $$ D(C)\geq\frac{6}{2b_2+4+6}=\frac{6}{30}=\frac{1}{5} . $$
\end{proofcite}

\section{Lower Bounds for the Square Grid}\label{squaresection}

Lemma~\ref{squarepairlemma1} establishes an upper bound of 8 for the degree of the graph $\Gamma$ formed by an $r$-identifying code in the square grid, which allows us to prove Theorem~\ref{squarer2theorem}.  Then we prove Lemma~\ref{strongsqpairlemma}, which bounds the average degree of $\Gamma$ by 7, allowing for the improvement in Theorem~\ref{theorem:mainsquare}.

It is worth noting that the proof of Lemma~\ref{squarepairlemma1} could be shortened significantly, but the proof is needed in order to prove Lemma~\ref{strongsqpairlemma}, which gives the result in Theorem~\ref{theorem:mainsquare}.

\begin{lemma}\label{squarepairlemma1}
Let $C$ be a 2-identifying code for the square grid. For each $c\in C$, $p(c)\le 8$.
\end{lemma}

\begin{proof}
Let $c\in C$, a $2$-identifying code in the square grid. Without loss of generality, we will assume that $c=(0,0)$.

\begin{figure}[ht]
    \centering
    \includegraphics{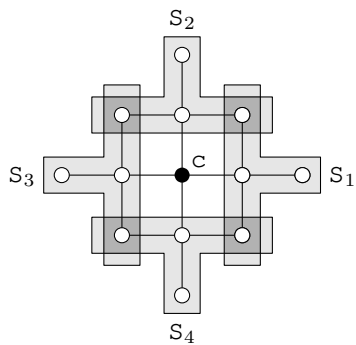}
    \caption[The sets $S_1$, $S_2$, $S_3$ and $S_4$ surrounding a codeword in the square grid]{The sets $S_1$, $S_2$, $S_3$ and $S_4$.}\label{fig:ssets}
\end{figure}

\begin{figure}[ht]
\centering
\includegraphics{squarecenter.mps}
\caption[A proof by picture that $p(c)\le 8$ when $|I_2(c)|=2$ in the square grid]{The ball of radius 2 around $c$.  A configuration of 9
vertices witnessing pairs is not possible if $|I_2(c)|= 2$.\newline $\bullet$~At most 7 of the vertices in gray triangles may
witness a pair.\newline $\bullet$~At most one of the vertices in white
triangles may witness a pair.}\label{fig:squarecenter}
\end{figure}

\textbf{Case 1:} $c$ witnesses a pair.

This case implies immediately that $|I_2(c)|=2$.  The other codeword in $I_2(c)$, namely $c'$, is in one of the
following 4 sets, the union of which is $B_2(c)-\{c\}$.  See Figure~\ref{fig:ssets}.
$$ \begin{array}{rrrrr}
      S_1:=\{ & (1,0),  & (1,1),  & (1,-1),  & (2,0)\} \\
      S_2:=\{ & (0,1),  & (1,1),  & (-1,1),  & (0,2)\} \\
      S_3:=\{ & (-1,0), & (-1,1), & (-1,-1), & (-2,0)\} \\
      S_4:=\{ & (0,-1), & (1,-1), & (-1,-1), & (0,-2)\}
\end{array}$$

If, however, $c'\in S_i$, then no $s\in S_i$ can witness a pair because $\{c,c'\}\subseteq I_2(s)$ and $s$ could not be distinguished from $c$. Without loss of generality, assume that $c'\in S_3$.  Thus, all vertices witnessing pairs in $I_2(c)$ are in the set
$$ R:=\left\{(x,y) : (x,y)\in B_2(c), x\geq 0\right\} . $$
But because
$$ B_2\left((1,0)\right)\subseteq\bigcup_{s\in S_1\cup\{c\}}B_2(s) , $$
Fact~\ref{subsetfact1} gives that not all members of $S_1\cup\{c\}$ can witness a pair. See Figure~\ref{fig:squarecenter}.

Therefore, $p(c)\leq 8$ and, without loss of generality, $c'\in S_3$ and at least one element of $S_1$ does not witness a pair.  This ends Case 1.

\textbf{Case 2:} $c$ does not witness a pair.

This case implies immediately that either $|I_2(c)|\geq 3$ or $I_2(c)=\{c\}$.

First suppose $|I_2(c)|\ge 3$. There must be two distinct codewords $c',c''\in S_1\cup S_2 \cup S_3 \cup S_4$. If $c',c''$ are in the same set $S_i$ for some $i$, then $\{c,c',c''\}\subset I_2(s)$ for any $s\in S_i$
and so no vertex in $S_i$ witnesses a pair. Thus, the only vertices which can witness a pair are in $B_2(c)-(S_i\cup \{c\})$.  There are only 7 of these, so $p(c)\le 7$. (See the gray vertices in Figure~\ref{fig:squarecenter}).

If $c'\in S_i$ and $c''\in S_j$ for some $i\neq j$, then only one vertex in each of $S_i$ and $S_j$ can witness a pair.  There are at most 5 other vertices not in $S_i\cup S_j-\{c\}$ and so $p(c)\le 7$.

Thus, if $|I_2(c)|\ge 3$, then $p(c)\le 7$.

\begin{figure}[ht]
\centering
\includegraphics{squaretriangle1.mps}
\caption[A right angle of witnesses in the square grid (1 of 2)]{A right angle of witnesses.  \newline $\bullet$~Black circles indicate codewords.
\newline $\bullet$~White circles indicate non-codewords.
\newline $\bullet$~Gray triangles indicate vertices that witness a pair.  \newline $\bullet$~White triangles indicate vertices that do not witness a pair.  \newline No vertices in $B_2(c)-\{c\}$ can be codewords, neither can those which are distance no more than 2 from two vertices in this right angle of witnesses.}\label{fig:squaretriangle1}
\end{figure}

\begin{figure}[ht]
\centering
\includegraphics{squaretriangle2.mps}
\caption[A right angle of witnesses in the square grid  (2 of 2)]{A right angle of witnesses, continuing from Figure~\ref{fig:squaretriangle1}.  Let $c=(0,0)$. Vertices $(-2,1)$ and $(3,1)$ must be codewords and so none of  \nobreak{$\{(-1,1),(-1,0),(-2,0),(2,0)\}$}
can witness pairs.}\label{fig:squaretriangle2}
\end{figure}

Second, suppose $I_2(c)=\{c\}$.  We will define a \textdef{right angle of witnesses} to be subsets of 3 vertices of $I_2(c)$ that all witness pairs and are one of the following 8 sets: $\{(1,0),(2,0),(1,\pm 1)\}$, $\{(0,1),(0,2),(\pm 1,1)\}$, $\{(-1,0),(-2,0),(-1,\pm 1)\}$, and \linebreak $\{(0,-1),(0,-2),(\pm 1,-1)\}$.
If a right angle is present then, without loss of generality, let it be $\{(0,1),(0,2),(1,1)\}$.  See Figure~\ref{fig:squaretriangle1}.  In order for these all to be witnesses, then $I_2((0,1))$ must have one codeword not in $B_2((0,2))\cup B_2((1,1))$, which can only be $(-2,1)$.  Since $\{(0,0),(-2,1)\}\subseteq B_2((-1,1)),B_2((-1,0)),B_2((-2,0))$, none of those three vertices can witness a pair.

In addition, $I_2((1,1))$ must contain a codeword not in $B_2((0,1))\cup B_2((0,2))$, which can only be $(3,1)$.  See Figure~\ref{fig:squaretriangle2}. Since $\{(0,0),(3,1)\}\subseteq B_2((2,0))$, the vertex $(2,0)$ cannot witness a pair.

Finally, it is not possible for all of $(-1,-1),(0,-1),(1,-1),(0,-2)$ to be witnesses because the only member of $B_2((0,-1))$ that is not in the union of the second neighborhoods of the others is the vertex $(0,1)$, which cannot be a codeword in this case.  Hence, at most 7 members of $B_2(c)$ can witness a pair if $B_2(c)$ has a right angle of witnesses.

Consequently, if $c$ does not witness a pair and $p(c)\geq 8$, then $I_2(c)=\{c\}$ and $B_2(c)$ fails to have a right angle of witnesses.  We can enumerate the remaining possibilities according to how many of the vertices $\{(1,1),(-1,1),(-1,-1),(1,-1)\}$ are witnesses.  If 1, 2 or 3 of them are witnesses and there is no right angle of witnesses, it is easy to see that there are at most 7 witnesses in $B_2(c)$ and so $p(c)\leq 7$.

The first remaining case is if 0 of them are witnesses, implying each of the eight vertices $(\pm 1,0)$, $(\pm 2,0)$, $(0,\pm 1)$ and $(0,\pm 2)$ are witnesses.  The second remaining case is if 4 of them are witnesses.  This implies that at most one of $\{(1,0),(2,0)\}$ are witnesses and similarly for $\{(0,1),(0,2)\}$, $\{(-1,0),(-2,0)\}$ and $\{(0,-1),(0,-2)\}$.

This ends both Case 2 and the proof of the lemma.  So, $p(c)\leq 8$ with equality only if one of two cases in the previous paragraph holds.
\end{proof}

\begin{proofcite}{Theorem~\ref{squarer2theorem}}
   Using Lemmas~\ref{generalpairlemma} and~\ref{squarepairlemma1}, if $C$ is a $2$-identifying code in the square grid, then
   $$ D(C)\geq\frac{6}{2b_2+4+8}=\frac{6}{38}=\frac{3}{19} . $$
\end{proofcite}

\begin{lemma}\label{strongsqpairlemma}
  Let $C$ be an $r$-identifying code for the square grid.
  Then $\sum_{c\in C\cap Q_m}p(c)\le 7|C\cap Q_m|$.
\end{lemma}

\begin{proof}
  Define
  $$ R(c)=\{c': I_2(v)=\{c,c'\}\text{ for some } v\in V(G_S)\}. $$
  Suppose that $p(c)=8$ for some $c\in C$.  We claim that one of the two following properties holds.
  \begin{enumerate}
    \item[(P1)] There exist distinct $c_1,c_2,c_3\in R(c)$
    such that $p(c_1)\le 4$ and $p(c_i) \le 6$ for $i=2,3$.
    \item[(P2)] There exist distinct $c_1,c_2,c_3,c_4,c_5,c_6\in R(c)$
    such that $p(c_i)\le 6$ for all $i$.
  \end{enumerate}

  We will prove this by characterizing all possible $8$-pair
  vertices, but first we wish to define 3 different types of
  codewords.  The definition of each type extends by taking translations and rotations. So, we may assume in defining the types that $c=(0,0)$.

  We say that $c$ is a \textdef{type 1} codeword if $(0,1),(0,-1)\in C$. See Figure~\ref{fig:vertextype1}.

  We say that $c$ is a \textdef{type 2} codeword if $(-1,2),(2,-1)\in C$. See Figure~\ref{fig:vertextype2}.

  We say that $c$ is a \textdef{type 3} codeword if $(-2,1),(2,1)\in C$. See Figure~\ref{fig:vertextype3}.

Claim~\ref{squareclaim1} shows that adjacent codewords do not need to be considered because they are in few pairs.

\begin{claim}\label{squareclaim1}  If $c$ is adjacent to another codeword, then $p(c)\le 6$.\end{claim}

\begin{proof} Without loss of generality, assume that $c=(0,0)$ and that $(0,1)$ is a codeword.  Then
$$ (-1,0), (0,0), (0,1), (0,2), (1,0), (1,1), (-1,0), (-1,1)$$
are all at most distance 2 from both codewords and so at most 1 of them can witness a pair.  Thus, the other 7 do not witness pairs containing $c$.  Since $|B_2(c)|=13$, $p(c)\le 13-7=6$.  This proves Claim~\ref{squareclaim1}.
\end{proof}

Claims~\ref{claimtype1},~\ref{claimtype2} and~\ref{claimtype3} show that types 1, 2 and 3 codewords, respectively, are not in many pairs.

\begin{claim}\label{claimtype1}
  If $c$ is a type 1 codeword, then $p(c)\le 4$.
\end{claim}

\begin{figure}[ht] \centering
\includegraphics{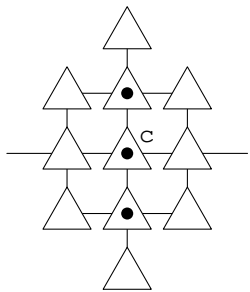}
\caption[A type 1 codeword in the square grid]{Vertex $c$ is a type 1 codeword.  At most 2 of the 11 vertices marked by triangles can witness a pair.}\label{fig:vertextype1}
\end{figure}

\begin{proof} Without loss of generality, let $c=(0,0)$.  We consider all vertices which are distance 2 from $c$ and either $(0,1)$ or $(0,-1)$.  There are 11 such vertices and at most 2 of them can witness pairs, so $p(c)\le 4$.  See
Figure~\ref{fig:vertextype1}. This proves Claim~\ref{claimtype1}.
\end{proof}

\begin{claim}\label{claimtype2}
   If $c$ is a type 2 codeword, then $p(c)\le 6$.
\end{claim}

\begin{figure}[ht]
    \centering
    \includegraphics{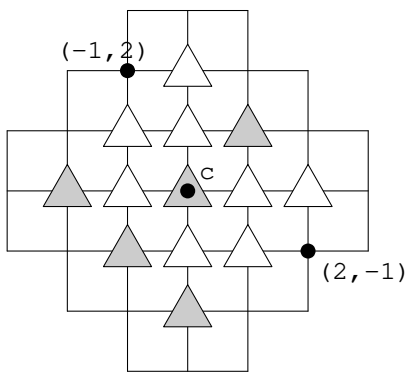}
    \caption[A type 2 codeword in the square grid]{Vertex $c=(0,0)$ is a type 2 codeword.  At most 2 of the 8 vertices marked by white triangles can witness pairs.  At most 4 of the 5 vertices marked by gray triangles can witness pairs.}\label{fig:vertextype2}
\end{figure}

\begin{proof} Without loss of generality, let $c=(0,0)$.  We consider all vertices which are distance at most 2 from $c$ and distance at most 2 from either $(-1,2)$ or $(2,-1)$.  There are 8 such vertices and at most 2 of them can witness pairs.  The remaining 5 vertices are $c$ and the vertices in the set $S=\{(-2,0),(-1,-1),(0,-2),(1,1)\}$.  But then $B_2(c)\subset\bigcup_{s\in S}B_2(s)$ and, by Fact~\ref{subsetfact1} at most 4 of those remaining 5 vertices can witness pairs. Thus, $p(c)\le 6$.  See Figure~\ref{fig:vertextype2}. This proves Claim~\ref{claimtype2}.
\end{proof}

\begin{claim}\label{claimtype3}
   If $c$ is a type 3 codeword, then $p(c)\le 6$.
\end{claim}

\begin{figure}[ht]
    \centering
    \includegraphics{vertextype3.mps}
    \caption[A type 3 codeword in the square grid]{Vertex $c=(0,0)$ is a type 3 codeword. \newline $\bullet$~$T_0$ vertices are black. \newline $\bullet$~$T_1$ vertices are
    white. \newline
    $\bullet$~$T_2$ vertices are marked by diagonal lines. \newline $\bullet$~$T_3$ vertices are gray.}\label{fig:vertextype3}
\end{figure}

\begin{proof} Without loss of generality, let $c=(0,0)$. We partition $B_2(c)- \{c\}$ into 4 sets:
$$ \begin{array}{rrrrr}
      T_0:= \{ &          &         & (0,1),  & (0,2)\} \\
      T_1:= \{ &          & (-2,0), & (-1,0), & (-1,1)\} \\
      T_2:= \{ &          &  (2,0), &  (1,0), &  (1,1)\} \\
      T_3:= \{ & (-1,-1), & (0,-1), & (1,-1), & (0,-2)\} \\
\end{array}$$

At most 1 vertex in $T_0$ witnesses a pair since $|I_2(0,1)|\ge 3$.

At most 1 vertex in $T_1$ can witness a pair since every vertex in $T_1$ is at most distance 2 from $(-2,1)$. Likewise, at most 1 vertex in $T_2$ can witness a pair.

If all vertices in $T_3$ witness pairs, then $I_2((0,-1)) = \{(0,0),(0,1)\}$ since $(0,1)$ is the only vertex in $B_2((0,-1))$ which is not in $B_2(s)$ for any other $s\in T_3$.  But then $c$ is adjacent to another codeword, and by Claim~\ref{squareclaim1}, $p(c)\le 6$.  So we may assume that at most 3 vertices in $T_3$ form pairs with $c$.

Now, if $c$ does not itself witness a pair, these partitions give $p(c)\le 6$.  If $c$ does witness a pair, then there must be another codeword $c'\in S_i$ for some $i$.  But then we see that no other vertex in $S_i$ can witness a pair, since every vertex in $S_i$ is at most distance two from  $c'$.  Thus, $p(c)\le 6$.  See    Figure~\ref{fig:vertextype3}.  This proves Claim~\ref{claimtype3}.
\end{proof}

We are now ready to characterize the 8-pair codewords.

\begin{claim}\label{claimpair}
   If $c\in C$ witnesses a pair and $p(c)=8$, then $c$ satisfies property (P1).
\end{claim}

\begin{figure}[ht]
   \centering
   \includegraphics{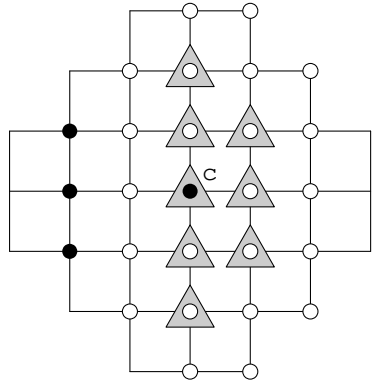}
   \caption[An 8-pair codeword with $|I_2(c)|=2$ in the square grid]{Codeword $c=(0,0)$ witnesses a pair and is an $8$-pair codeword.  The gray triangles are vertices that form pairs with $c$.  Vertex $(-2,0)$ is a type 1 codeword.}\label{fig:squarecpair}
\end{figure}

\begin{proof} Without loss of generality, let $c=(0,0)$.
Recall Case 1 of the proof of Lemma~\ref{squarepairlemma1}.  That is, $p(c)\leq 8$ and, without loss of generality, equality implies that there is a $c'\in C\cap S_3$ and at least one of $S_1=\left\{(1,-1),(1,0),(1,1),(2,0)\right\}$ does not witness a pair.

If $p(c)\leq 7$, the proof is finished, so let us assume that $p(c)=8$ and hence exactly one of the vertices in $S_1$ does not witness a pair.  We will show that it is $(2,0)$.  So, suppose that $(1,y)$ does not witness a pair. Recall that $R=\left\{(x,y) : (x,y)\in B_2(c), x\geq 0\right\}$.

If $y\in\{-1,1\}$, then
$$ B_2\left((1,0)\right)\subseteq\bigcup_{s\in R-\{(1,y),(1,0)\}}B_2(s) $$
and, by Fact~\ref{subsetfact1}, neither $(1,y)$ nor $(1,0)$ witnesses a pair and $p(c)\leq 7$.

If $y=0$, then
$$ B_2\left((1,1)\right)\subseteq\bigcup_{s\in R-\{(1,0),(1,1)\}}B_2(s) $$
and, by Fact~\ref{subsetfact1}, neither $(1,0)$ nor $(1,1)$ witnesses a pair and $p(c)\leq 7$. It follows that each vertex in $R'=R-\{(2,0)\}$ must witness a pair containing $c$.

Each vertex which is distance 2 or less from 2 vertices in $R'$ cannot be a codeword.  Thus, $(-2,0)$ is the only vertex in $B_2(c)$ other than $c$ which has not been marked as a non-codeword and so $(-2,0)\in C$.  Since $(0,0)\in C$, the vertex $(-2,1)$ is the only possibility for a second codeword for $(0,1)$ and $(-2,-1)$ is the only possibility for a second codeword for $(0,-1)$.  See Figure~\ref{fig:squarecpair}.

Then $(-2,0)$ is a type 1 codeword and so it is in at most 4 pairs.  Codewords $(-2,1)$ and $(-2,-1)$ are both adjacent to another codeword, so they are in at most 6 pairs.  Hence, $c$ satisfies Property (P1).  This proves Claim~\ref{claimpair}.
\end{proof}

\begin{claim}\label{claimnopair}
   If $c\in C$ does not witness a pair and $p(c)=8$, then $c$ satisfies either property (P1) or property (P2).
\end{claim}

\begin{proof} Without loss of generality, let $c=(0,0)$.
Recall Case 2 of the proof of Lemma~\ref{squarepairlemma1}.
That is, $p(c)\leq 8$ and, without loss of generality, equality implies $I_2(c)=\{c\}$.  Furthermore, one of the following two cases occurs: \\
(1) The eight witnesses are the vertices $(\pm 1,0)$, $(\pm 2,0)$, $(0,\pm 1)$ and $(0,\pm 2)$. \\
(2) The witnesses include $\{(1,1),(-1,1),(-1,-1),(1,-1)\}$ as well as exactly one of each of the following pairs: $\{(1,0),(2,0)\}$, $\{(0,1),(0,2)\}$, $\{(-1,0),(-2,0)\}$ and \\
$\{(0,-1),(0,-2)\}$.

If case (1) occurs, then the eight witnesses are the vertices $(\pm 1,0)$, $(\pm 2,0)$, $(0,\pm 1)$ and $(0,\pm 2)$.  In this case, simply observe that $B_2((1,0))$ is a subset of the other seven witnesses.  This contradicts Fact~\ref{subsetfact2} and so this case cannot occur.

\begin{figure}[ht]
    \centering
    \includegraphics{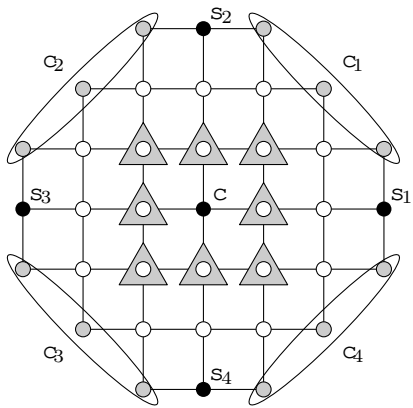}
    \caption[An 8-pair codeword with $|I_2(c)|=1$ in the square grid]{Codeword $c=(0,0)$ fails to witness a pair and is an $8$-pair codeword.  Exactly one of the gray vertices in each oval is a codeword.}\label{fig:square8cw}
\end{figure}

So, we may assume that case (2) occurs.  The vertex $(2,1)$ cannot be a codeword because $\{(0,0),(2,1)\}\subseteq B_2((1,1)),B_2((1,0)),B_2((2,0))$ and so at most one of these three vertices witness pairs, a contradiction to case (2).  By symmetry, none of the following vertices are codewords: $$ (2,1),(1,2),(-1,2),(-2,1),(-2,-1),(-1,-2),(1,-2),(2,-1) . $$

In order to distinguish $(1,0)$ from $(0,0)$, the only vertex available to be a codeword is $s_1:=(3,0)$ and symmetrically, $s_2:=(0,3)$, $s_3:=(-3,0)$ and $s_4:=(0,-3)$ are codewords.  This implies that each of $(1,0)$, $(0,1)$, $(-1,0)$ and $(0,-1)$ witness pairs.

Then, for the other 4 pairs, there are exactly 3 choices for codewords which are not in the ball of radius 2 for any of our other pairs. See Figure~\ref{fig:square8cw}.
$$ \begin{array}{c|rcl}
      \text{Vertex} & \multicolumn{3}{c}{\text{Other Codeword}} \\
      \hline
      (1,1) & c_1 & \in & \{(3,1),(2,2),(1,3)\} \\
      (-1,1) & c_2 & \in & \{(-1,3),(-2,2),(-3,1)\}  \\
      (-1,-1) & c_3 & \in & \{(-3,-1),(-2,-2),(-1,-3)\}  \\
      (1,-1) & c_4 & \in & \{(1,-3),(2,-2),(3,-1)\}
    \end{array} $$

For each $c_i$, either $c_i$ is adjacent to another codeword or $c_i$ is a type 2 codeword.  Claims~\ref{squareclaim1} and~\ref{claimtype2} imply that, in either case, $p(c_i)\le 6$.  It remains to show that one of the following holds:  (1) There exist $i\neq j$ such that $p(s_i)\le 6$ and $p(s_j)\le 6$, hence $c$ satisfies (P2). (2) There exists an $i$ such that $p(s_i)\le 4$, hence $c$ satisfies (P1).

First, suppose that there are $c_i$, $c_j$, $i\neq j$ such that $c_i$ is adjacent to $s_k$ and $c_j$ is adjacent to $s_\ell$.  If $k=\ell$, then $s_k$ is a type 1 codeword and so $p(s_k)\le 4$.  If $k\neq \ell$, then both $s_k$ and $s_\ell$ are adjacent to another codeword and so $p(s_k)\le 6$ and $p(s_\ell)\le 6$.  Either (P1) or (P2) is satisfied, respectively.

If there is at most one $c_i$ such that $c_i$ is adjacent to $s_k$ for some $k$, then we have three codewords of the form $(\pm 2,\pm 2)$.  Without loss of generality, assume that $(2,2),(2,-2),$ and $(-2,2)$ are codewords.  In this case, $(3,0)$ and $(0,3)$ are type 3 codewords and hence $p((3,0))\le 6$ and $p((0,3))\le 6$.  So again, (P2) is satisfied.

This proves Claim~\ref{claimnopair}.
\end{proof}

Finally, we can finish the proof of Lemma~\ref{strongsqpairlemma} by way of the discharging method. (For a more extensive application of the discharging method on vertex identifying codes, see Cranston and Yu~\cite{Cranston2009}.)  Let $\Gamma$ denote an auxiliary graph with vertex set $C\cap Q_m$ for some $m$. There is an edge between two vertices $c$ and $c'$ if and only if $I_2(v)=\{c,c'\}$ for some $v\in V(G_S)$. For each vertex $v$ in our auxiliary graph $\Gamma$, we assign it an initial charge of $d(v)-7$.  Note that $\sum_{c\in C\cap Q_m} p(c)-7 = \sum_{v\in \Gamma}\deg_{\Gamma}(v)-7$.  We apply the following discharging rules if $\deg_{\Gamma}(v)=8$.
\begin{enumerate}
  \item If $v$ is adjacent to one vertex of degree at most 4 and two of degree at most 6 (condition (P1)), then discharge 2/3 to a vertex of degree at most 4 and 1/6 to two vertices of degree at most 6.
  \item If $v$ is adjacent to 6 vertices of degree at most 6 (condition (P2)), then discharge 1/6 to 6 neighbors of degree at most 6.
\end{enumerate}

We have proven that one of the above cases is possible. Let $e(v)$ be the charge of each vertex after discharging takes place.  We show that $e(v)\le 0$ for each vertex in $\Gamma$.

If $\deg_{\Gamma}(v)=8$, then our initial charge was $1$.  In either of the two cases, we are discharging a total of 1 unit to its neighbors.    Since no degree 8 vertex receives a charge from any other vertex, we have $e(v)=0$.

If $d(v)=7$ then its initial charge is 0 and it neither gives nor receives a charge and so $e(v)=0$.

If $5\le \deg_{\Gamma}(v) \le 6$, then its initial charge was at most $-1$. Since this vertex has at most 6 neighbors and can receive a charge of at most $1/6$ from each of them, this gives $e(v)\le 0$.

If $\deg_{\Gamma}(v)\le 4$, then its initial charge was at most $-3$. Since this vertex has at most 3 neighbors and can receive a charge of at most $2/3$ from each of them, this gives $e(v)\le -1/3<0$.

Since no vertex can have degree more than 8, this covers all of the cases.  Then we have
$$ \sum_{c\in C\cap Q_m} (p(c)-7)=\sum_{v\in \Gamma}\left(\deg_{\Gamma}(v)-7\right)=\sum_{v\in \Gamma}e(v) \le 0 . $$
Therefore, it follows that $\sum_{c\in C\cap Q_m} p(c)\le \sum_{c\in C\cap Q_m} 7=7|C\cap Q_m|$.
\end{proof}

\begin{proofcite}{Theorem ~\ref{theorem:mainsquare}}
Consider $Q_m$ and let $C$ be an $r$-identifying code for $G_S$ and $C\cap Q_m=\{c_1,c_2,\ldots, c_K\}$.  Recall inequality (\ref{eqn:mainineq}) from Theorem~\ref{generalpairlemma}.  In this case, $b_2=13$ and Lemma~\ref{strongsqpairlemma} shows that $$ P_m\le \frac12\sum_{c\in C\cap Q_m}p(c)\le\frac 72|C\cap Q_m| .$$

Substituting the above inequality into inequality~(\ref{eqn:mainineq}) and rearranging gives
$$ \frac{|C\cap Q_m|}{|Q_{m-r}|}\ge \frac {6}{37} . $$  Taking the limit as $m\rightarrow\infty$ gives the desired $D(C)\ge 6/37$, completing the proof.
\end{proofcite}

\section{Conclusions}
\label{sec:conc}


The technique used for Lemma~\ref{generalpairlemma} is similar to the one in Cohen, Honkala, Lobstein and Z\'emor~\cite{Cohen2001}.
Define $$\ell = \min_{c\in C}|\{v\in B_r(c): |I_r(v)|\ge 3\}|.$$
An anonymous referee points out that the computations in \cite{Cohen2001} can lead one to conclude that
\begin{equation} D(C)\geq\frac{6}{3b_r+3-\ell} . \label{eq:Cohen2001} \end{equation}
From our definitions
$$k= \max_{c\in C}|\{v\in B_r(c): |I_r(v)|=2\}|.$$
Since $k+\ell\geq b_r-1$, one can use (\ref{eq:Cohen2001}) to derive the result in Lemma~\ref{generalpairlemma}.

As the referee also points out, $k+\ell\leq b_r$, so the denominator could potentially be improved by an additive factor of $1$ if it were possible to show that $k+\ell=b_r$.

Below is a table noting our improvements.

$$\begin{array}{|c|c||c|c|}
  \hline
  \multicolumn{4}{|c|}{\text{Hex Grid}}\\
  \hline
  r & \text{previous lower bounds} & \text{new lower bounds} & \text{upper bounds} \\
  \hline
  2 & 2/11\approx 0.1818 ^\text{~\cite{Karpovsky1998}} & 1/5 = 0.2 & 4/19\approx 0.2105 ^\text{~\cite{Charon2002}}\\
  \hline
  3 & 2/17\approx 0.1176 ^\text{~\cite{Charon2001}} & 3/25 = 0.12 & 1/6\approx 0.1667 ^\text{~\cite{Charon2002}}\\
  \hline
  \multicolumn{4}{|c|}{\text{Square Grid}}\\
  \hline
  2 & 3/20=0.15 ^\text{~\cite{Charon2001}} & 6/37\approx 0.1622 & 5/29\approx 0.1724 ^\text{~\cite{Honkala2002}}\\
  \hline
\end{array}$$

This technique works quite well for small values of $r$, but we note that $b_r=|B_r(v)|$ grows quadratically in $r$, so the denominator in Lemma~\ref{generalpairlemma} would grow quadratically.  But the known the lower bounds for $r$-identifying codes is proportional to $1/r$ in all of the well-studied grids (square, hexagonal, triangular and king).  Therefore, our technique is less effective as $r$ grows.

\phantomsection
\section*{Acknowledgements}

We would like to thank an anonymous referee for making helpful suggestions and directing us to the paper \cite{Cohen2001}.

\chapter{IMPROVED BOUNDS FOR $r$-IDENTIFYING CODES OF THE HEX GRID}\label{chapter:hexgridpaper}

\newcommand{\vsten}{\vspace{10pt}}

\begin{center} Based on a paper published in \emph{SIAM Journal on Discrete Mathematics}\bigskip \\ Brendon Stanton\end{center}

\section*{Abstract}
\addcontentsline{toc}{section}{Abstract}
For any positive integer $r$, an $r$-identifying code on a graph $G$ is a set $C\subset V(G)$ such that for every vertex in $V(G)$, the intersection of the radius-$r$ closed neighborhood with $C$ is nonempty and pairwise distinct.  For a finite graph, the density of a code is $|C|/|V(G)|$, which naturally extends to a definition of density in certain infinite graphs which are locally finite.  We find a code of density less than $5/(6r)$, which is sparser than the prior best construction which has density approximately $8/(9r)$.

\section{Introduction}

Given a connected, undirected graph $G=(V,E)$, define
$B_r(v)$, called the ball of radius $r$ centered at $v$, to be
$$B_r(v)=\{u\in V(G): d(u,v)\le r\}, $$ where $d(u,v)$ is the distance between
$u$ and $v$ in $G$.

We call any $C\subset V(G)$ a code.  We say that $C$ is $r$-identifying if $C$ has the following properties:
\begin{enumerate}
\item $B_r(v) \cap C \neq \emptyset \text{ for all } v\in V(G) \text{ and}$
\item $B_r(u) \cap C \neq B_r(v)\cap C \text{ for all distinct } u,v\in V(G).$
\end{enumerate}
The elements of $C$ are called codewords.
We define $I_r(v)=I_r(v,C)=B_r(v)\cap C$.  We call
$I_r(v)$ the identifying set of $v$ with respect to $C$.  If $I_r(u)\neq I_r(v)$ for some $u\neq v$, we say $u$ and $v$ are distinguishable.  Otherwise, we say they are indistinguishable.

Vertex identifying codes were introduced in ~\cite{Karpovsky1998} as
a way to help with fault diagnosis in multiprocessor computer
systems.  Codes have been studied in many graphs.  Of particular interest are codes in the infinite triangular, square, and hexagonal lattices as well as the square lattice with diagonals (king grid). We can define each of these graphs so that they have vertex set $\Z\times\Z$.  Let $Q_m$ denote the set of vertices $(x,y)\subset\Z\times\Z$ with $|x|\le m$ and $|y|\le m$.  The density of a code $C$ defined in ~\cite{Charon2002} is $$D(C)=\limsup_{m\rightarrow\infty}\frac{|C\cap Q_m|}{|Q_m|}.$$

When examining a particular graph, we are interested in finding the minimum density of an $r$-identifying code.  The exact minimum density of an $r$-identifying code for the king grid was found in~\cite{Charon2004}.  General constructions of $r$-identifying codes for the square and triangular lattices are given in~\cite{Honkala2002} and~\cite{Charon2001}.

For this paper, we focus on the hexagonal grid.  It was shown in
~\cite{Charon2001} that $$\frac{2}{5r}-o(1/r)\le D(G_H,r) \le
\frac{8}{9r}+o(1/r),$$  where $D(G_H,r)$ represents the minimum density of
an $r$-identifying code in the hexagonal grid.
The main theorem of this paper is Theorem ~\ref{maintheorem}:

\begin{thm}\label{maintheorem} There exists an $r$-identifying code of
density $$\frac{5r+3}{6r(r+1)}  \text{ if $r$ is even} ;\qquad
            \frac{5r^2+10r-3}{(6r-2)(r+1)^2}  \text{ if $r$ is odd}.$$
\end{thm}

The proof of Theorem ~\ref{maintheorem} can be found in Section ~\ref{mainthmsection}.  Section ~\ref{descriptionsection} provides a brief description of a code with the aforementioned density and gives a few basic definitions needed to describe the code.  Section ~\ref{distanceclaimssection} provides a few technical lemmas needed for the proof of Theorem ~\ref{maintheorem} and the proofs of these lemmas can be found in Section ~\ref{section:lemmaproofs}.

\section{Construction and Definitions}\label{descriptionsection}

\begin{figure}[ht]
\centering
\includegraphics[totalheight=0.3\textheight]{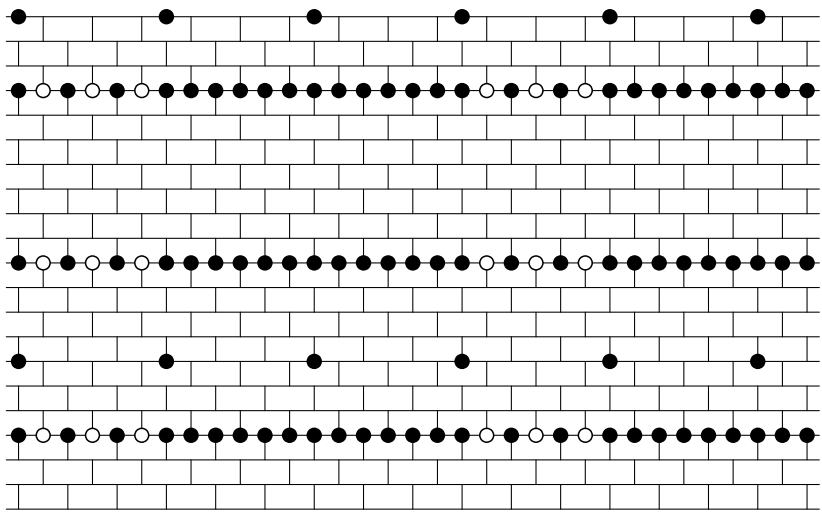}
\caption[An $r$-identifying code in the hexagonal grid for $r=6$.] {The code $C=C'\cup C''$ for $r=6$.  Black vertices are code
words. White vertices are vertices in $L_{n(r+1)}$ which are not in
$C$.}\label{fig:code6}
\end{figure}

\begin{figure}[ht]
\centering
\includegraphics[totalheight=0.3\textheight]{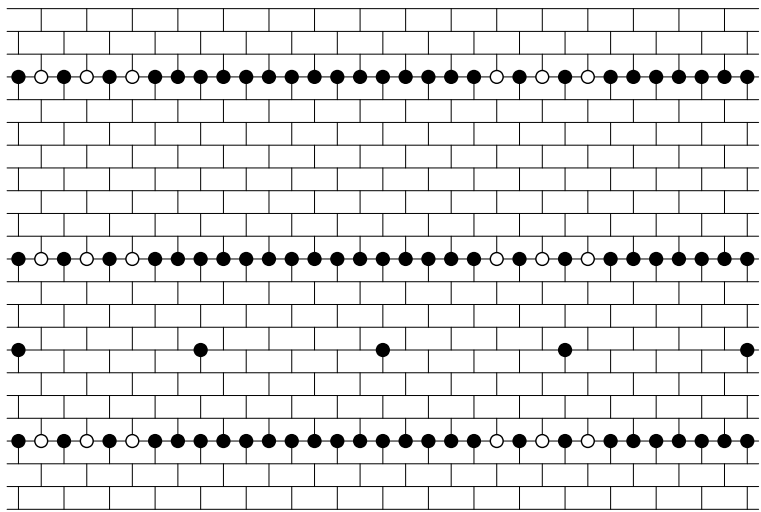}
\caption[An $r$-identifying code in the hexagonal grid for $r=7$.] {The code $C=C'\cup C''$ for $r=7$.  Black vertices are code
words. White vertices are vertices in $L_{n(r+1)}$ which are not in
$C$.}\label{fig:code7}
\end{figure}

For this construction we will use the brick wall representation of
the hex grid.  To describe this representation, we need to briefly consider the square grid $G_S$.
The square grid has vertex set $V(G_S)=\Z\times\Z$ and
$$E(G_S)=\{\{u=(i,j),v\}: u-v\in \{(0,
\pm 1),(\pm 1,0)\}\}.$$

Let $G_H$ represent the hex grid.  Then $V(G_H)=\Z\times\Z$ and
$$E(G_H)=\{\{u=(i,j),v\}: u-v\in \{(0,(-1)^{i+j+1}),(\pm 1,0)\}\}.$$

In other words, if $x+y$ is even, then $(x,y)$ is adjacent to
$(x,y+1),(x-1,y),$ and $(x+1,y)$.  If $x+y$ is odd, then $(x,y)$ is
adjacent to $(x,y-1),(x-1,y),$ and $(x+1,y)$. However, the first
representation shows clearly that the hex grid is a subgraph of the
square grid.

For any integer $k$, we also define a \textit{horizontal line} $L_k=\{(x,k): x\in \Z \}$.

Note that if $u,v\in V(G_S)$, then the distance between them (in
the square grid) is $\|u-v\|_1$.  From this point forward, let $d(u,v)$
represent the distance between two vertices in the hex grid.  If
$u\in V(G_H)$ and $U,V\subset(G_H)$, we define $d(u,V)=\min
\{d(u,v):v\in V\}$ and $d(U,V)=\min \{d(u,V): u\in U\}$.

Let $\delta=0$ if $k$ is even and $\delta=1$ otherwise. Define $$L_k'=\left\{\begin{array}{cc}
                       L_k\cap \{(x,k):x\not\equiv 1,3,5,\ldots,r-1\mod 3r\}& \text{if $r$ is even}; \\
                       L_k\cap \{(x,k):x\not\equiv 1,3,5,\ldots,r-1\mod 3r-1\}& \text{if $r$ is odd} \\
                     \end{array}\right.$$
and $$L_k''=\left\{\begin{array}{cc}
                       L_k\cap \{(x,k):x\equiv \delta \mod r\}& \text{if $r$ is even}; \\
                       L_k\cap \{(x,k):x\equiv 0\mod r+1\}& \text{if $r$ is odd} \\
                     \end{array}\right..$$

Finally, let
$$C'=\bigcup_{n=-\infty}^\infty L_{n(r+1)}' \qquad\text{and}\qquad
C''=\bigcup_{n=-\infty}^\infty L_{\lfloor (r+1)/2\rfloor+2n(r+1)}''.$$
Let $C=C'\cup C''$.  We will show in Section ~\ref{mainthmsection} that $C$ is a valid $r$-identifying code of the density described in Theorem ~\ref{maintheorem}.  Partial pictures of the code are shown for $r=6$ in Figure ~\ref{fig:code6} and $r=7$ in Figure ~\ref{fig:code7}.

\section{Distance Claims}\label{distanceclaimssection}

We present a list of lemmas on the distances of vertices in the hex
grid.  These lemmas will be used in the proof of our construction.
The proofs of these lemmas can be found in Section
\ref{section:lemmaproofs}.\vsten

\begin{lemma}\label{taxicab1} For $u,v\in V(G_H)$, $d(u,v)\ge
\|u-v\|_1$.\end{lemma}

\begin{proof}
Note that Lemma ~\ref{taxicab1} simply says that the distance
between any two vertices in our graph is no less than the
their distance in the square grid.  Since $G_H$ is a subgraph of
$G_S$, any path between $u$ and $v$ in $G_H$ is also a path in $G_S$
and so the lemma follows.
\end{proof}\vsten

\begin{lemma}\label{taxicab2} \textbf{(Taxicab Lemma)}
For $u=(x,y),v=(x'y')\in V(G_H)$, if $|x-x'|\ge|y-y'|$, then
$d(u,v)=\|u-v\|_1$.\end{lemma}

This fact is used so frequently that we call it the Taxicab
Lemma.  It states that if the horizontal distance between two
vertices is no less than the vertical distance, then the
distance between those two vertices is exactly the same as it would
be in the square grid.

The proof of the Taxicab Lemma and the remainder of these lemmas can be found in
Section ~\ref{section:lemmaproofs}.

Lemmas ~\ref{vertexlinedistance} and ~\ref{vertexlinedistance2} say that the distance
between a point $(k,a)$ and a line $L_b$ is either $2|a-b|$ or $2|a-b|-1$ depending on various factors.  It also follows from these lemmas that $d(L_a,L_b)=2|a-b|-1$ if $a\neq b$. \vsten

\begin{lemma}\label{vertexlinedistance}
Let $a<b$; then $$d((k,a),L_b)=\left\{\begin{array}{cc}
                                              2(b-a)-1, & \text{if $a+k$ is even}; \\
                                              2(b-a), & \text{if $a+k$ is odd}.
                                            \end{array}
\right.$$
\end{lemma} \vsten

\begin{lemma}\label{vertexlinedistance2} Let $a>b$; then $$d((k,a),L_b)=\left\{\begin{array}{cc}
                                              2(a-b) & \text{if $a+k$ is even}; \\
                                              2(a-b)-1 & \text{if $a+k$ is odd}.
                                            \end{array}
\right.$$
\end{lemma}

The next three lemmas all basically have the same idea.  If we are looking at a point $(x,y)$ and a horizontal line $L_k$ such that $(x,y)$ is within some given distance $d$, we can find a sequence $S$ of points on this line such that each point is at most distance $d$ from $(x,y)$ and the distance between each point in $S$ and its closest neighbor in $S$ is exactly 2. \vsten

\begin{lemma}\label{evenvertexdistance} Let $k$ be a positive integer.
There exist paths of length $2k$ from $(x,y)$ to $v$ for each $v$ in
$$\{(x-k+2j,y\pm k): j=0,1,\ldots,k\}.$$
\end{lemma} \vsten

\begin{lemma}\label{oddevenvertexdistance} Let $k$ be a positive integer and let $x+y$ be even.
There exist paths of length $2k+1$ from $(x,y)$ to $v$ for each $v$
in
$$\{(x-k+2j,y+k+1): j=0,1,\ldots,k\}\cup\{(x-k-1+2j,y-k): j=0,1,\ldots,k+1\}.$$
\end{lemma} \vsten

\begin{lemma}\label{oddoddvertexdistance} Let $k$ be a positive integer and let $x+y$ be odd.
There exist paths of length $2k+1$ from $(x,y)$ to $v$ for each $v$
in
$$\{(x-k+2j,y-k-1): j=0,1,\ldots,k\}\cup\{(x-k-1+2j,y+k): j=0,1,\ldots,k+1\}.$$
\end{lemma}

In the final lemma, we are simply stating that if we are given a point $(x,y)$ and a line $L_k$ such that $d((x,y),L_k)<r$, we can find a path of vertices on that line which are all distance at most $r$ from $(x,y)$. \vsten

\begin{lemma}\label{vertexballdistance} Let $(x,y)$ be a vertex and
$L_k$ be a line.  If $d((x,y),L_k)< r$, then
$$\{(x-(r-|y-k|)+j,k): j=0,1,\ldots, 2(r-|y-k|)\}\subset B_r((x,y)).$$
\end{lemma}

\section{Proof of Main Theorem}\label{mainthmsection}Here, we wish to show that the set described in Section~\ref{descriptionsection} is indeed a valid $r$-identifying code. \vsten

\begin{proofcite}{Theorem~\ref{maintheorem}}

Let $C=C'\cup C''$ be the code described in Section ~\ref{descriptionsection}.
 Let $d(C)$ be the density of $C$ in $G_H$,
$d(C')$ be the density of $C'$ in $G_H$, and $d(C'')$ be the density
of $C''$ in $G_H$.  Since $C'$ and $C''$ are disjoint, we see that

\begin{eqnarray*}
  d(C) &=& \limsup_{m\rightarrow\infty}\frac{|C\cap G_m|}{|G_m|} \\
  &=& \limsup_{m\rightarrow\infty}\frac{|(C'\cup C'')\cap
  G_m|}{|G_m|} \\
  &=& \limsup_{m\rightarrow\infty}\frac{|C'\cap
  G_m|}{|G_m|}+\limsup_{m\rightarrow\infty}\frac{|C''\cap
  G_m|}{|G_m|} \\
  &=& d(C')+d(C'').
\end{eqnarray*}

It is easy to see that both $C'$ and $C''$ are periodic tilings of
the plane (and hence $C$ is also periodic).  The density of periodic
tilings was discussed By Charon, Hudry, and
Lobstein~\cite{Charon2002} and it is shown that the density of a
periodic tiling on the hex grid is
$$D(C)=\frac{\text{\# of codewords in tile}}{\text{size of tile}} .$$

For $r$ even, we may consider the density of $C'$ on the tile
$[0,3r-1]\times[0,r]$. The size of this tile is $3r(r+1)$. On this
tile, the only members of $C'$ fall on the horizontal line $L_0$,
of which there are $2r+r/2$.

For $r$ odd, we may consider the density of $C'$ on the tile
$[0,3r-2]\times[0,r]$. The size of this tile is $(3r-1)(r+1)$.  On
this tile, the only members of $C'$ fall on $L_0$, of which there
are $2r+(r-1)/2$. Thus

$$d(C')=\left\{\begin{array}{cc}
                \frac{5}{6(r+1)} & \text{if $r$ is even;} \\
                \frac{5r-1}{(6r-2)(r+1)} & \text{if $r$ is odd.}
              \end{array}
\right.$$

For $C''$, we need to consider the tiling on $[0,r-1]\times[0,2r+1]$
if $r$ is even and $[0,r]\times[0,2r+1]$ if $r$ is odd.  In either
case, the tile contains only a single member of $C''$.  Hence,
$$d(C'')=\left\{\begin{array}{cc}
                \frac{1}{2r(r+1)} & \text{if $r$ is even;} \\
                \frac{1}{2(r+1)^2} & \text{if $r$ is odd.}
              \end{array}
\right.$$  Adding these two densities together gives us the numbers
described in the theorem.\vsten

It remains to show that $C$ is a valid code.  We will say that a
vertex $v$ is \emph{nearby} $L_k'$ if $I_r(v)\cap L_k'\neq
\emptyset$. An outline of the proof is as follows:\vsten
\begin{enumerate}
  \item Each vertex $v\in V(G_H)$ is nearby $L_{n(r+1)}'$ for
  exactly one value of $n$ (and hence, $I_r(v)\neq\emptyset$).
  \item Since each vertex is nearby $L_{n(r+1)}'$ for exactly one value
  of $n$, we only need to distinguish between vertices which are
  nearby the same horizontal line.
  \item The vertices in $C''$ distinguish between vertices that fall
  above the horizontal line and those that fall below the line.
  \item We then show that $L_{n(r+1)}'$ can distinguish between any
  two vertices which fall on the same side of the horizontal line (or in the
  line).
\end{enumerate}\vsten

\textbf{Part (1)}: We want to show that each vertex is nearby $L_{n(r+1)}$ for
exactly one value of $n$.

From Lemma \ref{vertexlinedistance}, it immediately follows that
$d(L_{n(r+1)},L_{(n+1)(r+1)})=2r+1$.  So by the triangle inequality,
no vertex can be within a distance $r$ of both of these lines. Thus,
no vertex can be nearby more than one horizontal line of the form
$L_{n(r+1)}$.

\begin{lemma}\label{codewordclaim}
Let $u,v\in L_{n(r+1)}$. Then

\begin{enumerate}
  \item If $u\sim v$ then at least one of $u$ and $v$ is in $C'$.
  \item If $r\le d(u,v)\le 2r+1$, then at least one of $u$ and $v$
  is in $C'$.
\end{enumerate}
  Furthermore, note that for any $x$ at least one of the following
vertices is in $C$:
$$\{(x+2k,n(r+1)): k=0,1,\ldots,\lceil(r+1)/2\rceil\} .$$
\end{lemma}

The proof of this lemma is in Section ~\ref{section:lemmaproofs}.

From this point forward, let $v=(i,j)$ with $n(r+1)\le j <
(n+1)(r+1)$.

First, suppose that $r$ is even.  Then for $j\le r/2+ n(r+1)$ we have
$d((i,j) ,(i-r/2,j-r/2))=r$ by the Taxicab Lemma. Since $j-r/2\le
n(r+1)$, there is $u\in L_{n(r+1)}$ such that $d((i,j),u)\le r$.  If
$d((i,j),u)<r$, then $u$ and $u+(1,0)$ are both within distance $r$
of $(i,j)$ and by Lemma ~\ref{codewordclaim} one of them is in $C'$.  If
$d((i,j),L_{n(r+1)})=r$, then by Lemma \ref{evenvertexdistance}
$(i,j)$ is within distance $r$ of one of the vertices in the set
$\{(i-r/2+2k: k=0,1,\ldots,r/2\}$ and again by Lemma ~\ref{codewordclaim} one of them must be in $C'$.
By a symmetric argument, we can show that if $j>r/2$, then there is
a $u\in C'\cap L_{(n+1)(r+1)}$ such that $d((i,j),u)\le r$.

If $r$ is odd, and $j\neq (r+1)/2$, then we can use the same
argument to show there is a codeword in $C'\cap L_{(n+1)(r+1)}$ or
$C'\cap L_{(n+1)(r+1)}$ within distance $r$ of $(i,j)$.  If
$j=(r+1)/2$, then we refer to Lemmas \ref{oddevenvertexdistance},
\ref{oddoddvertexdistance}, and \ref{codewordclaim} to find an appropriate codeword.  Since
no vertex can be nearby more than one horizontal line of the form
$L_{n(r+1)}$, this shows that each vertex is nearby exactly one
horizontal line of this form.

So we have shown that $v$ is nearby exactly one horizontal line of
the form $L_{n(r+1)}'$, but this also shows that  $I_r(v)\neq
\emptyset$ for any $v\in V(G_H)$.

\textbf{Part (2)}: We now turn our attention to showing
that $I_r(u)\neq I_r(v)$ for any $u\neq v$. We first note that if $u$ is nearby $L_{n(r+1)}'$, $v$ is nearby
$L_{m(r+1)}'$, and $m\neq n$, then there is some c in $I_r(v)\cap
L_{m(r+1)}'$ but $c\not\in I_r(u)$ since $u$ is not nearby
$L_{m(r+1)}'$.  Thus, $u$ and $v$ are distinguishable.
Hence it suffices to consider only the case that
$n=m$.

\textbf{Part (3)}:  We consider the case that $u$ and $v$ fall on opposite sides
of $L_{n(r+1)}$.  Suppose that $u=(i,j)$, where $j>n(r+1)$ and
$v=(i',j')$, where $j'<n(r+1)$. We see that if $n=2k$, then
$n(r+1)<\lfloor(r+1)/2\rfloor+2k(r+1)< (n+1)(r+1)$ and if $n=2k+1$,
then $(n-1)(r+1)<\lfloor(r+1)/2\rfloor+2k(r+1)< n(r+1)$.  In the
first case, we see that
\begin{eqnarray*}
    d((i',j'),L_{\lfloor(r+1)/2\rfloor+2k(r+1)}) &\ge& d(L_{2k(r+1)-1},L_{\lfloor(r+1)/2\rfloor+2k(r+1)})\\
    &=&  2\lfloor (r+1)/2\rfloor + 1\\
    &\ge& r+1
\end{eqnarray*}
and likewise in the second case we have that $d((i,j),
L_{\lfloor(r+1)/2\rfloor+2k(r+1)})\ge r+1$ and so it follows that at
most one of $(i,j)$ and $(i',j')$ is within distance $r$ of a codeword in $C''$.

If $r$ is odd, then $\lfloor(r+1)/2\rfloor = (r+1)/2$ and
$|j-(r+1)/2+2k(r+1)|\le(r-1)/2$.  Hence, we see that
$d((i,j),L_{(r+1)/2+2k(r+1)})\le r-1$.  From Lemma
\ref{vertexballdistance}, it follows that $$\{(m,(r+1)/2+n(r+1)):
(r-1)/2\le m \le (3r-1)/2\}\subset B_r((i,j)) .$$  We note that the
cardinality of this set is at least $r+1$ and so in the case that
$n$ is even we see that at least one of these must be in $C''$.  A
symmetric argument applies if $n$ is odd, showing that $(i',j')$ is
within distance $r$ of a codeword in $C''$.

If $r$ is even, we apply a similar argument to show that $(i,j)$ is
within distance $r-1$ of $r$ vertices in
$L_{n(r+1)+\lfloor(r+1)/2\rfloor}$ and $(i',j')$ is within distance
$r$ of $r-1$ vertices of $L_{n(r+1)-\lceil(r+1)/2\rceil}$.  If $n$
is even, then clearly $(i,j)$ is within distance $r$ of some codeword in $C''$. However, if $n$ is odd, we have only shown that
$$\{(m,\lceil(r+1)\rceil/2+n(r+1)): r/2-2\le m \le 3r/2-1\}\subset
B_{r-1}((i,j)) .$$  Consider the set
$$\{(m,\lfloor(r+1)\rfloor/2+n(r+1)): r/2-2\le m \le 3r/2-1\}.$$  This set has $r$ vertices and so one of them must
be a codeword.  Furthermore, by the way we constructed $C''$, the
sum of the coordinates of codeword are even and so there is an edge
connecting it to a vertex in $$\{(m,\lceil(r+1)\rceil/2+n(r+1)):
r/2-2\le m \le 3r/2-1\}$$ and so it is within radius $r$ of
$(i',j')$.  In either case, we have exactly one of $u$ and $v$ is
within distance $r$ of a codeword in $C''$.

\textbf{Part (4)}: Now we have shown that two vertices are distinguishable if they are
nearby two different lines in $C'$ or if they fall on opposite sides
of a line $L_{n(r+1)}$ for some $n$.  To finish our proof, we need
to show that we can distinguish between $u=(i,j)$ and $v=(i',j')$ if
$u$ and $v$ are nearby the same line $L_{n(r+1)}'$ and $u$ and $v$
fall on the same side of that line.  Without loss of generality,
assume that $j,j'\ge n(r+1)$.

\textbf{Case 1}: $u,v\in L_{n(r+1)}$.

Without loss of generality, we can write $u=(x,n(r+1))$ and
$v=(x+k,n(r+1))$ for $k>0$.  If $k>1$, then $(x-r,n(r+1))$ and
$(x-r+1,n(r+1))$ are both within distance $r$ of $u$ but not of $v$
and one of them must be a codeword by Lemma \ref{codewordclaim}.  If $k=1$, then $(x-r,n(r+1))$
is within distance $r$ of $u$ but not $v$ and $(x+r+1,n(r+1))$ is
within distance $r$ of $v$ but not $u$. Since
$d((x-r,n(r+1)),(x+r+1,n(r+1)))=2r+1$, Lemma \ref{codewordclaim} states that at least one of them must be a
codeword.

\textbf{Case 2}: $u\in L_{n(r+1)}$, $v\not\in L_{n(r+1)}$.

Without loss of generality, assume that $i\le i'$.  We have
$u=(i,n(r+1))$.  If $i=i'$, then $v=(i,n(r+1)+k)$ for some $k>0$.
Consider the vertices $(i-r,n(r+1))$ and $(i+r,n(r+1))$.  These are
each distance $r$ from $u$ and they are distance $2r$ from each
other, so by Lemma \ref{codewordclaim} at least one of them is a codeword.  However, by Lemma
\ref{taxicab1}, these are distance at least $r+k>r$ from $v$ and so
$I_r(u)\neq I_r(v)$.

If $i<i'$, then we can write $v=(i+j, n(r+1)+k)$ for $j,k>0$.
Consider the vertices $x_1=(i-r, n(r+1))$ and $x_2=(i-r+1,n(r+1))$.
By the Taxicab Lemma, $d(u,x_1)=r$ and $d(u,x_2)=r-1$.  Further,
since they are adjacent vertices in $L_{n(r+1)}$, one of them is in
$C'$ by Lemma \ref{codewordclaim}.  However, by Lemma \ref{taxicab1} we have $d(v,x_1)\ge
r+j+k>r$ and $d(v,x_2)\ge r-1+j+k>r$.  Hence, neither of them is in
$I_r(v)$ and so $I_r(u)\neq I_r(v)$.

\textbf{Case 3}: $u,v\not\in L_{n(r+1)}$.

Assume without loss of generality that $n=0$ and so $1\le j,j'\le
\lceil r/2\rceil$.

\textbf{Subcase 3.1}: $j<j'$, $i=i'$.

From Lemma \ref{vertexlinedistance2} it immediately follows that
$d((i,j),L_0)<d((i',j'),L_0)\le r$.  By the Taxicab Lemma, we have
$d((i-(r-j),0),(i,j))=d((i+(r-j),0),(i,j))=r$. Further note that
$d((i-(r-j),(i+(r-j),0))=2(r-j)$ and since $1\le j<\lceil r/2\rceil$
we have $r+1\le 2(r-j)\le 2r-2$ and so at least one of these two vertices is
a codeword by Lemma \ref{codewordclaim}.  By the Taxicab Lemma, neither of these two vertices is
in $B_r((i',j'))$.

\textbf{Subcase 3.2}: $j<j'$, $i\neq i'$.

Without loss of generality, we may assume that $i<i'$. Then, we wish
to consider the vertices $(i-(r-j),0)$ and $(i-(r-j)+1,0)$.  We see
that $d((i,j),(i-(r-j)+1,0))=r-1$ and so by the Taxicab Lemma,
neither of these vertices is in $B_r((i',j'))$.  But since these
vertices are adjacent and both in $L_0$ at least one of them is a
codeword by Lemma \ref{codewordclaim} and so that one is in $I_r((i,j))$ but not in
$I_r((i',j'))$.

\textbf{Subcase 3.3}: $j=j'$, $d(u,L_0)<r$ or $d(v,L_0)<r$.

Without loss of generality, assume that $d((i,j),L_0)<r$ and that
$i<i'$.  Then $v=(i+k,j)$ for some $k>0$.

First suppose that $k>1$. Then, we wish to consider the vertices
$(i-(r-j),0)$ and $(i-(r-j)+1,0)$. We see that
$d((i,j),(i-(r-j)+1,0))=r-1$ and so by the Taxicab Lemma, neither of
these vertices is in $B_r((i',j'))$. But since these vertices are
adjacent and both in $L_0$, by Lemma \ref{codewordclaim} at least one of them is a codeword and so
that one is in $I_r((i,j))$ but not in $I_r(i',j'))$.

If $k=1$, consider the vertices $x_1=(i-(r-j),0)$ and
$x_2=(i+1+(r-j),0)$.  We see that
$$\begin{array}{ll}
  d(x_1,u) &= r, \\
  d(x_1,v) &= r+1, \\
  d(x_2,u) &= r+1, \\
  d(x_2,v) &= r
\end{array}$$
all by the taxicab lemma.  Furthermore, $d(x_1,x_2)=2(r-j)+1$ and so
$r+1\le d(x_1,x_2)\le 2r+1$ as in the above case and so one of them is
a code word by Lemma \ref{codewordclaim}.  Hence $I_r(u)\neq I_r(v)$.

\textbf{Subcase 3.4}: $j=j'$, $d(u,L_0)=d(v,L_0)=r$.

Without loss of generality, suppose $i'<i$.  We wish to consider the vertices
in the two sets
$$U=\{(i-(r-j)+2k,0): k=0,1,2,\ldots,r-j\}$$ and
$$V=\{(i'-(r-j)+2k,0):k=0,1,2,\ldots,r-j\}.$$  Note that
$$|U|=|V|=\left\{\begin{array}{cc}
                         r/2+1 & \text{if $r$ is even}, \\
                         (r+1)/2 & \text{if $r$ is odd}.
                       \end{array}
\right.$$

From Lemmas \ref{evenvertexdistance},
\ref{oddevenvertexdistance}, and \ref{oddoddvertexdistance} we have
$U\subset B_r(u)$ and $V\subset B_r(v)$. Each of $U$ and $V$ contain a codeword by Lemma \ref{codewordclaim}.  Hence, if $U\cap
V=\emptyset$  then  $u$ and
$v$ are distinguishable.



If $U\cap V\neq \emptyset$ then the leftmost vertex in $U$ is also in $V$.  We further note that the leftmost vertex in $U$ is not the leftmost vertex in $V$ or else $U=V$ and $u=v$.  Thus, we must have $i'-(r-j)+2\le i\le i'+(r-j)$ and so $2\le i-i'\le 2(r-j)$.  Let
$x_1=(i+(r-j),0)$ and let $x_2=(i'-(r-j),0)$.  By definition $x_1\in
U$ and $x_2\in V$.  By the Taxicab Lemma we have $d(x_1,v)=|i+r-j-i'|+|j|=r+i-i'>r$ and likewise $d(x_2,u)>r$ so $x_1\not\in I_r(v)$ and $x_2\not\in I_r(u)$.  Also we have $d(x_1,x_2)=2(r-j)+i-i'$ which gives $2(r-j)+2\le d(x_1,x_2)\le 4(r-j)$.  Since $d(u,L_0)=r$ we must have $r-j=r/2$ if $r$ is even and $r-j=(r-1)/2$ if $r$ is odd by Lemmas ~\ref{vertexlinedistance} and ~\ref{vertexlinedistance2}.  In either case this gives
$r+2\le d(x_1,x_2)\le 2r$. By Lemma \ref{codewordclaim}, one of $x_1$ and $x_2$ is a
codeword and so $u$ and $v$ are distinguishable.

This completes the proof of Theorem \ref{maintheorem}.
\end{proofcite}

\section{Proof of Lemmas}\label{section:lemmaproofs}

We now complete the proof by going back and finishing off the proofs of the lemmas that were presented in Section \ref{distanceclaimssection}.\vsten

\begin{proofcite}{Lemma~\ref{taxicab2}}\textbf{(Taxicab Lemma)}
By Lemma ~\ref{taxicab1}, it suffices to show there exists a path of length $\|u-v\|$ between $u$ and $v$.  Since $|x-x'|\ge|y-y'|$, either $u=v$ or $x\neq x'$. If
$u=v$, then the lemma is trivial, so assume without loss of
generality that $x< x'$.  We will first assume that $y\le y'$.

We note that no matter what the parity of $i+j$, there is always a
path of length 2 from $(i,j)$ to $(i+1,j+1)$.  Then for $1\le i \le
|y-y'|$, define $P_i$ to be the path of length 2 from
$(x+i-1,y+i-1)$ to $(x+i,y+i)$.  Then $\bigcup_{i=1}^{|y-y'|} P_i$
is a path of length $2(y'-y)$ from $(x,y)$ to $(x+y'-y, y')$.  Since $(i,j)\sim (i+1,j)$ for all $i,j$, there is a path  $P'$ of
length $x'-x+y-y'$ (Note: this number is nonnegative since
$x'-x\ge y'-y$.) from $(x+y'-y, y')$ to $(x',y')$ by simply moving
from $(i,y')$ to $(i+1,y')$ for $x+y'-y\le i< x'$.  Then, we
calculate the total length of $$P'\cup \bigcup_{i=1}^{|y-y'|}P_i$$
to be $|x'-x|+|y'-y|$.  By Lemma \ref{taxicab1}, it follows that
$d(u,v)= |x-x'|+|y-y'|=\|u-v\|_1$.

If $y>y'$, then a symmetric argument follows by simply following a
downward diagonal followed by a straight path to the right.
\end{proofcite}\vsten

\begin{proofcite}{Lemma~\ref{vertexlinedistance}}
We proceed by induction on $b-a$.

The base case is trivial since if $a+k$ is even, then
$(k,a)\sim(k,a+1)=(k,b)$.  If $a+k$ is odd, then there is no edge
directly to $L_b$ and so any path from $(k,a)$ to $L_b$ has length
at least 2.  This is attainable by the path from $(k,a)$ to $(k+1,a)$ to $(k+1,b)$.

The inductive step follows similarly.  Let $a+k$ be even.  Then
there is an edge between $(k,a)$ and $(k,a+1)$.  Since $k+a+1$ is
odd, the shortest path from $(k,a+1)$ to $L_b$ has length $2(b-a-1)$
and so the union of that path with the edge $\{(k,a),(k,a+1)\}$
gives a path of length $2(b-a)-1$.  Likewise, if $k+a$ is odd, then
we take the path from $(k,a)$ to $(k+a,a)$ to $(k+1,a+1)$ in union with a path from
$(k+1,a+1)$ to $L_b$ gives us a path of length $2(b-a)$.  Further, any path from $L_a$ to $L_b$ must contain at least one
point in $L_{a+1}$.  So if we take a path of length $\ell$ from
$(k,a)$ to $(j,a+1)$ (note that $j+a+1$ must be odd), then we get a
path of length $\ell+2(b-a-1)$.  From the base case, we have chosen
our paths from $L_a$ to $L_{a+1}$ optimally and so this path is
minimal.
\end{proofcite}\vsten

\begin{proofcite}{Lemma~\ref{vertexlinedistance2}}
The proof is symmetric to the previous proof.
\end{proofcite}\vsten

\begin{proofcite}{Lemma~\ref{evenvertexdistance}}
  The proof is by induction on $k$.  If $k=1$, then as noted before,
  there is always a path of length 2 from $(x,y)$ to $(x\pm 1, y\pm
  1)$.

  For $k>1$, there is a path of length 2 from $(x,y)$ to $(x-1,y+1)$
  and to $(x+1,y+1)$.  By the inductive hypothesis, there is a path of
  length $2(k-1)$ from $(x-1, y+1)$ to each vertex in
  $$S=\{(x-k+2j,y+k):j=0,1,\ldots,k-1\}$$ and a path of
  length $2(k-1)$ from $(x+1, y+1)$ to $(x+k,y+k)$.  Taking the
  union of the path of length 2 from $(x,y)$ to $(x-1,y+1)$ and the
  path of length $2(k-1)$ from $(x-1,y+1)$ to each vertex in $S$
  gives us a path of length $2k$ to each vertex in $$\{(x-k+2j,y+ k): j=0,1,\ldots,k-1\} .$$
  Then, taking the path of length 2 from $(x,y)$ to $(x+1,y+1)$ and
  the path of length $2(k-1)$ from $(x+1,y+1)$ to $(x+k,y+k)$ gives
  us a path of length $2k$ from $(x,y)$ to each vertex in $$\{(x-k+2j,y+ k): j=0,1,\ldots,k\}.$$

  By a symmetric argument, we can find paths of length $2k$ from
  $(x,y)$ to each vertex in $$\{(x-k+2j,y-k): j=0,1,\ldots,k\} .$$
\end{proofcite}\vsten

\begin{proofcite}{Lemma~\ref{oddevenvertexdistance}}
Since $x+y$ is even, $(x,y)\sim(x,y+1)$.  By Lemma ~\ref{evenvertexdistance}, there are paths of length $2k$ from $(x,y+1)$ to each vertex in $\{(x-k+2j,y+k+1): j=0,1,\ldots,k\}$ and hence there are paths of length $2k+1$ from $(x,y)$ to each vertex in that set.

Since $(x,y)\sim(x-1,y)$, by Lemma ~\ref{evenvertexdistance}, there
are paths of length $2k$ from $(x-1,y)$ to each vertex in
$\{(x-k-1+2j,y-k): j=0,1,\ldots,k\}$.  Similarly, since
$(x,y)\sim(x+1,y)$ there is a path of length $2k$ from $(x+1,y)$ to
$(x+k+1,y-k)$ and so there is a path of length $2k+1$ from $(x,y)$
to each vertex in $\{(x-k-1+2j,y-k): j=0,1,\ldots,k+1\}$.
\end{proofcite}\vsten

\begin{proofcite}{Lemma~\ref{oddoddvertexdistance}}
Since $x+y$ is odd, $(x,y)\sim(x,y-1)$.  By Lemma
~\ref{evenvertexdistance}, there are paths of length $2k$ from
$(x,y-1)$ to each vertex in $\{(x-k+2j,y-k-1): j=0,1,\ldots,k\}$ and
hence there are paths of length $2k+1$ from $(x,y)$ to each vertex
in that set.

Since $(x,y)\sim(x-1,y)$, by Lemma ~\ref{evenvertexdistance}, there
are paths of length $2k$ from $(x-1,y)$ to each vertex in
$\{(x-k-1+2j,y+k): j=0,1,\ldots,k\}$.  Similarly, since
$(x,y)\sim(x+1,y)$ there is a path of length $2k$ from $(x+1,y)$ to
$(x+k+1,y+k)$ and so there is a path of length $2k+1$ from $(x,y)$
to each vertex in $\{(x-k-1+2j,y+k): j=0,1,\ldots,k+1\}$.
\end{proofcite}\vsten

\begin{proofcite}{Lemma~\ref{vertexballdistance}}
Without loss of generality, assume that $y\ge k$.  If $y=k$, the
Lemma is trivial, so let $\ell>0$ be the length of the
shortest path between $(x,y)$ and $L_k$.  By assumption, $\ell\le
r-1$.

First assume that $\ell$ is even.  By Lemma
~\ref{evenvertexdistance}, there is a path of length $\ell$ from
$(x,y)$ to each vertex in
$S=\{((x-\ell/2+2j,k):j=0,1,\ldots,\ell/2\}$.  Note that the
vertices in the set
$S'=\{((x-\ell/2+2j-1,k):j=0,1,\ldots,\ell/2+1\}$ all fall within
distance 1 of a vertex in $S$ and so $S\cup S' =
\{((x-\ell/2-1+j,k):j=0,1,\ldots,\ell+2\}\subset B_r((x,y))$.

Now note that $d((x,y),(x-\ell/2,k))=\ell$, so if $x-(r-|y-k|)\le
x_0\le x-\ell/2$ for some $x_0$, then
\begin{eqnarray*}
d((x_0,k),(x,y))&\le& d((x_0,k),(x-\ell/2,k))+d((x-\ell/2,k),(x,y))
\\
&\le& (r-|y-k|-\ell/2)+\ell .
\end{eqnarray*}
Since $|y-k|=\ell/2$ this gives $d((x_0,k),(x,y))\le r$.  Likewise,
if $x+\ell/2\le x_0\le x+(r-|y-k|)$ for some $x_0$, then
$d((x_0,k),(x,y))\le r$.  Hence, $d((x_0,k),(x,y))\le r$ for all
$x-(r-|y-k|)\le x_0\le x+(r-|y-k|)$ and so the Lemma follows.

If $\ell$ is odd, the Lemma follows by using either Lemma
~\ref{oddevenvertexdistance} or Lemma ~\ref{oddoddvertexdistance}
and applying a similar argument.

\end{proofcite}\vsten

\begin{proofcite}{Lemma~\ref{codewordclaim}}
This first part of this lemma is immediate from the definition of $C'$ since the only vertices in $L_{n(r+1)}$ which are not in $L_{n(r+1)}'$ are separated by distance 2.  To see the second part, let $\ell=3r$ if $r$ is even and $\ell=3r-1$ if $r$ is odd.  Write $L_{n(r+1)}'=S_1\cup S_2\cup S_3$, where
\begin{eqnarray*}
  S_1 &=& L_{n(r+1)}'\cap\{(x,y), x\equiv 1,\ldots, r-1 \mod \ell\},\\
  S_2 &=& L_{n(r+1)}'\cap\{(x,y), x\equiv r,r+1,\ldots, 2r-1 \mod \ell\},\\
  S_3 &=& L_{n(r+1)}'\cap\{(x,y), x\equiv 2r,2r+1,\ldots, \ell \mod \ell\}.\\
\end{eqnarray*}
Note that each vertex in $S_2\cup S_3$ is in $C'$ and so if $u,v\in L_{n(r+1)}$ are both not in $C'$, they are both in $S_1$.  However, it is clear that for two vertices in $S_1$, that their distance is either strictly less than $r$ or at least than $2r+1$.

Finally, for the last part of the lemma, we use the same partition of $L_{n(r+1)}$ so that all noncodewords fall in $S_1$.  However, the distance between $(x,n(r+1))$ and $(x+2\lceil(r+1)/2\rceil)$ is at least $r+1$ so one of those two vertices cannot fall in $S_1$ and so it is a codeword.
\end{proofcite}

\section{Conclusion}

In an upcoming paper, we also provide improved lower bounds for $D(G_H,r)$ when $r=2$ or $r=3$~\cite{Martin2010}.  Below are a couple tables noting our improvements.  This includes
the results not only from our paper, but also from the aforementioned paper.

$$\begin{array}{c}

\begin{array}{|c|c||c|c|}
  \hline
  \multicolumn{4}{|c|}{\text{Hex Grid}}\\
  \hline
  r & \text{previous lower bounds} & \text{new lower bounds} & \text{upper bounds} \\
  \hline
  2 & 2/11\approx 0.1818 ^\text{~\cite{Karpovsky1998}} & 1/5 = 0.2^\text{~\cite{Martin2010}} & 4/19\approx 0.2105 ^\text{~\cite{Charon2002}}\\
  \hline
  3 & 2/17\approx 0.1176 ^\text{~\cite{Charon2001}} & 3/25 = 0.12 & 1/6\approx 0.1667 ^\text{~\cite{Charon2002}}\\
  \hline
  \multicolumn{4}{|c|}{\text{Square Grid}}\\
  \hline
  2 & 3/20=0.15 ^\text{~\cite{Charon2001}} & 6/37\approx 0.1622^\text{~\cite{Martin2010}} & 5/29\approx 0.1724 ^\text{~\cite{Honkala2002}} \\
  \hline
\end{array}\\
\vspace{1em}\\

\begin{array}{|c|c|c||c|}
  \hline
  \multicolumn{4}{|c|}{\text{Hex Grid}}\\
  \hline
  r & \text{lower bounds} & \text{new upper bounds} & \text{previous upper bounds} \\
  \hline
  15 & 2/77\approx 0.0260 ^\text{~\cite{Charon2001}} & 1227/22528\approx  0.0523 & 1/18\approx 0.0556 ^\text{~\cite{Charon2002}} \\
  \hline
  16 & 2/83\approx 0.0241 ^\text{~\cite{Charon2001}} & 83/1632\approx  0.0509    & 1/18\approx 0.0556^\text{~\cite{Charon2002}}\\
  \hline
  17 & 2/87\approx 0.0230 ^\text{~\cite{Charon2001}} &                           & 1/22\approx 0.0455^\text{~\cite{Charon2002}}\\
  \hline
  18 & 2/93\approx 0.0215 ^\text{~\cite{Charon2001}} & 31/684 \approx  0.0453    & 1/22\approx 0.0455 ^\text{~\cite{Charon2002}}\\
  \hline
  19 & 2/97\approx 0.0206 ^\text{~\cite{Charon2001}} & 387/8960\approx  0.0432    & 1/22\approx 0.0455 ^\text{~\cite{Charon2002}}\\
  \hline
  20 & 2/103\approx 0.0194 ^\text{~\cite{Charon2001}} & 103/2520\approx  0.0409    & 1/22\approx 0.0455 ^\text{~\cite{Charon2002}}\\
  \hline
  21 & 2/107\approx 0.0187 ^\text{~\cite{Charon2001}} &                            & 1/26\approx 0.0385 ^\text{~\cite{Charon2002}}\\
  \hline
\end{array}
\end{array}$$

For even $r\ge 22$, we have improved the bound from approximately $8/(9r)$~\cite{Charon2001} to $(5r+3)/(6r(r+1))$ and for odd $r\ge 23$ we have improved the bound from approximately $8/(9r)$~\cite{Charon2001} to $(5r^2+7r-3)/((6r-2)(r+1)^2)$.

Constructions for $2\le r\le 30$ given in
~\cite{Charon2002} were previously best known.  However, the best general
constructions were given in ~\cite{Charon2001}.\\~\\

\section*{Acknowledgement}
\addcontentsline{toc}{section}{Acknowledgment}The author would like to thank anonymous referees whose comments improved the manuscript.

\chapter{VERTEX IDENTIFYING CODES FOR THE $n$-DIMENSIONAL LATTICE}\label{chapter:latticepaper}

\begin{center} Based on a paper submitted to \emph{Discrete Mathematics}\bigskip \\ Brendon Stanton\end{center}

\section*{Abstract}
\addcontentsline{toc}{section}{Abstract}
An $r$-identifying code on a graph $G$ is a set $C\subset V(G)$ such that for every vertex in $V(G)$, the intersection of the radius-$r$ closed neighborhood with $C$ is nonempty and different. If $n=2^k-1$ for some integer $k$, it is easy to show that an optimal 1-identifying code can be constructed for the $n$-dimensional lattice through a construction using a Hamming code.  Here, we show how to construct sparse codes in the case that $n\neq 2^k-1$ by giving two different types of constructions.  In addition, we provide a better bound for the 4-dimensional case through a connection with the king grid.  We also prove a monotonicity theorem for $r$-identifying codes on the $n$-dimensional lattice and show that for fixed $n$, the minimum density of an $r$-identifying code is $\Theta(1/r^{n-1})$.

\section{Introduction}

Vertex identifying codes were introduced by Karpovsky, Chakrabarty, and Levitin in~\cite{Karpovsky1998} as a way to help with fault diagnosis in multiprocessor computer systems.  That paper also addresses the issue of finding codes on a wide variety of graphs including the square, hexagonal and triangular lattices and the $n$-dimensional hypercube.  In addition, it briefly explores the idea of codes on the $n$-dimensional hypercube with $p^n$ vertices.  Taking the limit as $p\rightarrow \infty$ we then get the $n$-dimensional lattice.  An interesting result is that if $n=2^k-1$ for some integer $k$, we can find a code of optimal density through the use of a Hamming code.  Denote by $\mathcal{D}(G,r)$ the minimum possible density of an $r$-identifying code for a graph $G$.  Let $L_n$ denote the $n$-dimensional lattice.

\begin{thm}[\cite{Karpovsky1998}]\label{thm:Hamming}
  If $n=2^k-1$ for some integer $k$, then
  $$\mathcal{D}(L_n,1)=\frac{1}{n+1}.$$
\end{thm}

In addition to addressing these optimal codes, they also show that $1/(2^k+1)\le \mathcal{D}(L_{2^k},1)\le 1/2^k$.  In this paper, the first thing we wish to do is expand on this and provide reasonable upper bounds for $\mathcal{D}(L_n,1)$ in the case that $n\neq 2^k,2^k-1$.  If $n$ is small, we are able to use the idea of a dominating set on the $n$-dimensional hypercube to find good upper bounds for the size of our code as in Figure~\ref{fig:boundtable}.

\begin{figure}[ht]
$$\begin{array}{|ccccc|}
\hline
    && \mathcal{D}(L_1,1)&=&1/2  \\
    && \mathcal{D}(L_2,1)&=&7/20^\text{~\cite{Ben-Haim2005}}  \\
    \hline
    && \mathcal{D}(L_3,1)&=&1/4    \\
  1/5&\le& \mathcal{D}(L_4,1) &\le& 2/9^{\text{[Theorem~\ref{thm:4Dcode}]}} \\
  \hline
  1/6&\le& \mathcal{D}(L_5,1) &\le& 7/32^{\text{[Theorem~\ref{thm:domsettheorem}}]}  \\
    1/7&\le& \mathcal{D}(L_6,1) &\le& 3/16^{\text{[Theorem~\ref{thm:domsettheorem}}]}   \\
    \hline
    && \mathcal{D}(L_7,1)&=&1/8    \\
  1/9&\le& \mathcal{D}(L_8,1) &\le& 1/8  \\
  \hline
  1/10&\le& \mathcal{D}(L_9,1) &\le& 31/256^{\text{[Theorem~\ref{thm:domsettheorem}}]}  \\
  1/11&\le& \mathcal{D}(L_{10},1) &\le& 15/128^{\text{[Theorem~\ref{thm:domsettheorem}}]}  \\
  \hline
\end{array}$$
\caption[A table of bounds of densities of codes for small values of $n$.]{A table of bounds of densities of codes for small values of $n$.  All bounds not cited are due to \cite{Karpovsky1998}.}\label{fig:boundtable}
\end{figure}

For larger values of $n$ we are also able to find asymptotic bounds.

\begin{thm}\label{thm:maintheorem}  For sufficiently large $n$, there is a constant $b$ such that:
$$\frac{1}{n+1}\le \mathcal{D}(L_{n},1) \leq \left( 1 + \frac{b \ln \ln n}{\ln n} \right) \frac{1}{n+1}. $$
\end{thm}

Both the table of bounds listed above and Theorem~\ref{thm:maintheorem} follow immediately from Theorem~\ref{thm:domsettheorem} presented in Section~\ref{section:r1case} along with the numbers in Figure~\ref{fig:dominatingtable} and equation (\ref{eqn:hndomset}) respectively.

Next, in Section~\ref{section:monotonicity} we are able to show that $\mathcal{D}(L_n,r)$ is monotonically decreasing with respect to $n$.  This gives a nice result, but also gives us another upper bound for $\mathcal{D}(L_n,1)$.

\begin{thm}\label{thm:maintheorem2}
   Let $k$ be an integer with $2^k\le n+1$, then $$\frac{1}{n+1}\le \mathcal{D}(L_n,1)\le \frac{1}{2^k}.$$
\end{thm}

The theorem follows immediately from Theorem~\ref{monotonicitytheorem} and Theorem~\ref{thm:Hamming}.

It is worth noting that the two different upper bound theorems may give better bounds depending on $n$.  For example, if $n=2^k$, then this gives an upper bound $\mathcal{D}(L_n)\le 1/n$ as was shown in~\cite{Karpovsky1998}.  More generally, if $n=2^k+s$, this gives $\mathcal{D}(L_n)\le 1/(n-s)$.  If $s$ is small enough compared to $n$, then this would give a better bound than Theorem~\ref{thm:maintheorem}.  However, if $n=2^k-s$ for $s>2$ and $s$ sufficiently small compared to $n$, then Theorem~\ref{thm:maintheorem} gives a better result.  In any case, it is clear that $1/(n+1)\le \mathcal{D}(L_n,1)<2/(n+1)$ for all $n$.


In Section~\ref{4Dsection}, we examine a connection between $L_4$ and the king grid before turning to a general construction and lower bound for $\mathcal{D}(L_n,r)$ in Section~\ref{section:generalbounds}.  From the lower bound given in Theorem~\ref{thm:lowerboundtheorem} and Corollary~\ref{cor:upperboundcor}, this gives us another asymptotic bound.

\begin{thm}\label{thm:thetabound}
  For fixed $n$, $\mathcal{D}(L_n,r)=\Theta(1/r^{n-1})$ as $r\rightarrow\infty$.
\end{thm}

\section{Definitions}\label{section:latticedefs}

Given a connected, undirected graph $G=(V,E)$, we define
$B_r(v)$, called the ball of radius $r$ centered at $v$ to be
$$B_r(v)=\{u\in V(G): d(u,v)\le r\}. $$

We call any nonempty subset $C$ of $V(G)$ a \textdef{code} and its elements \textdef{codewords}.  A code $C$ is called \textdef{$r$-identifying} if it has the properties:
\begin{enumerate}
\item $B_r(v) \cap C \neq \emptyset$ for all $v$
\item $B_r(u) \cap C \neq B_r(v)\cap C$, for all $u\neq v$
\end{enumerate}
When $r=1$ we simply call $C$ an identifying code.  When $C$ is understood, we define $I_r(v)=I_r(v,C)=B_r(v)\cap C$.  We call
$I_r(v)$ the identifying set of $v$.  If $I_r(u)\neq I_r(v)$ for some $u\neq v$, the we say $u$ and $v$ are \textdef{distinguishable}.  Otherwise, we say they are \textdef{indistinguishable}.

Given a graph $G$, we say that $D\subset V(G)$ is a \textdef{dominating set} if for every $v\in V(G)\setminus D$, $v\sim u$ ($v\sim u$ means $v$ is adjacent to $u$) for some $u\in D$.

Next, we wish to define two families of graphs.  First we define the $n$-dimensional lattice $L_n=(V,E)$ where$$V=\Z^n,\qquad
E = \left\{\{(x_1,\ldots,x_n),(y_1,\ldots,y_n)\}: \sum_{i=1}^n |x_i-y_i| = 1\right\}.$$

We also define the binary cube $H_n$ to be the induced subgraph of $L_n$ on the vertex set $\{0,1\}^n$.

The density of a code $C$ for a finite graph $G$ is defined as $|C|/|V(G)|$.  Let $Q_m$ denote the set of vertices $(x_1, \ldots, x_n)\in
\Z^n$ with $|x_i|\le m$ for all $1\le i \le n$.  We define the
density $D$ of a code $C$ in $L_n$ similarly to how it is defined in~\cite{Charon2002} by
$$D=\limsup_{m\rightarrow\infty}\frac{|C\cap Q_m|}{|Q_m|}.$$

\section{Using Dominating Sets for Constructions when $r=1$}\label{section:r1case}

\begin{thm}[\cite{Karpovsky1998}]\label{LBtheorem}
  If $G$ is a $d$-regular graph and there exists a code for $G$, then $$\mathcal{D}(G,1)\ge\frac{2}{d+2}.$$\end{thm}

Although this theorem was originally stated for finite graphs, the theorem generalizes to our graphs as well.

Since $L_n$ is $2n$-regular, this gives an immediate corollary.
\begin{cor}\label{cor:LBlattice1}
  $$\mathcal{D}(L_n,1)\ge\frac{1}{n+1}$$
\end{cor}

Karpovsky, Chakrabarty, and Levitin~\cite{Karpovsky1998} later go on to address the issue of 1-identifying codes on the $n$-dimensional hypercube with $p^n$ vertices by use of error correcting codes.  In the case that $n=2^k-1$ for some integer $k$, the use of a Hamming code gives an identifying code which matches the lower bound.  For $L_n$, we can get the same result by taking the limit as $p\rightarrow\infty$ to achieve and optimal 1-identifying code.  However, for $n\neq 2^k-1$ we build on their technique in order to achieve reasonable upper bounds:

\begin{thm}\label{thm:domsettheorem}
  Let $D$ be a dominating set on $H_n$.  Then there exists a 1-identifying code $C$ on $L_n$ with density $|D|/2^n$.
\end{thm}

\begin{proof}
  First define $$C=\bigcup_{\mathbf{v}\in \Z^n} \bigcup_{\mathbf{d}\in D} \mathbf{d}+2\mathbf{v}.$$  In other words, we ``tile'' $L_n$ with copies of $H_n$ and copy $D$ to each of these tiles.  We have $c\in C$ if and only if $(c\mod 2)\in D$ where the modulus is taken coordinatewise.  As in~\cite{Charon2002}, the density of $C$ will be its density in each tile, which is $|D|/2^n$.

  It remains to show that $C$ is indeed a 1-identifying code.  Let $v=(x_1,\ldots, x_n)\in V(L_n)$.

  First we see that $v$ can be written as $v=2u+p$ where $p\in \{0,1\}^n$.  This follows since we can take $p=(c\mod 2)$ and then $v-p$ has coordinates which are all even.  Since $D$ is a dominating set, it follows that either $p\in D$ or $p\sim x$ for some $x\in D$.  If $p\in D$, then $p+2u=v\in C$ by definition.  If $p\sim x$ for some $x\in D$, then $(x+2u)\in C$ and $(x+2u)\sim v$.  In either case, we get $I_1(v)\neq\emptyset$.  Now we need to show that the second condition for a 1-identifying code holds.

  \textbf{Case 1}:  Suppose that $v\sim c$ for some $c\in C$.

   Then $c$ differs from $v$ in exactly one component.  If $c$ differs from $v$ in the $i$th component, then $c=(x_1,\ldots, x_i^*,\ldots, x_n)$. where $x_i^*=x_i+1$ or $x_i^*=x_i-1$.  However, note that $(x_1,\ldots, x_i+1,\ldots, x_n)$ and $(x_1,\ldots, x_i-1,\ldots, x_n)$ are both in $C$ since they are equivalent mod 2.  Furthermore, $v$ is adjacent to both of them.  We then note that no other vertex can have both of these vertices in its closed neighborhood.  Thus, $\{(x_1,\ldots, x_i+1,\ldots, x_n),(x_1,\ldots, x_i-1,\ldots, x_n)\}\subset I_1(v)$ but $\{(x_1,\ldots, x_i+1,\ldots, x_n),(x_1,\ldots, x_i-1,\ldots, x_n)\}\not\subset I_1(u)$ for any $u\neq v$.

  \textbf{Case 2}: $v\not\sim c$ for some $c\in C$.

  Since we have tiled $V(L_n)$ with copies of $H_n$ and $D$ is a dominating set on $H_n$, it follows that $v\in C$ and so $I_1(v)=\{v\}$.  Let $u$ be any other vertex in $V(L_n)$.  If $u\sim c$ for some $c\in C$, then $|I_1(u)|\ge 2$ as in Case 1 and so $I_1(v)\neq I_1(u)$.  Otherwise, $c\not\in I_1(u)$ and so $I_1(u)\neq I_1(v)$.

  Thus, $C$ is an identifying code and our proof is complete.\end{proof}

The minimum size of a dominating set on $H_n$ is well-studied, which gives us good general improvements.  For example, Table 6.1 of \cite{Cohen1997} gives values and bounds for the minimum size $K(n, 1)$ of a dominating set of $H_{n}$.  Figure~\ref{fig:dominatingtable} gives part of that table:

\begin{figure}[ht]\label{fig:smalldomsets}
$$\begin{array}{|c||cccccccccc|} \hline
n & 1 & 2 & 3 & 4 & 5 & 6 & 7 & 8 & 9 & 10 \\ \hline
K(n, 1) & 1 & 2 & 2 & 4 & 7 & 12 & 16 & 32 & \leq 62 & \leq 120 \\ \hline
\end{array}$$
\caption[A table of bounds for the size of a dominating set on $H_n$.]{A table of bounds for the size of a dominating set on $H_n$ given in~\cite{Cohen1997}.}\label{fig:dominatingtable}
\end{figure}

This gives good bounds for some small values of $n$ as seen in Figure~\ref{fig:boundtable}.

More generally, in \cite{Kabatyanskiui1988}, Kabatyanski{\u\i} and Panchenko prove that there is a constant $b$ such that for sufficiently large $n$,
\begin{equation}\label{eqn:hndomset}
K(n, 1) \leq \left( 1 + \frac{b \ln \ln n}{\ln n} \right) \frac{2^{n}}{n+1}.\end{equation}

\section{A Monotonicity Theorem}\label{section:monotonicity}
\begin{thm}\label{monotonicitytheorem}
   $\mathcal{D}(L_n,r)\ge \mathcal{D}(L_{n+1},r)$.
\end{thm}
\begin{proof}
Let $C$ be a code for $L_n$.  For short, if $\mathbf{x}=(x_1,\ldots, x_n)$ then we write $(x_1,\ldots, x_n, k) = (\mathbf{x},k)$.  Define $$C'=\bigcup_{\mathbf{c}\in C}\bigcup_{k\in \Z} (\mathbf{c},k).$$

In other words, $\mathbf{c'}=(\mathbf{c},k)\in C'$ if and only if $\mathbf{c}\in C$.  It is clear that $C'$ has the same density in $L_{n+1}$ that $C$ has in $L_n$. We now demonstrate that $C'$ is $r$-identifying for $L_{n+1}$.

First, it is easy to check that for each $v\in V(L_{n+1})$, there is a codeword $c$ such that $d(v,c)\le r$.  Write $v= (\mathbf{v}, k)$.  Since $C$ is an $r$-identifying code for $L_n$, there exists a codeword $\mathbf{c}\in C$ such that $d(\mathbf{v},\mathbf{c})\le r$.  Then $d(v,(\mathbf{c}, k))\le r$.

Let $u=(\mathbf{x},k)$ and $v=(\mathbf{y},\ell)$ with $u\neq v$.  If $\mathbf{x}=\mathbf{y}$ then there is some $\mathbf{c}\in C$ with $d_{xc}:=d(\mathbf{x},\mathbf{c})\le r$.  Since $u\neq v$, then without loss of generality we have $k<\ell$. Then we have $c'=(\mathbf{c},k-(r-d_{xc}))\in C'$ and
$$d(u,c')=d(\mathbf{x},\mathbf{c})+|k-(k-(r-d_{xc}))|=r$$ since $r-d_{xc}\ge0$, but
\begin{eqnarray*}
d(v,c') &=& d(\mathbf{x},\mathbf{c})+|\ell-(k-(r-d_{xc}))| \\
&=& d_{xc} +|(\ell-k)+(r-d_{xc})|\\
&=& r+(\ell-k)>r
\end{eqnarray*}
 and so $u$ and $v$ are distinguishable.

If $\mathbf{x}\neq\mathbf{y}$, then without loss of generality, there is some $\mathbf{c}\in C$ with $d(\mathbf{x},\mathbf{c})\le r$ and $d(\mathbf{y},\mathbf{c})>r$.  Then we have $c'=(\mathbf{c},k)\in C'$ and $d(u,c')\le r$ but $d(v,c') = d(\mathbf{y},\mathbf{c})+|k-\ell|>r$ and so $u$ and $v$ are distinguishable, completing the proof.
\end{proof}


\section{The 4-dimensional case}\label{4Dsection}

In this section we explore some connections between the 4-dimensional lattice and the king grid (square grid with diagonals).  It was shown by Charon, Hudry and Lobstein in~\cite{Charon2002} that there was an identifying code of density 2/9 for the king grid.  It was shown by Cohen, Honkala, Lobstein and Z\'{e}mor in~\cite{Cohen2001} that this is an optimal construction.  We wish to use this code of density 2/9 to construct a code for $L_4$, however, we do not show whether or not it is optimal.

The king grid, $G_K$, is defined to be the graph on vertex set $\Z\times\Z$ with edge set $E_K=\{\{u,v\}: u-v\in\{(0,\pm1),(\pm1,0),(1,\pm1),(-1,\pm1)\}\}$.

\begin{lemma}\label{kinglemma}
 $\mathcal{D}(L_4,1)\le \mathcal{D}(G_K,1)$.
\end{lemma}
\begin{proof}
  Let $C$ be a code of density $D$ on the king grid.  Let $c=(c_1,c_2)\in C$.  Then we define a code $C'$ for $L_4$ as $$C' = \bigcup_{c\in C}\bigcup_{i\in \Z}\bigcup_{j\in \Z}[(c_1,c_2,0,0)+i(1,1,1,0)+j(1,-1,0,1)].$$  In other words, we take our original copy of $C$ and copy it to two-dimensional cross-sections of $L_4$--shifting it when moving in the $x_3$ and $x_4$ directions.  Since $C'$ just consists of copies of $C$, it is clear that it also has density $D$.  It suffices to show that $C'$ is a 1-identifying code for $L_4$.

  Throughout, let $I^K(v)$ denote the identifying set of a vertex $v$ in the king grid, while $I(v)$ denotes the identifying set of a vertex in $L_4$.

  Let $u,v\in V(L_4)$.  Let us say that $u', v' \in \Z \times \Z$ are \textdef{diagonally adjacent} if $u'-v' \in \{\pm(1, 1), \pm (1, -1) \}$ and are \textdef{orthogonally adjacent} if $u'-v' \in \{ \pm(1, 0), \pm (0, 1) \}$.

  Without loss of generality, assume that $u=(x,y,0,0)$.

  \textbf{Case 1:} $v=(x',y',0,0)$.

  Note that there is some $c=(c_1,c_2)\in C$ such that $c\in I^K((x,y))\triangle I^K((x',y'))$.  Without loss of generality assume that $c\in I^K((x,y))\setminus I^K((x',y'))$.  If $c$ is orthogonally adjacent to $u$ then $(c_1,c_2,0,0)\in I(u)\setminus I(v)$ and so $u$ and $v$ are distinguishable.  Next, suppose that $c=(x+1,y+1)$.  Then $c'=(x,y,-1,0)=(x+1,y+1,0,0)+(-1)\cdot(1,1,1,0)\in C'$.  Further, $d(u,c')=1$ but $d(v,c') = |x-x'|+|y-y'| + 1 \ge 2$ since $(x,y)\neq (x',y')$.  Hence, $u$ and $v$ are distinguishable.  Similar arguments apply if $c\in\{(x+1,y-1),(x-1,y+1),(x-1,y-1)\}$, completing Case 1.

  Before moving to Case 2, note that this shows how we can find a codeword in $I(v)$ for any vertex.  If $v$ is orthogonally adjacent to a codeword $(c_1,c_2)$, then $(c_1,c_2,0,0)\in I(v)$.  If it is diagonally adjacent, say $c=(x+1,y+1)$, then $c'=(x,y,-1,0)=(x+1,y+1,0,0)+(-1)\cdot(1,1,1,0)\in I(v)$, with the other cases following similarly.

  \textbf{Case 2:} $v\neq (x',y',0,0)$.

  Since $(1,0,0,0),(0,1,0,0),(1,1,1,0),$ and $(1,-1,0,1)$ are linearly independent, we may write $v=(x',y',0,0)+i(1,1,1,0)+j(1,-1,0,1)$ uniquely with $i$ and $j$ not both equal to 0.

  \textbf{Subcase 2.1:} $(x,y)=(x',y')$.

  We see that $d(u,v) = |i+j|+|i-j|+|i|+|j|$.  If $i$ and $j$ are both non-zero, then either $|i+j|$ or $|i-j|$ is at least 1 and so $d(u,v)\ge 3$.  If $j=0$, then $d(u,v)=3|i|\ge 3$ and likewise if $i=0$ then $d(u,v)\ge 3|j|\ge 3$.  Since there exists $c_u\in I(u)$ from Case 1, the reverse triangle inequality gives $d(v,c_u)\ge d(u,v)-d(u,c_u)\ge 2$ and so $u$ and $v$ are distinguishable.  This completes Subcase 2.1.

  \textbf{Subcase 2.2} $(x,y)\neq(x',y')$.

  Without loss of generality, there is some $c=(c_1,c_2)\in I^K((x,y))\setminus I^K((x',y'))$.

    First suppose that either $c=(x,y)$ or that $c$ is orthogonally adjacent to $(x,y)$.  Then  $c'=(c_1,c_2,0,0)\in I(u)$.  If follows that $$d(v,c')=|x'+i+j-c_1|+|y'+i-j-c_2|+|i|+|j|.$$ If $|i|+|j|\ge 2$ then this distance is at least 2 and we are done, so we first assume that $|i|=1$ and $j=0$.  This gives $$d(v,c')=|(x'-c_1)+i|+|(y'-c_2)+i| + 1.$$ If $x'-c_1\neq -i$ or $y'-c_2\neq -i$ then this distance is at least 2 and we are done.  However, if we have equality in both cases, then $(x',y')-(c_1,c_2)=(-i,-i)\in\{(1,1),(-1,-1)\}$ and so $c\in I^K((x',y'))$--a contradiction.  Likewise, if $|j|=1$,$i=0$ and $d(v,c')< 2$ then $(x',y')-(c_1,c_2)=(-j,j)\in\{(1,-1),(-1,1)\}$ and so we have a contradiction.  Thus, $c\not\in I(v)$.

    Finally, we need to consider the case that $c$ is diagonally adjacent to $(x,y)$.  Once again, suppose that $c=(x+1,y+1)$--the other cases follow similarly.
    Then $c_u=(x,y,-1,0)=(x+1,y+1,0,0)+(-1)\cdot(1,1,1,0)\in I(u)$.  Then we have $$d(v,c')=|x'+i+j-x|+|y'+i-j-y|+|i+1|+|j|.$$

    Once again, if $|i+1|+|j|\ge 2$, then we are done.  First suppose that $|i+1|=0$ and so $|j|=1$.  This gives $$d(v,c')=|(x'-x)+j-1|+|(y'-y)+j-1|+1.$$ Then, this is at least 2 unless $x-x' = j-1$ and $y-y'=j-1$.  If $j=1$, then this gives $(x,y)=(x',y')$, which contradicts our hypothesis.  Hence, we must have $j=-1$ and so $x'=x+2$ and $y' = y+2$.  But then we have $(x',y')\sim c = (x+1,y+1)$ in the king grid, which is a contradiction.

    If $j=0$ and $|i+1|=1$, then $$d(v,c')=|(x'-x)+i|+|(y'-y)+i|+1$$ and so we must have $x-x'=i$ and $y-y'=i$.  Again, since $(x,y)\neq (x',y')$ and $|i+1|=1$, this must mean $i=-2$ and so $x'=x+2$ and $y'=y+2$, giving the same contradiction as above.  This completes both Case 2 and our proof.
\end{proof}

\begin{thm}\label{thm:4Dcode}
  $\mathcal{D}(L_4,1)\le 2/9$.
\end{thm}
\begin{proof} The proof is immediate from~\cite{Charon2002} and Lemma~\ref{kinglemma}.
\end{proof}

\section{General Bounds and Construction}\label{section:generalbounds}

We finally wish to produce some general bounds for $r$-identifying codes on the $L_n$.  We start with a lower bound proof, in the style of Charon, Honkala, Hudry and Lobstein\cite{Charon2001}.  First, we define $b_k^{(n)} = |B_k(v)|$ for $v\in V(L_n)$.

\begin{lemma}\label{lemma:rlowerbound}
  The minimum density of an $r$-identifying code for $L_n$ is at least $$\frac{\lceil \log_2(2n+1)\rceil}{b_{r+1}^{(n)} - b_{r-1}^{(n)}}.$$
\end{lemma}
\begin{proof}
  Let $v\in V(L_n)$ and $u_1,u_2,\ldots, u_{2n}$ be its neighbors.  If $d(v,x)>r+1$, then it is easy to see that $d(u_i,x)\ge r+1$ for all $i$.  Likewise, it is easy to check that if $d(v,x)\le r-1$, then $d(u_i,x)\le r$ for all $i$.  In other words, all vertices outside of $B_{r+1}(v)$ are not in $B_r(s)$ for any $s\in S=\{v,u_1,u_2,\ldots, u_{2n}\}$ and all vertices inside of $B_{r-1}(v)$ are in $B_r(s)$ for all $s\in S$.

  Next, let $C$ be an $r$-identifying code for $L_n$.  For $s,s'\in S$ with $s\neq s'$, we must have $I_r(s)\triangle I_r(s')\subset B_{r+1}(v)\setminus B_{r-1}(v)$.  Let $K(s)=I_r(s)\cap (B_{r+1}(v)\setminus B_{r-1}(v))$.  We claim for $K(s)\neq K(s')$.  Suppose not.  Then $I_r(s) = K(s) \cup (C\cap B_{r-1}(v)) = I_r(s')$ and so they are not distinguishable.  Hence, $K(s)$ must be distinct for each $s\in S$.  Since the minimum number of elements of a set to produce $2n+1$ distinct subsets is $\lceil\log_2(2n+1)\rceil$, there must be $\lceil\log_2(2n+1)\rceil$ codewords in $B_{r+1}(v)\setminus B_{r-1}(v)$.  We refer to the methods used in \cite{Charon2001} to show this gives the lower bound described in the statement of the lemma.
\end{proof}

\begin{cor}\label{cor:34LBs}Let $r\ge1$ be an integer, then
    $$\mathcal{D}(L_3,r)\ge\frac{3}{8(r^2+r+1)}\qquad\text{and}\qquad\mathcal{D}(L_4,r)\ge\frac{15}{8(2r^3+5r^2+5r+3)}.$$
\end{cor}
\begin{proof}
  Using Equation~(\ref{eqn:recurr}) presented in Theorem~\ref{thm:lowerboundtheorem}, it is easy to calculate $$b_r^{(3)}=\frac13(4r^3+6r^2+8r+3)\qquad \text{and}\qquad b_r^{(4)}=\frac13(2r^4+4r^3+10r^2+8r+3), $$ whence the corollary follows.
\end{proof}

Corollary~\ref{cor:34LBs} suggests a general pattern for the bound obtained in Lemma~\ref{lemma:rlowerbound}.  We should have $\mathcal{D}(L_n,r)\ge a/(br^{n-1} + o(r^{n-1}))$.  In Theorem~\ref{thm:lowerboundtheorem} we show that this pattern does indeed hold.

\begin{thm}\label{thm:lowerboundtheorem}
  $$\mathcal{D}(L_n,r)\ge\frac{(n-1)!\lceil \log_2(2n+1)\rceil}{2^{n+1}r^{n-1} + p_{n-2}(r)}$$ where $p_{n-2}(r)$ is a polynomial in $r$ of degree no more than $n-2$.
\end{thm}
\begin{proof}

  It is easy to check that $b_r^{(n)}$ is the number of solutions in integers to $$|x_1|+|x_2|+\cdots + |x_n| \le r.$$  We may then derive a recurrence relation for $b_r^{(n+1)}$ by setting $x_{n+1}=i$ for $i\in[-r,r]$.  Then for each $i$, this gives $$|x_1|+|x_2|+\cdots + |x_n| \le r-|i|.$$  Summing over all $i$ gives the recurrence
  \begin{equation}\label{eqn:recurr}
    b_r^{(n+1)} = 2\left(\sum_{k=0}^r b_{k}^{(n)}\right) - b_r^{(n)}.
  \end{equation}
  If we define $b_{-1}^{(n)}=0$ then letting $a_r^{(n)}=b_{r+1}^{(n)} - b_{r-1}^{(n)}$ this gives:
  \begin{equation}\label{eqn:recurr2}
    a_r^{(n+1)} = 2\left(\sum_{k=0}^r a_{k}^{(n)}\right) - a_r^{(n)}.
  \end{equation}
  Note that $a_r^{(n)}$ is exactly the denominator in Lemma~\ref{lemma:rlowerbound}.  We wish to show by induction that $$a_r^{(n)}=c_n r^{n-1} + p_{n-2}(r)$$ where $p_{n-2}(r)$ is a polynomial of $r$ of degree no more than $n-2$.  For $n=1$ we have $$a_r^{(1)} = b_{r+1}^{(1)} - b_{r-1}^{(1)} = (2r+3)-(2r-1) = 4.$$ and so our base case holds. Next we apply induction:
  \begin{eqnarray*}
    a_r^{(n+1)} &=& 2\left(\sum_{k=0}^r a_{k}^{(n)}\right) - a_r^{(n)}\\
    &=& 2\sum_{k=0}^r [c_n k^{n-1} + p_{n-2}(k)] + c_nr^{n-1} + p_{n-2}(r) \\
    &=& 2c_n\sum_{k=0}^r k^{n-1} + q_{n-1}(r)
  \end{eqnarray*}
  At this point we pause for a second to discuss $$q_{n-1}(r) := c_nr^{n-1} + p_{n-2}(r) + 2\sum_{k=0}^r p_{n-2}(k).$$  If we write $$p_{n-2}(k)=\alpha_{n-2}k^{n-2} + \alpha_{n-3}k^{n-3} + \cdots + \alpha_0$$ then $$\sum_{k=0}^r p_{n-2}(k)=\sum_{i=0}^{n-2}\alpha_i\sum_{k=0}^r k^i$$ is simply a polynomial of $r$ of degree no more than $n-1$  since $\sum_{k=0}^r k^i$ can be written as a polynomial of degree no more than $i+1$.

  Before we finish up, we simply need to note that $$\sum_{k=0}^r k^{n-1} = \frac{r^n}{n} + \hat{q}_{n-1}(r)$$ where $\hat{q}_{n-1}(r)$ is a polynomial of degree no more than $n-1$.  Thus we have $$a_r^{(n+1)} = \frac{2c_n}{n}r^n + 2c_n\hat{q}_{n-1}(r) + q_{n-1}(r)$$ is of the form we desire. Furthermore, this gives the recurrence relation $c_{n+1} = 2c_n/n$.  Combining this with the fact that $c_1=4$ we get $$c_n = \frac{2^{n+1}}{(n-1)!}.$$ Plugging this in gives the desired result.
\end{proof}

\begin{thm}\label{thm:upperboundeventheorem}
  For $r_0,n\ge 2$, if $2r_0$ is divisible by $n+1$, then $$\mathcal{D}(L_n,r)\le \frac{(n+1)^{n-1}}{2^{n}r_0^{n-1}} $$
  for all $r\ge r_0$.
\end{thm}
\begin{proof}
  Let $2r_0$ be divisible by $n+1$ and let $k=2r_0/(n+1)$.  We define a code $$C=\left\{\left(kx_1,kx_2,\ldots, kx_{n-1},\ell\right): x_1 + x_2 + \cdots + x_{n-1} \equiv 1 \pmod 2 \right\}.$$  Further, let $$S=\left\{\left(kx_1,kx_2,\ldots, kx_{n-1},\ell\right): x_1 + x_2 + \cdots + x_{n-1} \equiv 0 \pmod 2 \right\}.$$
  $C$ will be our code and $S$ will serve as a set of reference points which we will use later.

  We first wish to calculate the density $C\cup S$.  This is simply a tiling of $\Z^{n}$ by the region $[0,k-1]^{n-1}\times\{0\}$ which has only a single codeword in it.  Hence, the density of $C\cup S$ is $1/k^{n-1}=(n+1)^{n-1}/(2^{n-1}r_0^{n-1})$.  Then since $C$ and $S$ are disjoint, isomorphic copies of each other, the density if $C$ is half the density of $C\cup S$, which is the density stated in the theorem.

  Next, we wish to show that $C$ is an $r$-identifying code for $r\ge r_0$.  Let $e^{(i)}$ represent the $n$-dimensional vector with a 1 in the $i$th coordinate and a 0 in all other coordinates.  For any vertex $u$, let $u_j$ denote the coordinate in the $j$th position of $u$.

  For $s\in S$, we define the \emph{corners} of $s$ to be the codewords $c$ of the form $c=s \pm ke^{(i)}$ for some $1\le i\le n-1$.

  The remainder of the proof consists of 3 steps:

  \begin{enumerate}
    \item Each vertex $v\in V(G)$ is distance at most $nk/2$ from some $s\in S$ and $v$ is distance at most $r$ from each of the corners of $s$ (in addition, this shows that $I_r(v)$ is nonempty).
    \item If $v=(\mathbf{v},\ell)$, we can uniquely determine $\ell$ from $I_{r}(v)$.  Furthermore, if $c=(\mathbf{c},\ell)\in I_{r}(v)$, we can determine $d(v,c)$.
    \item If $v=(v_1,\ldots, v_{n-1},\ell)$, we can uniquely determine $v_i$ from $I_{r}(v)$ for each $i$.  Thus, $v$ is distinguishable from all other vertices in the graph.
  \end{enumerate}

  \textbf{Step 1}:  Let $v=(v_1,v_2,\ldots, v_{n-1}, \ell)$.  Without loss of generality, we may assume that $(v_1,v_2,\ldots, v_{n-1})\in [0,k]^{n-1}$.    For $i=1,2,\ldots,n-2$ define $$a_i=\left\{\begin{array}{cc}
                                           0 & \text{if $v_i\le k/2$} \\
                                           k & \text{if $v_i> k/2$}
                                         \end{array}
  \right..$$  We then see that $|v_i-a_i|\le k/2$ in either case.  Now consider the vertices $(a_1,a_2,\ldots, a_{n-2},0,\ell)$ and $(a_1,a_2,\ldots, a_{n-2},k,\ell)$.  One of these is in $S$.  Let $a_{n-1}=0$ if the former is in $S$ and $a_{n-1}=k$ if the latter is in $S$.  Then $|v_{n-1}-a_{n-1}|\le k$.  Hence we have
  \begin{eqnarray*}
    d(v,(a_1,a_2,\ldots,a_{n-2},a_{n-1},\ell)) &=& |v_{n-1}-a_{n-1}| + \sum_{i=1}^{n-2} |v_i-a_i| \\
    &\le& k + (n-2)k/2 \\
    &=& nk/2.
  \end{eqnarray*}

  Let $c$ be a corner of $s=(a_1,a_2,\ldots, a_{n-2},a_{n-1},\ell)$.  Then $$d(v,c) \le d(v,s) + d(s,c) \le nk/2 + k = (n+2)k/2 = r_0\le r.$$

  \textbf{Step 2}: Next, we need to determine the last coordinate of $v$.  Write $v=(\mathbf{v},\ell)$.  Suppose that $c=(\mathbf{c},j)\in I(v)$.  We then see that $(\mathbf{c},\ell)\in I(v)$ since $d(v,(\mathbf{c},\ell))\le d(v,c)$.  If $d(v,(\mathbf{c},\ell))=d_1\le r$, then we see that $(\mathbf{c},\ell\pm j)\in I(v)$ for $j=0,1,\ldots, r-d_1$.  Hence, these codewords form a path of length $2(r-d_1)+1$. Thus, if $\ell_1=\min\{j:(\mathbf{c},j)\in I(v)\}$ and $\ell_2=\max\{j:(\mathbf{c},j)\in I(v)\}$, it follows that $$\ell = \frac{\ell_1+\ell_2}{2}.$$  Furthermore, this tells us once we know $\ell$, we can uniquely determine the distance between $v$ and $c$ to be $$d(v,c)=d(v,(\mathbf{c},\ell_2))-d((\mathbf{c},\ell_2),(\mathbf{c},\ell)) = r - (\ell_2-\ell).$$

  \textbf{Step 3}:  Finally, from Step 1 we know that there is some vertex $s\in S$ such that the codewords $s\pm ke^{(i)}\in I_r(v)$ for each $i$.  Hence, we are guaranteed that there are $m\ge 2$ codewords $c^{(0)},\ldots, c^{(m-1)}$ such that $c^{(j)}=c^{(0)} + 2kje^{(i)}$ and $c^{(j)}\in I_r(v)$ for each $j$.

  Now let $$S=\sum_{{p=1}\atop {p\neq i}}^{n-1} |v_p - c_{p}^{(0)}|.$$  We then see that $$d(v,c^{(j)}) = |v_i-c_i^{(j)}| + S$$ which is minimized by minimizing $|v_i-c_i^{(j)}|$.  Furthermore, the expression $|v_i-x|$ is unimodal and so the two smallest values of $|v_i-c_i^{(j)}|$ must happen for consecutive integers and they must be amongst our aforementioned $m$ codewords.  Let $a=c^{(\ell)}$ and $b=c^{(\ell+1)}$ be these codewords.  It is easy to check that $a_i\le v_i \le b_i$ by considering evenly spaced point plotted along  the graph of $f(x)=|v_i-x|$.

    This gives
  \begin{eqnarray*}
    d(v,a) &=& v_i-a_i + S \\
    d(v,b) &=& b_i-v_i + S
  \end{eqnarray*}
  Since $a$ and $b$ are codewords, the distances listed above are all known quantities from Step 2.  Subtracting the second line from the first and solving for $v_i$ gives: $$v_i=\frac{d(v,a)-d(v,b)+a_i+b_i}{2}.$$


  Since these are all known quantities, we can compute $v_i$, completing the proof.

\end{proof}
  \begin{cor}\label{cor:upperboundcor}
    If $n$ is even, $0\le k<n$, $r\ge n+1$, and $r\equiv k\pmod{(n+1)}$ then $$\mathcal{D}(L_n,r)\le \frac{(n+1)^{n-1}}{2^{n}(r-k)^{n-1}}. $$
    If $n$ is odd, $0\le k<n/2$, $r\ge (n+1)/2$, and $r\equiv k\pmod{(n+1)/2}$ then $$\mathcal{D}(L_n,r)\le \frac{(n+1)^{n-1}}{2^{n}(r-k)^{n-1}}. $$
  \end{cor}


\section{Conclusions}
It is worth noting that the lower bound given in Theorem~\ref{thm:lowerboundtheorem} can only be evaluated as $r\rightarrow\infty$ and not as $n\rightarrow\infty$ since the polynomial in the denominator is a polynomial in $r$, but the coefficients depend on $n$.  However, for fixed $n$ we can make a comparison of the bounds by taking the ratio of the upper bound to the lower bound.  This gives:
\begin{eqnarray*}
 &&\frac{(n+1)^{n-1}}{2^{n}r^{n-1}}\left/ \frac{(n-1)!\lceil \log_2(2n+1)\rceil}{2^{n+1}r^{n-1} + o(r^{n-1})}\right. \\
 &=& \frac{2^{n+1}r^{n-1} + o(r^{n-1})}{2^nr^{n-1}}\cdot\frac{(n+1)^{n-1}}{(n-1)!\lceil\log_2(2n+1)\rceil}\\
 &\approx& (2+o(1))\cdot\frac{(n+1)^{n-1}}{(n-1)^{n-1}}\cdot\frac{e^{n-1}}{\sqrt{2\pi n}\lceil\log_2(2n+1)\rceil} \\
  &\approx& \frac{2e^{n+1}}{\sqrt{2\pi n}\lceil\log_2(2n+1)\rceil}\\
\end{eqnarray*}
and so our lower bound differs from our upper bound by slightly less than a multiplicative factor of $e^n$ when $r\gg n\gg0$.
\section{Acknowledgements}
I would like to acknowledge the support of a National Science Foundation grant.  Part of this research was done as a research assistant under an NSF grant with co-PIs Maria Axenovich and Ryan Martin.  I would also like to thank  Jonathan D.H. Smith and Cliff Bergman for their help with this paper.  In addition, I would like to thank the referees from a previous submission whose help made this paper much better than it would have been otherwise. 
\chapter{ADDENDUM TO ``LOWER BOUNDS FOR IDENTIFYING CODES IN SOME INFINITE GRIDS'' AND ANOTHER RESULT}\label{chapter:squareAddendum}

As promised in our paper ``Lower Bounds for Identifying Codes in Some Infinite Grids'', we need to complete the proof of that $D(G_H,3)\ge 3/25$.  To get this, we just need one key lemma and then the result will follow immediately from Lemma~\ref{generalpairlemma}.  We also include another result about $(r\leq 2)$-identifying codes.

\section{Proof of Theorem~\ref{hexr3theorem}}

\begin{lemma}\label{lemma:r3hex} Let $C$ be a 3-identifying code for the hex grid.
For each $c\in C$, $p(c)\le 8$.
\end{lemma}

\begin{proof}
We start by noting that if there are code words in $B_1(c)$ or
$B_2(c)$, then $p(c)\le 8$.

Suppose first that there is another code word which is in $B_1(c)$.
We can easily see that there are 14 vertices in $B_3(c)$ which are
distance 3 or less from both $c$ and our second code word.  Hence,
at most one of those can witness a pair containing $c$.  Since there
are only 5 other vertices in $B_3(c)$ we have $P(c)\le 6$ (see
figure \ref{fig:r1}).

\begin{figure}[ht]
\centering
\includegraphics[totalheight=0.3\textheight]{hexball3c1.mps}
\caption[A ball of radius 3 in the hexagonal grid (1 of 3)]{The Ball of Radius 3 surrounding $c$.  Only one of the
vertices in white squares can witness a pair containing
$c$.}\label{fig:r1}
\end{figure}

Likewise, if there is another codeword which is distance 2 from $c$,
there are 12 vertices in $B_3(c)$ which are distance 3 or less from
both $c$ and our second code word.  Hence, at most one of those can
witness a pair containing $c$.  Since there are only 7 other
vertices in $B_3(c)$ we have $p(c)\le 8$ (see figure \ref{fig:r2}).

Hence, for the rest of the proof, we may assume that there are no
other codewords in $B_2(c)$.

\begin{figure}[ht]
\centering
\includegraphics[totalheight=0.3\textheight]{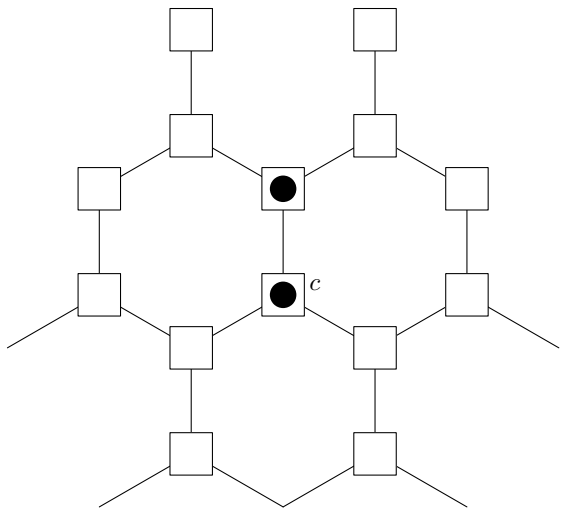}
\caption[A ball of radius 3 in the hexagonal grid (2 of 3)]{The Ball of Radius 3 surrounding $c$.  Only one of the
vertices in white squares can witness a pair containing
$c$.}\label{fig:r2}
\end{figure}

Next, we name the vertices in $B_3(c)$.  The immediate neighbors of
$c$ are named $u_1,u_2,$ and $u_3$.  The neighbors of $u_i$ are
labeled $u_{i,1}$ and $u_{i,2}$ and the neighbors of $u_{i,j}$ are
labeled $u_{i,j,1}$ and $u_{i,j,2}$ as in Figure \ref{fig:b3}.

\begin{figure}[ht]
\centering
\includegraphics[totalheight=0.3\textheight]{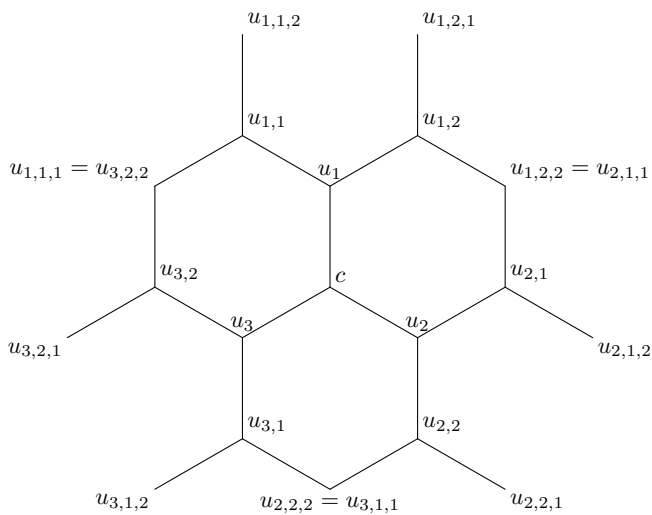}
\caption[A ball of radius 3 in the hexagonal grid (3 of 3)]{The Ball of Radius 3 surrounding $c$.}\label{fig:b3}
\end{figure}

We define a \emph{branch} $A_i$ of $B_3(c)$ to be the set
$\{u_i,u_{i,1},u_{i,2},u_{i,1,1},u_{i,1,2},u_{i,2,1},u_{i,2,2}\}$
for $i=1,2,3$. We note that $A_1\cup A_2 \cup A_3 =
B_3(c)\setminus\{c\}$ and that $|A_i\cap A_j|=1$ for $i\neq j$.

  \begin{fact}\label{facttype1}
    If there are no codewords in $B_2(c)$ and $|I_3(u_i)|\ge 2$,
then $A_i$ contains at most 3 vertices which witness a pair
containing $c$.
  \end{fact}
  To see that this is true, note that there is some other codeword
  $c'\in B_3(u_i)$.  Thus, there is a path of length at most 3
  between $u_i$ and $c'$.  This path cannot go through $c$, since
  then $d(c,c')\le 2$, contradicting the fact that there are no
  other code words in $B_2(c)$.  Hence, without loss of generality,
  we may assume the path goes through $u_{i,2}$.  Thus,
  $d(u_{1,2},c')\le 2$ and so $d(u_{1,2,i},c')\le 3$ for $i=1,2$.
  Hence, at most one of the vertices in
  $S=\{u_i,u_{i,2},u_{i,2,1},u_{i,2,2}\}$ can witness a pair
  containing $c$. If none of these witness pairs containing $c$,
  then our claim is finished, hence assume that $v$ witnesses a pair
  containing $c$ for some $v\in S$.

  We need to show that one of $u_{i,1}, u_{i,1,1},$ and $u_{i,1,2}$
  does not witness a pair containing $c$.  By a symmetric argument,
  we see that if there is path of length at most 3 through $u_{i,1}$
  between $u_i$ and some other code word $c''$, then at most one of
  $u_i,u_{i,1},u_{i,1,1},u_{i,1,2}$ can witness a pair containing
  $c$ and so there are at most 2 vertices which witness a pair
  containing $c$ in $A_i$.  But now we see that $B_3(u_{1,1})\subset
  B_3(u_i)\cup B_3(u_{1,1,1})\cup B_3(u_{1,1,2})$.  Hence, if
  $u_{i,1}$ witnesses a pair containing $c$ this path cannot go
  through $u_i$ (or that codeword is in $B_3(u_i)$). So the path
  goes through either $u_{i,1,1}$ or $u_{i,1,2}$ and thus that
  vertex cannot witness a pair containing $c$. Hence, there can be
  at most 3 pairs in $A_i$.  This concludes the proof of fact
  \ref{facttype1}.

  We already know that if there are any code words in $B_2(c)$, then
  $p(c)\le 8$ and so we are only left with the cases where there is
  there is a codeword in $B_3(c)\setminus B_2(c)$ or there are no
  other codewords in $B_3(c)$.

  We start with the case that there $B_3(c)\cap C = \{c\}$.  Note
that in this case, we have $|I_3(c)|=1$.   Since $I_r(u_i)\neq
I_r(c)=\{c\}$, by fact \ref{facttype1} this gives that $A_i$
contains no more than 3 pairs for $i=1,2,3$.

  We will start by talking about $A_1$.  There must be another
codeword in $B_3(u_1)$, so up to symmetry, there are 3 distinct
cases.

  \textbf{Case 1:}  The vertex $v$ adjacent to $u_{1,1,2}$ and
$u_{1,2,1}$ is in $C$.

  We note that this vertex is at most distance 3 from every other
vertex in $A_1$ and so at most one vertex in $A_1$ can witness a
pair containing $c$.  Since $A_2$ and $A_3$ contain at most 3
vertices witnessing pairs and $c$ is does not witness a pair with
itself, this gives $p(c)\le 7$.

  \textbf{Case 2:}  A vertex $v$ adjacent to $u_{1,2,1}$ which is
not in $B_3(c)$ is in $C$.

  Note there are two such vertices.  In either of these cases, each
of $u_1, u_{1,2},u_{1,2,1},u_{1,1,2},u_{1,2,2}$ is distance at most
distance 3 from $v$ and so at most one of them can witness a pair
containing $c$.  If one of the other 2 vertices $u_{1,1}$ or
$u_{1,1,1}$ does not witness a pair containing $c$, the there are
only 2 pairs in $A_1$. Since $A_2$ and $A_3$ contain at most 3 pairs
each and $c$ is not a pair, this would give $p(c)\le 8$.

  If however, both of them form pairs with $c$ then note that
$u_{1,1,1}=u_{3,2,2}$ is also in $A_3$.  Since $A_3$ contains at
most 3 vertices which witness pairs, there are at most 2 other
vertices in $A_3$ which witness pairs containing $c$.  This gives a
total of at most 5 vertices that witness pairs containing $c$ in
$A_1$ and $A_3$ and since $A_2$ contains at most 3 pairs, $p(c)\le
8$.

  \textbf{Case 3:} The vertex $v$ adjacent to $u_{1,2,2}$ which is
not in $B_3(c)$ is in $C$.

  First, assume that $u_1$ does not witness a pair containing $c$.

  Let $w$ be the vertex adjacent to $u_{1,1,1}$ which is not in
$B_3(c)$.  If $w$ is in $C$ then we see that $u_{1,2},u_{1,2,1},$
and $u_{1,2,2}$ are all within distance 2 of $v$ so at most one of
them can witness a pair.  Likewise, $u_{1,1},u_{1,1,1},$ and
$u_{1,1,2}$ are all within distance 2 of $w$ so at most one of them
can witness a pair containing $c$. Since $u_1$ does not witness a
pair, this gives at most 2 vertices witnessing pairs in $A_1$. As in
the other cases, this gives $p(c)\le 8$.

  If one of the other vertices in $B_3(u_1)$ is a codeword, then we
are in either Case 1 or Case 2 again, so $p(c)\le 8$. Thus, we may
assume that $u_1$ witnesses a pair containing $c$.

  We now wish to turn our attention to $A_2$.  Note that $u_2,
u_{2,1}, u_{2,1,1}$ and $u_{2,1,2}$ are all within distance 2 of
$v$.  Thus, none of them can witness pairs containing $c$.  This
leaves us with only $u_{2,2}, u_{2,2,1}, $ and $u_{2,2,2}$ which
could witness pairs containing $c$.  However, note that
$$B_3(u_{2,2})\subset B_3(c) \cup B_3(u_{2,2,1})\cup
B_3(u_{2,2,2}).$$  Since $B_3(c)$ contains no codewords other than
$c$, any codeword in $B_3(u_{2,2})$ must also be in either
$B_3(u_{2,2,1})$ or $B_3(u_{2,2,2})$ and so all 3 of them can't
witness pairs containing $c$.  Thus, $A_2$ contains at most 2
vertices witnessing pairs with $c$ and so $p(c)\le 8$.

  This concludes the case where there are no other codewords in
$B_3(c)$.  We now only have to consider the case where there is a
codeword in $B_3(c)\setminus B_2(c)$.  Up to symmetry, there are
only 2 distinct cases to consider.

\textbf{Case 1:} $u_{2,2,2}\in C$.

We partition $B_3(c)$ into sets as follows:

\begin{eqnarray*}
  S_0 &=& \{c, u_{2,2,2}, u_{3,1}, u_{3,1,2}, u_3, u_{3,2}, u_{2,2}, u_2, u_{2,1},u_{2,1,2}\} \\
  S_1 &=& \{u_{1,1}, u_{1,1,1}, u_{1,1,2}\}\\
  S_2 &=& \{u_{1,2}, u_{1,2,1}, u_{1,2,2}\}\\
  S_3 &=& \{u_1\}\\
  S_4 &=& \{u_{3,2,1}, u_{2,2,1}\}\\
\end{eqnarray*}

Note that each vertex in $S_0$ is within distance 3 of $u_{2,2,2}$
and so at most one of them can witness a pair containing $c$.

In $S_1$ we see that $B_3(u_{1,1})\subset B_3(u_{1,1,1}) \cup
B_3(u_{1,1,2})\cup B_2(c)$.  Since we are assuming that there are no
code words in $B_2(c)$ any codeword within distance 3 of $u_{1,1,1}$
is also within distance 3 of one of the other 2 codewords in $S_1$.
Hence, not all 3 can witness pairs containing $c$ and so at most 2
vertices in $S_1$ form pairs with $c$.

A symmetric argument shows that at most 2 vertices in $S_2$ can
witness pairs containing $c$.

Next, we consider the lone vertex $u_1$ in $S_3$.  Suppose that it
witnesses a pair containing $c$.  Then there is some vertex $c'\in
C$ such that there is a path between $u_1$ and $c'$ which is of
distance less than 3.  If the path goes through $c$, then $c'\in
B_2(c)$ which is a contradiction.  Thus, the path goes through
either $u_{1,1}$ or $u_{1,2}$.  If it goes through $u_{1,1}$ then
$c'\in B_2(u_{1,1})$ and hence in $B_3(v)$ for all $v\in S_1$ and so
nothing in $S_1$ can witness a pair containing $c$.  Likewise, if
the path goes through $u_{1,2}$, nothing in $S_2$ can witness a pair
containing $c$. Totalling this up we see that there is at most 1
vertex in $S_0$ which can witness a pair containing $c$, 2 vertices
in $S_1\cup S_2$, and the three remaining vertices in $S_3$ and
$S_4$ which can witness pairs containing $c$. This gives $p(c)\le
6$.

If $u_1$ does not witness a pair containing $c$ then we have at most
1 vertex in $S_0$ which witnesses a pair containing $c$, 2 vertices
in each of $S_1$ and $S_2$ which witness pairs containing $c$, and
the two remaining vertices in $S_4$ which can witness pairs
containing $c$. This gives $p(c)\le 7$.

\textbf{Case 2:} $u_{1,2,1}\in C$

This case runs basically the same way as the first case.  We
partition $B_3(c)$ into sets:

\begin{eqnarray*}
  S_0 &=&\{c,u_{1,2,1},u_{1,1,2}, u_{1,1},u_1,u_{1,2},u_{1,2,2},u_{2,1}\} \\
  S_1 &=&\{u_{3,2},u_{3,2,1},u_{3,2,2}\} \\
  S_2 &=&\{u_{3,1},u_{3,1,1},u_{3,1,2}\} \\
  S_3 &=&\{u_2,u_{2,2},u_{2,2,1},u_{2,1,2}\} \\
  S_4 &=&\{u_3\} \\
\end{eqnarray*}

Each vertex in $S_0$ is within distance 3 of $u_{1,2,1}$ and so at
most one of them can witness a pair containing $c$.

Using an argument similar to the one we used for $S_1$ in the first
case, we see that at most 2 vertices in $S_1$ and $S_2$ can witness
pairs containing $c$.

In $S_3$, we see that $B_3(u_2)\subset B_3(u_{2,2})\cup
B_3(u_{2,2,1})\cup B_3(u_{2,1,2})\cup B_2(c)$.  Similarly to the
argument used for $S_1$ in the first case, we see that at most 3 of
the 4 vertices in this set can witness pairs containing $c$.

Now considering the lone vertex $u_3\in S_4$, we see that if $u_3$
witnesses a pair containing $c$, then there is a path to some $c'\in
C$ and that path must run through either $u_{3,2}$ or $u_{3,1}$.
Thus, there can be no vertices witnessing pairs in $S_1$ or $S_2$
respectively. This gives a max of one vertex forming a pair in
$S_0$, 2 vertices witnessing pairs in $S_1\cup S_2$, 3 vertices
witnessing pairs in $S_3$ and one vertex witnessing a pair in $S_4$
and so $p(c)\le 7$.

If $u_3$ does not witness a pair containing $c$, then there is a max
of one vertex witnessing a pair in $S_0$, 2 vertices witnessing
pairs in each of $S_1$ and $S_2$, and 3 vertices witnessing pairs in
$S_3$ and so $p(c)\le 8$.
\end{proof}

\begin{cor}
  The minimum density of a 3-identifying code of the
hex grid is at least 3/25.
\end{cor}

\section{A Lower Bound for $(r,\leq 2)$-Identifying Codes}\label{section:r2hexcodes}

\begin{lemma}
  For an $(r,\leq2)$-identifying code of the Hexagonal grid, there must be 5 codewords in the region $B_{r+1}(v)\setminus
  B_{r-1}(v)$ for any vertex $v$.
\end{lemma}

\begin{proof}
  Let $S=B_{r+1}(v)\setminus B_{r-1}(v)$ and let $S_x=S\cap B_r(x)$.
   Let $u_1,u_2,$ and $u_3$ be the neighbors of $v$.  We see that
   $B_r(v)\cap B_r(u_1)\cap B_r(u_2)\cap B_r(u_3) = B_{r-1}(v)$ and
   likewise, $B_r(v)\cup B_r(u_1)\cup B_r(u_2)\cup B_r(u_3) =
   B_{r+1}(v)$.  Hence, $I_r(S_1)\triangle I_r(S_2)\subset S$ for
   $S_i\subset\{v,u_1,u_2,u_3\}$, $|S_i|\le 2$.

  We will show
  that 4 codewords is not enough to distinguish between all $S\subset\{v,u_1,u_2,u_3\}$, $|S|\le
  2$.

  Suppose that $\{c_1,c_2,c_3,c_4\}\cap S = C\cap S$.  Let
  $$K(X)=\left(\bigcup_{x\in X} S_x\right)\cap C .$$   We write $K(\{x_1,\ldots,x_n\})=K(x_1,\ldots, x_n)$ for short.
  It is clear that our code can only be valid if $K(X)\neq K(Y)$ for $X,Y\subset\{v,u_1,u_2,u_3\}$ with $X\neq
  Y$ and $|X|,|Y|\le 2$.

  \textbf{Fact 1:} For  $x,y\in\{v,u_1,u_2,u_3\}$ with $x\neq y$, if $K(x)\subset K(y)$,
  then $C$ is not an $(r,\le 2)$-identifying code.
  \begin{proof}
     $K(y)=K(x,y)$ and so $C$ is not an $(r,\le 2)$-identifying code.
  \end{proof}
  \textbf{Fact 2:}
     If $c\in K(v)$ then $c\in K(u_i)$ for some $i$.
  \begin{proof}
     The fact is immediate since $B_r(v)\subset \bigcup_{i=1}^3 B_r(u_i)$.
  \end{proof}

  \textbf{Case 1:} There are three or more code words in $S_{u_i}$ for some $i$.

  Without loss of generality, assume that $\{c_1,c_2,c_3\}\subset
  S_{u_1}$.  If $c_4\not\in S_{u_2}$, then
  $K(u_2)\subset K(u_1)$ and so $C$ is not an $(r,\le 2)$-identifying code.
  Hence, $c_4\in S_{u_2}$ and by a symmetric argument, $c_4\in
  S_{u_3}$.  But then $K(u_1,u_2)=K(u_1,u_3)=\{c_1,c_2,c_3,
  c_4\}$.  Hence, $|K(u_i)|\le 2$ for each
  $i$.

  \textbf{Case 2:} There are two codewords in $S_{u_i}\cap S_{u_j}$ for
  $i\neq j$.

  Assume again without loss of generality that $S_{u_1}\cap
  S_{u_2}=\{c_1,c_2\}$.  From case 1, there can be no other
  codewords in $S_{u_1}$ or $S_{u_2}$ and so $K(u_1)=K(u_2)=\{c_1,c_2\}$.
  Hence $|S_{u_i}\cap S_{u_j}|\le 1$ for $i\neq j$.

  \textbf{Case 3:} $S_{u_i}\cap S_{u_j}=\emptyset$ for $i\neq j$.

  From the above cases, we must have 2 code words in 1 component and
  1 in the other two.  Hence, we may assume $K(u_1)=\{c_1,c_2\}$,
  $K(u_2)=\{c_3\}$, and $K(u_3)=\{c_4\}$.

  If $c_3\in K(v)$, then $K(u_2)\subset K(v)$ and so $c_3\not \in K(v)$.
  Likewise, $c_4\not \in K(v)$.  But then  $K(v)\subset \{c_1,c_2\}=K(u_1,u_2)$.

  \textbf{Case 4:} $|S_{u_i}\cap S_{u_j}|=1$ for some $i\neq j$.

  Assume without loss of generality that $c_1\in K(u_1)\cap K(u_2)$.  If
  $K(u_1)=\{c_1\}$ then $K(u_1)\subset K(u_2)$ and so there must be at least one more code
  word in $K(u_1)$.  But from case 1, $|K(u_1)|\le 2$ and so assume
  that $K(u_1)=\{c_1,c_2\}$ and likewise, we may assume that
  $K(u_2)=\{c_1,c_3\}$.  Since $S_{u_1}\cap S_{u_2}\cap S_{u_3}=\emptyset$ we must have $c_1\not\in S_{u_3}$ and since
  $S_{u_1}\cup S_{u_2}\cup S_{u_3}=S$, we must have $c_4\in S_{u_3}$.

  If $K(u_3)=\{c_4\}$, then $c_4\not\in K(v)$ as in case 3.  If $c_2\not\in K(v)$ then
  $K(v)\subset \{c_1,c_3\}=K(u_2)$, hence $c_2\in K(v)$.  But then
  $K(v,u_2)=K(u_1,u_2)=\{c_1,c_2,c_3\}$.

  Hence, we must have one of $c_2$ and $c_3$ in $S_{u_3}$, but we
  cannot have both or else we are in case 1.  So without loss of
  generality, assume $K(u_3)=\{c_2,c_4\}$.  If $|K(v)|=1$ then $K(v)\subset
  K(u_i)$ for some $i$.  Hence, $|K(v)|>1$.  We already have
  \begin{eqnarray*}
    K(u_1) &=& \{c_1,c_2\} \\
    K(u_2) &=& \{c_1,c_3\} \\
    K(u_3) &=& \{c_2,c_4\} \\
    K(u_1,u_2) &=& \{c_1,c_2,c_3\} \\
    K(u_1,u_3) &=& \{c_1,c_2,c_4\} \\
    K(u_2,u_3) &=& \{c_1,c_3,c_3, c_4\} \\
  \end{eqnarray*}
  which leaves only 5 possible combinations for $K(v)$.  We
  enumerate them here and show that each case causes a
  contradiction.
  $$\begin{array}{c|c}
    K(v) & \text{Contradiction} \\
    \hline
    \{c_1,c_4\} & K(u_1,u_3)=K(u_1,v)=\{c_1,c_2,c_4\} \\
    \{c_2,c_3\} & K(u_1,u_2)=K(u_1,v)=\{c_1,c_2,c_3\} \\
    \{c_3,c_4\} & K(u_2,u_3)=K(u_1,v)=\{c_1,c_2,c_3, c_4\} \\
    \{c_1,c_3,c_4\} & K(u_2,u_3)=K(u_1,v)=\{c_1,c_2,c_3, c_4\} \\
    \{c_2,c_3,c_4\} & K(u_2,u_3)=K(u_1,v)=\{c_1,c_2,c_3, c_4\}
  \end{array}$$
\end{proof}

\begin{thm}
   The minimum density of an $(r,\leq 2)$-identifying code in
   the hex grid is at least $$\frac{5}{6r+3} .$$
\end{thm}
\begin{proof}
We refer to the methods of Charon, Honkala, Hudry and
Lobstein~\cite{Charon2001} used to prove a minimum density for
$(r,\leq 1)$-identifying codes for various infinite grids.
Since $|B_{r+1}(v)\setminus B_{r-1}(v)| = 6r+3$~\cite{Charon2002}
for any $v\in V(G_H)$ and $B_{r+1}(v)\setminus B_{r-1}(v)$ must
contain at least 5 code words by the previous lemma, any code must
have density at least $\displaystyle\frac{5}{6r+3}$.
\end{proof}

\chapter{VERTEX IDENTIFYING CODES FOR REGULAR GRAPHS}\label{chapter:regulargraphcodes}

\section{Introduction}
We wish to improve and expand upon some of the ideas mentioned in the introduction.  In particular, we want to look at the densities of codes on graphs that are dense with edges.  We first focus on improving the general lower bound for such graphs with the theorems in Section~\ref{section:RGtheorems} and then construct regular (or near regular) graphs that have codes attaining this lower bound in Section~\ref{section:RGconstructions}.

\section{Theorems}\label{section:RGtheorems}
We start by presenting lower bounds or the size of regular graph codes.

\begin{thm}\label{2bound}
  Let $C$ be an identifying code in an $n$ vertex graph $G$ with maximum degree $\Delta$.  Then $$|C|\ge\max\left\{\frac{2n}{\Delta+2},\frac{-(2\Delta+5)+\sqrt{(2\Delta+5)^2+24n}}{2}\right\}$$
\end{thm}

\begin{thm}\label{sbound}
  Let $C$ be an identifying code in an $n$ vertex graph $G$ with maximum degree $\Delta$.  Then $$|C|\ge\max_{1\le s\le \Delta}\left\{\frac{(s+1)n}{\Delta+1}\left[1-\frac{(s+1)^{s-1}n^{s-1}}{(s-1)!(\Delta+1)^s}\right]\right\}$$
\end{thm}

\section{Proofs}

Recall Theorem~\ref{thm:regularGraphCodeLowerBound} which shows that for a $d$-regular graph the
minimum size of a vertex identifying code $$|C|\ge \frac{2n}{d+2}.$$
This is done by a double counting argument in which we count the size of the identifying sets as well as the ball of radius 1 surrounding each codeword.  We note that this bound holds for the code of any graph of maximum degree $\Delta$ as well.
We then generalize this argument to get better
bounds.

{\bf Proof of Theorem \ref{2bound}.}

\begin{proof}
Recall the equation $$\sum_{j=1}^k |B_r(c_j)| = \sum_{v\in V(G)} |I_r(v)|.$$  If we have $k$ codewords, then we can  upper bound the left hand side by $(\Delta +1)k$.

If $n \ge k + {k \choose 2}$, then note that we have at most $k$ identifying sets of size $k$, $\binom k2$ identifying sets of size 2 and so the rest of our identifying sets have size at least 3.  Thus, we can lower bound the right hand size by $$k + 2{k \choose 2} +
3\left[n-k-{k \choose 2}\right]$$ and so
\begin{equation}\label{ineq1}(\Delta+1)k\ge k + 2{k \choose
2} + 3\left[n-k-{k \choose 2}\right] .\end{equation}

Solving for $k$ gives $$k\ge \frac{-(2\Delta+5)+\sqrt{(2\Delta+5)^2 +
24n}}{2}$$

However, suppose that $n<k + {k \choose 2}$.  From Karpovsky, we
already know that any code satisfies $(\Delta+1)k\ge 2n-k$.  But
then
\begin{eqnarray*}
  (\Delta+1)k&\ge& 2n-k \\
  &\ge& 2n-k + \left[n-k-{k\choose 2}\right] \\
  &=& k + 2{k \choose 2} + 3\left[n-k-{k \choose 2}\right]
\end{eqnarray*}

and so $k$ satisfies equation (\ref{ineq1}) anyway and

$$|C|\ge\max\left\{\frac{2n}{\Delta+2},\frac{-(2\Delta+5)+\sqrt{(2\Delta+5)^2+24n}}{2}\right\} .$$
\end{proof}

Next, define the function $$f(n,k,s)=\sum_{i=1}^s i\binom{k}{i} + (s+1)\left[n-\sum_{i=1}^s \binom{k}{i}\right] .$$  We will use this function in the proof of theorem \ref{sbound} so we first present a quick lemma.

\begin{lemma}\label{fLemma}
    $f(n,k,s) = f(n,k,s-1) + n - \sum_{i=1}^s \binom{k}{i}$
\end{lemma}
\begin{proof} The result follows by simple algebraic manipulation.
  \begin{eqnarray*}
    f(n,k,s) &=& \sum_{i=1}^s i\binom{k}{i} + (s+1)\left[n-\sum_{i=1}^s \binom{k}{i}\right]\\
    &=& \sum_{i=1}^{s-1} i\binom{k}{i} + s\left[n-\sum_{i=1}^{s-1} \binom{k}{i}\right] + n - \sum_{i=1}^s \binom{k}{i}\\
    &=& f(n,k,s-1) + n - \sum_{i=1}^s \binom{k}{i}
  \end{eqnarray*}
\end{proof}

\begin{lemma}\label{kLemma}
    Suppose a graph $G$ with maximum degree $\Delta$ has a code of size $k$.  Then $k$ satisfies \begin{equation}(\Delta+1)k\ge f(n,k,s)\label{ineq2}\end{equation} for $1\le s\le \Delta$.
\end{lemma}
\begin{proof}
  Here we generalize the argument from Theorem \ref{2bound}.  For any $s\ge 1$, if $n\ge \sum_{i=1}^s \binom{k}{i}$, then $$(\Delta+1)k\ge \sum_{i=1}^s i\binom{k}{i} + (s+1)\left[n-\sum_{i=1}^s \binom{k}{i}\right] = f(n,k,s) .$$
  Let $s^*$ be the largest $s$ for which $n\ge \sum_{i=1}^s \binom{k}{i}$ holds.  Then for any $s>s^*$ we have $n - \sum_{i=1}^s \binom{k}{i}<0$.
  So,
  \begin{eqnarray*}
    (\Delta+1)k &\ge& f(n,k,s^*) \\
    &\ge& f(n,k,s^*) + n - \sum_{i=1}^{s^*+1} \binom{k}{i}\\
    &=& f(n,k,s^*+1) .
  \end{eqnarray*}
  Then by induction we get $(\Delta+1)k \ge f(n,k,s^*+r)$ for any $r\ge 0$.
\end{proof}
{\bf Proof of Theorem \ref{sbound}.}

\begin{proof}
First we note that
\begin{eqnarray*}
  f(n,k,s) &=& \sum_{i=1}^s i\binom{k}{i} + (s+1)\left[n-\sum_{i=1}^s \binom{k}{i}\right]\\
  &=& (s+1)n - \sum_{i=1}^s i\binom{k}{i} + \sum_{i=1}^s (s+1)\binom{k}{i}\\
  &=& (s+1)n + \sum_{i=1}^s (s+1-i)\binom{k}{i} .\\
\end{eqnarray*}
From Lemma \ref{kLemma}, we know that a code of size $k$ satisfies
$(\Delta +1)k \ge \sum_{i=1}^k (d(c_i)+1) \ge f(n,k,s)$ so this
gives
\begin{eqnarray*}
  (\Delta+1)k &\ge& (s+1)n + \sum_{i=1}^s (s+1-i)\binom{k}{i}\\
  &\ge& (s+1)n -s\binom{k}{s} \\
  &\ge& (s+1)n -\frac{k^s}{(s+1)!}
\end{eqnarray*}
and so $$n(s+1)(s+1)! \le k^s + (s-1)!k(\Delta+1) .$$

Next, let $$k=(1-\epsilon)\frac{(s+1)n}{\Delta+1} \le \frac{(s+1)n}{\Delta+1}$$

Then we can plug this into the expression above to get
\begin{eqnarray*}
  n(s+1)(s+1)! &\le& k^s + (s-1)!k(\Delta+1) \\
  n(s+1)(s+1)! &\le& \left(\frac{(s+1)n}{\Delta+1}\right)^s + (s-1)!(1-\epsilon)(s+1)n \\
  \epsilon(s+1)(s-1)!n &\le& \left(\frac{(s+1)n}{\Delta+1}\right)^s \\
  \epsilon &\le & \frac{(s+1)^{s-1}n^{s-1}}{(s-1)!(\Delta+1)^s}
\end{eqnarray*}
and hence
\begin{eqnarray*}
1-\epsilon &\ge & 1-\frac{(s+1)^{s-1}n^{s-1}}{(s-1)!(\Delta+1)^s}\\
(1-\epsilon)\frac{(s+1)n}{\Delta+1} &\ge&  \frac{(s+1)n}{\Delta+1}\left[1-\frac{(s+1)^{s-1}n^{s-1}}{(s-1)!(\Delta+1)^s}\right]\\
k &\ge&
\frac{(s+1)n}{\Delta+1}\left[1-\frac{(s+1)^{s-1}n^{s-1}}{(s-1)!(\Delta+1)^s}\right]
.
\end{eqnarray*}
\end{proof}

\section{Construction of Graphs Satisfying the Bounds}\label{section:RGconstructions}
Let $k$ be given and $s+1\le k$.  We then construct a graph of order
$n=\sum_{i=1}^{s+1} \binom{k}{i}$ and max degree
$\Delta=\sum_{i=2}^{s+1} \binom{k-1}{i-1}$ satisfying
 (\ref{ineq2}).

Let $C$ be a set of $k$ vertices.  For $i=2,\ldots, s+1$ let $C_i$
be a set of $\binom{k}{i}$ vertices. Then there is a bijection
$\varphi:C_i\rightarrow \binom{C}{i}$.  So for each vertex $v\in
C_i$, we create an edge between $v$ and each vertex in $\varphi(v)$.

The degree of a vertex in $C_i$, $i\ge 2$ is exactly $i$.  Then, for
each vertex in $C$, it is adjacent to 1 vertex in $C_i$ for each
subset of size $i-1$ of the remaining $k-1$ vertices and so the
degree of a vertex in $C$ is exactly $\Delta=\sum_{i=2}^{s+1}
\binom{k-1}{i-1}$.

Now we claim that $C$ is a code.  For each vertex in $c\in C$, it is
not adjacent to any other vertex in $C$ and so $N[c]\cap C = \{c\}$
and for each vertex in $C_i$, it is adjacent to a unique subset of
size $i$.  Hence, $N[v]\cap C$ is unique for every vertex in our
graph.  Note that in this case we have exact equality in equation
(\ref{ineq2}).

Of particular interest are graphs that are regular or almost
regular.  Note that if you create a graph in the fashion described
above, that $C$ will still be a code if you add edges between any
two vertices not in $C$ since $N[v]\cap C$ remains unchanged for
each vertex when these edges are added.

\begin{thm}
  Let $k=3t+2$ and $t\equiv 1\text{ or } 2\pmod 4$.  There exists a $d$ regular graph on $n$ vertices admitting a 1-identifying code of size $$\frac{-(2d+5)+\sqrt{(2d+5)^2+24n}}{2}$$ where $d=\binom k2$ and $n = k +\binom k2 + \binom k3$.
\end{thm}
\begin{proof}

Create a regular graph in the method described above with $s=2$ and $k=3t+2$.  Then we have $d=\binom{k-1}{1}+\binom{k-1}{2}=\binom{k}{2}$. Next, the degree of each vertex in $C_2$ is currently 2.  So if we connect each pair of vertices with an edge they will have degree $\binom{k}{2}+1$.  However, under our conditions for $k$, $\binom{k}{2}$ will be even so we then remove a matching and each vertex with then have degree $d$.

Finally, there are $\binom{k}{3}=\frac{k-2}{3}\binom{k}{2}$ vertices in $C_3$.  Since $\frac{k-2}{3}=t$ is an integer, we can partition it into $t$ sets of $\binom{k}{2}$ vertices and turn each of those into a complete graph and then remove a Hamiltonian cycle from each of them so that each vertex in $C_3$ now has degree $d$ as well.
\end{proof}

We next wish to show the existence of arbitrarily large graphs that attain equality in equation (\ref{ineq2}) for higher values of $s$.
\begin{thm}\label{thm:regexistthm}
  For all $k,s\ge 2$, there exists a graph on $n$ vertices which is either $d$-regular or that has $n-1$ vertices of degree $d$ and 1 vertex of degree $d-1$ that admits a code of size $k$ where $$d=\sum_{i=2}^{s+1} \binom{k-1}{i-1}$$ and $$n=\sum_{i=1}^{s+1} \binom{k}{i}.$$
\end{thm}

Before proving the theorem, we need a technical lemma.

\begin{lemma}\label{dsLemma}
    Let $k$ be given and $s\le k-1$ and define $m=\sum_{i=2}^{s+1}
    \binom{k}{i}$ and $d=\sum_{i=2}^{s+1} \binom{k-1}{i-1}$.  Then there
    exists a graph of degree sequence
    $$(\underbrace{d-2,\ldots,d-2}_{\binom{k}{2}\text{ times}},\underbrace{d-3,\ldots,d-3}_{\binom{k}{3}\text{ times}},\ldots,\underbrace{d-(s+1),\ldots,d-(s+1)}_{\binom{k}{s+1}\text{ times}})$$
    if the sums of the degrees is even.  If the sum of the degrees in the above sequence is
    odd, there exists a graph of the degree sequence above, replacing the last $d-2$ with $d-3$.
\end{lemma}
\begin{proof}By the Erd\H{o}s-Gallai Theorem\cite{Erdos1960}, we know that a graph with this degree sequence exists if
the sum of the degrees in the sequence is even and
$$\sum_{i=1}^r d_i \le r(r-1) + \sum_{i=r+1}^m \min(d_i,r)$$ for all $1\le r\le m$.
From the statement of the lemma, it is clear that our sequence is even.

Next, assuming the convention that $\binom{n}{k}=0$ if $k>n$, we
note that
$$m=\sum_{i=2}^{s+1}\binom{k}{i}=
       \sum_{i=2}^{s+1}\left[\binom{k-1}{i-1}+\binom{k-1}{i}\right]
       =d+\sum_{i=2}^{s+1}\binom{k-1}{i}$$
also, we have $$m=2d+\binom{k-1}{s+1}-\binom{k-1}{1}. $$

Dealing with the left hand side of the inequality, note that $\sum_{i=1}^r d_i \le r(d-2)$.
For the right hand side, we break into 3 cases.

Case 1: $r\le d-s-1$

Then

\begin{eqnarray*}
  r(r-1) + \sum_{i=r+1}^m \min(d_i,r) &=& r(r-1) + \sum_{i=r+1}^m r \\
  &=& r(r-1) + (m-r)r \\
  &=& r(m-1) \\
  &\ge& r(d-2)
\end{eqnarray*}

Where the last step holds since $d<m$.

Case 2: $r\ge d-2$

Then

\begin{eqnarray*}
  r(r-1) + \sum_{i=r+1}^m \min(d_i,r)   &\ge& r(r-1) \\
  &\ge& r(d-2)
\end{eqnarray*}

Case 3: $d-s-1<r<d-2$

For the left hand side, we now have $$\sum_{i=1}^r d_i \le r(d-2)\le
(d-2)^2$$ and the right hand side becomes
\begin{eqnarray*}
  r(r-1) + \sum_{i=r+1}^m \min(d_i,r) &=& r(r-1) + \sum_{i=r+1}^m d-s-2 \\
  &\ge& r(d-s-2) + (m-r)(d-s-2) \\
  &=& m(d-s-2) \\
  &\ge& \left(2d+\binom{k-1}{s+1}\right)(d-s-2)
\end{eqnarray*}

Thus, we turn our attention to showing that $$(d-2)^2 \le
\left(2d+\binom{k-1}{s+1}\right)(d-s-2) .$$  After algebraic
manipulation, we rewrite this as
$$d(d-2s)-4+\binom{k-1}{s+1}(d-s-2)\ge 0 .$$

The last term is clearly positive and we note that $d(d-2s)-4$ is
positive if $d\ge 2s+1$ and $d\ge 4$.  To show the first inequality,
first note that $s\ge 2$.  Then $$d\ge \binom{s}{1} + \binom{s}{2}=
s + (s^2+s)/2 .$$  Then through simple algebraic manipulation we see
that $d\ge s+(s^2+s)/2\ge 2s+1$ as long as $s\ge 2$. The second
inequality then immediately follows from the first since $d\ge 2s+1
\ge 4$.
\end{proof}

\begin{proofcite}{Theorem~\ref{thm:regexistthm}}
 Again, create a graph $G$ on $n$ vertices with a code of size $k$ as described earlier.  Each vertex in $C_i$ currently has degree $i$.  Thus, we wish to create
a graph on $C_2\cup C_3 \cup \cdots\cup C_{s+1}$ such that each vertex in $C_i$ has degree $d-i$.  This is impossible if the sums of the degrees is odd, but in the following lemma, we will show that it is possible to find a graph that has the proper degree sequence or with at most one vertex being one degree
$d-i-1$.

Then, by Lemma~\ref{dsLemma} we may create a graph on $C_2\cup C_3 \cup \cdots\cup C_{s+1}$ such that each vertex in $C_i$ has degree $d-i$ or all but one vertex does and the remaining vertex has degree $d-i-1$.  Superimposing this on our graph $G$ gives the desired result.
\end{proofcite} 


\unappendixtitle
\interlinepenalty=300

\phantomsection
\bibliographystyle{plain}
\addcontentsline{toc}{chapter}{BIBLIOGRAPHY}
\bibliography{bibtex}

\end{document}